\documentclass[11pt,leqno]{article}
\usepackage{bbm}
\usepackage{mathrsfs}
\usepackage{amsfonts}
\usepackage{tipa}
\usepackage{amssymb}
\usepackage{amsmath}
\usepackage{amsthm}
\usepackage{amstext}
\usepackage{amsopn}
\usepackage{texdraw}
\usepackage{graphicx}
\usepackage{yhmath}
\usepackage{indentfirst}
\usepackage{mathtools}
\usepackage{enumitem}
\usepackage{pgfplots}
\usepackage{chngpage}
\usepackage{tikz}
\usepackage{epsfig}
\usepackage{verbatim}
\usepackage{subfig}
\usepackage{cases}
\usepackage{appendix}
\usepackage{hyperref}
\hypersetup{hypertex=true,
	colorlinks=true,
	linkcolor=blue,
	anchorcolor=blue,
	citecolor=blue,
	urlcolor=blue}
\usetikzlibrary{arrows}
\usetikzlibrary{patterns}
\newtheorem{theorem}{Theorem}[section]
\newtheorem{lemma}[theorem]{Lemma}

\newtheorem{proposition}[theorem]{Proposition}

\theoremstyle{definition}

\theoremstyle{remark}

\numberwithin{equation}{section}

\begin{document}

\title{\text\bf Spatiotemporal pattern formations in a two-layer coupled reaction-diffusion Lengyel-Epstein system}
\author{Qidong Wu, Fengqi Yi\footnote{Corresponding Author (F. Yi). Email: yi@dlut.edu.cn}\footnote{The authors were partially by National Key R\;$\&$\;D Program of China (2023YFA1009200), the National Natural Science Foundation of China (11971088, 12371160, 12271144) and the Fundamental Research Funds for the Central Universities (DUT18RC(3)070, DUT23LAB304). } \\
\\
{\small School of Mathematical Sciences, Dalian University of Technology, Dalian,\hfill{\ }}\\
{\small Liaoning Province, 116024, China.\hfill{\ }}
  \\
}

\date{}
\maketitle

\newcommand{\Om}{\Omega}
\newcommand{\cOm}{\overline{\Omega}}
\newcommand{\ep}{\epsilon}
\newcommand{\la}{\lambda}
\newcommand{\ola}{\overline{\la}}
\newcommand{\ula}{\underline{\la}}
\newcommand{\lam}{\lambda^M}
\newcommand{\omu}{\overline{\mu}}
\newcommand{\umu}{\underline{\mu}}
\newcommand{\mum}{\mu^M}
\newcommand{\R}{{\mathbf R}}
\newcommand{\C}{{\mathbf C}}
\newcommand{\Rp}{{\mathbf R}^+}
\newcommand{\rn}{{\mathbf R}^n}
\newcommand{\Po}{\partial \Omega}
\newcommand{\po}{\partial \Omega}
\newcommand{\pa}{\partial}
\newcommand{\al}{\alpha}
\newcommand{\io}{\int_{\Om}}
\newcommand{\rnp}{{\bf R}$^n_+$}
\newcommand{\ds}{\displaystyle}
\newcommand{\mb}{\mbox}
\newcommand{\lbl}{\label}
\newcommand{\ulb}{\overline{u}_{\la}}
\newcommand{\uld}{\overline{u}_{\la_d}}
\newcommand{\noi}{\noindent}
\newcommand\bu{{\mathbf u}}
\newcommand\bL{{\mathbf L}}
\newcommand\bB{{\mathbf B}}
\newcommand\bc{{\mathbf c}}
\newcommand\bA{{\mathbf A}}
\newcommand\bH{{\mathbf H}}
\newcommand\bI{{\mathbf I}}
\newcommand\bQ{{\mathbf Q}}
\newcommand\bG{{\mathbf G}}
\newcommand\bX{{\mathbf X}}
\newcommand\bR{{\mathbb{R}}}
\newcommand\bU{{\mathbf U}}
\newcommand\bF{{\mathbf F}}
\newcommand\bV{{\mathbf V}}
\newcommand\bv{{\mathbf v}}
\newcommand\bY{{\mathbf Y}}
\newcommand\bZ{{\mathbf Z}}
\newcommand\bC{{\mathbf C}}
\newcommand\bP{{\mathbf P}}

\begin{abstract}
Spatiotemporal pattern formations in two-layer coupled reaction-diffusion Lengyel-Epstein system with distributed delayed couplings are investigated. Firstly, for the original decoupled system, it is proved that when the intra-reactor diffusion rate $\ep$ of the inhibitor is sufficiently small and the intra-reactor diffusion rate $d$ of the inhibitor is large enough, then the subsystem can exhibit non-constant positive steady state $(\widetilde{u}(x;\ep),\widetilde{v}(x;\epsilon))$ with large amplitude, and that as the parameter $\tau$ varies, the stability of $(\widetilde{u}(x;\ep),\widetilde{v}(x;\epsilon))$ changes, leading to the emergence of periodic solutions via Hopf bifurcation. Secondly, for the two-layer coupled system, the stability of the symmetric steady state $(\widetilde{u}(x;\ep),\widetilde{v}(x;\epsilon),\widetilde{u}(x;\ep),\widetilde{v}(x;\epsilon))$ is studied by treating $k_1,k_2$ (the inter-reactor diffusion rates) and $\al$ (the delay parameter) as the main parameters. In case of non-delayed couplings, the first quadrant of the $(k_1, k_2)$ parameter space can be divided into two regions: one is stable region, the other one is unstable region, and the two regions have the common boundary, which is the primary Turing bifurcation curve. In case of delayed couplings, it is shown that the first quadrant of the $(k_1, k_2)$ parameter space can be re-divided into three regions: the first one is unstable region, the second one is stable region, while the third one is the potential \lq\lq bifurcation\rq\rq region, where Hopf bifurcation may occur for suitable $\al$. Our analysis is mainly based on the singular perturbation techniques and the implicit function theorem, and the results show some different phenomena from those of the original decoupled system in one reactor.\\
		
\noindent{\bf Keywords}:  Coupled Lengyel-Epstein system; Distributed-delay coupling; Non-constant layered solution; Turing bifurcation and Hopf bifurcation; Singular limit eigenvalue problem.

\noindent{\bf MSC}:  35B32, 35P20, 35Q92, 92E20.
\end{abstract}

\section{Introduction}
One of the crucial issues in developmental biology is to concern the spatial pattern formation in the early embryo. In a remarkable work \cite{Tu}, Turing predicted that a system of reacting and diffusing chemicals could generate spatial patterns in chemical concentrations from initial near homogeneity. He suggested that this process might lead to a variety of morphogenetic phenomena. The possibility of such diffusion-driven instability, now usually termed Turing instability, has been widely investigated in the chemical and biological literatures. 

In 1990, Kepper and the collaborators \cite{CDBD} discovered the formation of the stationary patterns in the chlorite–iodide–malonic acid-starch (CIMA) reaction in an open unstirred gel reactor, which is the first unambiguous experimental evidence of the existence of the symmetry breaking reaction-diffusion structures predicted by Turing in 1952. In this CIMA chemical reaction, there are five reactants, which makes the mathematical description very complicated. However, numerical calculations showed that during oscillations in the CIMA reaction, only the concentrations of $I^-$ (the activator iodide) and $\text{ClO}_2^-$ (the inhibitor chlorite) vary significantly, and they react to produce the immobile complexing reactant malonic acid (MA). 

The partial differential equations reflecting the evolutionary behaviors of the concentrations $u:=[I^-]$, $v:=[\text{ClO}_2^-]$ and $w:=[\text{MA}]$ are given by (see \cite{LE1, LE2, LE3})
\begin{equation}\label{LE-00}
	\begin{cases}
		\begin{split}
			\dfrac{\partial u}{\partial t}=&d_1\Delta u+a-u-\dfrac{4uv}{1+u^2}-k_+u+k_-w,\\
			\dfrac{\partial v}{\partial t}=&d_2\Delta v+b\big(u-\dfrac{uv}{1+u^2}\big),\\
			\dfrac{\partial w}{\partial t}=&k_+u-k_-w,
		\end{split}
	\end{cases}
\end{equation}
where $d_1>0$ and $d_2>0$ are the diffusion rates of the activator $I^-$ and the inhibitor $\text{ClO}_2^-$, respectively. Since malonic acid is assumed to be immobile, there is no diffusion in the third equation of \eqref{LE-00}. The constants $a, b, k_+$ and $ k_-$ are all positive. Moreover, the homogeneous Neumann boundary conditions are imposed on the boundary.

If the formation and the dissociation of the complex are rapid, then $w=ku$ for some $k>0$ (see \cite{LE2} for details). Thus, by setting $\sigma=1+k$, we can reduce system \eqref{LE-00} to the following so-called reaction-diffusion Lengyel-Epstein system (\cite{LE2})
\begin{equation}\label{LE-0}
	\begin{cases}
		\begin{split}
			\dfrac{\partial u}{\partial t}&=\dfrac{1}{\sigma}\big(d_1\Delta u+a-u-\dfrac{4uv}{1+u^2}\big),\;\\
			\dfrac{\partial v}{\partial t}&=d_2\Delta v+b\big(u-\dfrac{uv}{1+u^2}\big).
		\end{split}
	\end{cases}
\end{equation}

Setting $t'=\sigma t$ and $x'=x/\sqrt{d_1}$, one can reduce system \eqref{LE-0} to
\begin{equation}\label{LE-0-b}
	\begin{cases}
		\begin{split}
			\dfrac{\partial u}{\partial t}&=\Delta u+a-u-\dfrac{4uv}{1+u^2},\;\\
			\dfrac{\partial v}{\partial t}&=\sigma c\Delta v+\sigma b\big(u-\dfrac{uv}{1+u^2}\big),
		\end{split}
	\end{cases}
\end{equation}
where we drop the $'$ from $t'$ and $x'$ for convenience, and $c$ is the effective diffusion ratio $d_2/d_1$.

The dynamics of the reaction-diffusion Lengyel-Epstein system mainly concentrate on the dynamics of the rescaled system \eqref{LE-0-b}. For example, in \cite{NT}, Ni and Tang studied Turing instability of the unique constant steady state of system \eqref{LE-0-b} and further considered the  existence/nonexistence of the non-constant steady states by using topological degree theory. In \cite{JNT}, Jang, Ni and Tang studied the global bifurcation structure of the set of the non-constant steady states for system \eqref{LE-0-b}. In \cite{JSWY}, Jin, Shi, Wei and Yi studied the interactions between Hopf bifurcations and steady state bifurcations of system \eqref{LE-0-b}. In \cite{Yi}, Yi studied Turing instability of the spatially homogeneous periodic solutions of system \eqref{LE-0-b} by using the regular perturbation methods. In \cite{ZYW}, Zhao, Yu and Wang introduced discrete time delay into the system and studied Turing bifurcation, Hopf bifurcation, Turing-Turing bifurcation, Turing-Hopf bifurcation and Bogdanov-Takens bifurcation. These results have the common features: the existence of the non-constant steady states has been established, however, the stability of these solutions was unsettled. As one of the purposes of this paper, we shall give an affirmative answer for the stability issue. 

On the other hand, since bilayer membranes or multilayer tissues are often found in biological, physical and chemical systems, spatiotemporal pattern formation in coupled system that consists of two or more coupled chemical reacting subsystems is of great significance (\cite{KS, LE2, LWO, PPR}). Generally speaking, the width of the membrane cannot be negligible compared with the size of the system, it is then necessary to include the time delay in the coupling terms (see \cite{AEKV, Bar, JL, PL, TK, YDZE, YE}). 

Motivated by \cite{JL, LE2, LWO, TK, TMN}, we make the following assumptions:
\begin{enumerate}[leftmargin=4.5em]
	\item[$($H1$)$.] There are two identical well-stirred reactors (layers), where the chemical reactions and initial or feed concentrations of the reactants are exactly the same. They are connected by a semipermeable membrane through which the chemicals can diffuse (termed \lq\lq inter-reactor\rq\rq diffusion) by Fick's first law (see \cite{LE2, TK}): the mass flux $J_X$ of a species (denoted by $X$) from reactor 1 to reactor 2 is directly proportional to the concentration difference $\Delta X=X_2-X_1$, that is, $J_X=-\rho_X\Delta X$, where $\rho_X$ is the constant of proportionality. The amount of $X$ leaving reactor 1 through the membrane of area $A$ in time $\Delta t$ is $J_XA\Delta t$. Then, the rate of change of concentration $X$ in reactor 1 (resp., reactor 2) of volume $V$ is $k_X(X_2-X_1)$ (resp., $-k_X(X_2-X_1)$), where $k_X:=\rho_XA/V$ is the \lq\lq inter-reactor effective diffusion constant\rq\rq.
	\item[$($H2$)$.] In absence of coupling, in each of the reactors, there also have \lq\lq intra-reactor\rq\rq diffusions, which are the random diffusions in the spatial domain $\Omega$ of one reactor and are described by the Laplacian terms in each equation. The dynamics of each of the two reactors (designated 1 and 2) is described by the following reaction-diffusion equations in the absence of coupling:
	\begin{equation*}
		\text{Reactor}\;i\;(i=1,2):
		\begin{cases}
			\begin{split}
				\dfrac{\partial u_i}{\partial t}&=\dfrac{1}{\sigma}\big(d_1\Delta u_i+a-u_i-\dfrac{4u_iv_i}{1+u_i^2}\big),~&(x,t)\in\Omega\times(0,+\infty),\\
				\dfrac{\partial v_i}{\partial t}&=d_2\Delta v_i+b\big(u_i-\dfrac{u_iv_i}{1+u_i^2}\big), ~&(x,t)\in\Omega\times(0,+\infty),
			\end{split}
		\end{cases}
	\end{equation*}
	subject to homogeneous Neumann boundary conditions.
	\item[$($H3$)$.] There has the distributed time delay occurring when the species cross the membrane (see \cite{JL, LYGYW, RF, ZLLX} for some practical backgrounds). For our technical reasons, we assume that distributed time delay only occurs in the couplings of activators. Although it might be more interesting if both the couplings of the activators and the inhibitors involve distributed delays (see \cite{JL}), mathematical analysis tends to be much more difficult.
\end{enumerate}

Based on $($H1$)$--$($H3$)$, if we couple system \eqref{LE-0} with an identical system, then the resulting four-component coupled system can be described by
\begin{equation}\label{C-LE-00}
	\begin{cases}
		\begin{split}
			\footnotesize\dfrac{\partial u_i}{\partial t}&=\small\frac{1}{\sigma}\bigg(d_1\Delta u_i+a-u_i-\frac{4u_iv_i}{1+u_i^2}\bigg)+k_1\bigg({\small\int_{-\infty}^{t} }F(t-s)u_j(x,s)ds-u_i\bigg), ~&\text{in}\;\Omega_\infty,\\
			\dfrac{\partial v_i}{\partial t}&=d_2\Delta v_i+b\bigg(u_i-\dfrac{u_iv_i}{1+u_i^2}\bigg)+k_2(v_j-v_i),~&\text{in}\;\Omega_\infty,\\
			\dfrac{\partial u_i}{\partial n}&=\dfrac{\partial v_i}{\partial n}=0,~&\text{on}\;\Po_\infty,
		\end{split}
	\end{cases}
\end{equation}
where $i,j\in\{1,2\}$ and $i\neq j$; $\Omega_\infty:=\Omega\times(0,\infty)$ and $\Po_\infty:=\Po\times(0,\infty)$, in which $\Omega\subset\R^N$ $(N\geq1)$ is the open bounded domain with a sufficient smooth boundary; $n$ is the outer unit normal derivative across the boundary $\Po$; $u_i$ and $v_i$ ($i=1,2$) stand for the concentrations of the activators and the inhibitors in the reactor $i$, respectively; $k_1$ and $k_2$ are the inter-reactor diffusion rates of the activators and inhibitors, respectively; the kernel $F(s)$ describes the distribution of delay effect. For our interests, we specify $F(t)$ to be the weak kernel $\alpha e^{-\alpha t}$ with $\alpha>0$ and $t\geq0$. Clearly, $\int_{-\infty}^tF(t-s)ds=1$ (see \cite{RCF, SY3}). 

Motivated by \cite{MTH, TMN}, by making the following change of variables
\begin{equation*}
	\epsilon=\sqrt{d_1/\sigma},\;\tau=b\sqrt{\sigma/d_1},\;d=d_2/b, \;\widehat t=bt, \; \widehat{k}_1=k_1\sqrt{\sigma/d_1},
\end{equation*}
we can reduce system \eqref{C-LE-00} to the following system
\begin{equation}\label{C-LE-D}
	\begin{cases}
		\begin{split}
			\tau\dfrac{\partial u_i}{\partial t}&=\epsilon\Delta u_i+\dfrac{1}{\sigma\epsilon}\bigg(a-u_i-\dfrac{4u_iv_i}{1+u_i^2}\bigg)+k_1\bigg(\int_{-\infty}^{t} F(t-s)u_j(x,s)ds-u_i\bigg),&\text{in}\;\Omega_\infty,\\
			\dfrac{\partial v_i}{\partial t}&=d\Delta v_i+u_i-\dfrac{u_iv_i}{1+u_i^2}+k_2\big(v_j-v_i\big),&\text{in}\;\Omega_\infty,\\
			\dfrac{\partial u_i}{\partial n}&=\dfrac{\partial v_i}{\partial n}=0,&\text{on}\;\Po_\infty,
		\end{split}
	\end{cases}
\end{equation}
where $i,j\in\{1,2\}$ and $i\neq j$; we still use $t$ (resp., $k_1$) to stand for $\widehat t$ (resp., $k_1$) if there is no confusion. In system \eqref{C-LE-D}, the parameter $\tau$ has the significant chemical meanings. Indeed, in the absence of coupling, if $\tau=O(1)$, $u$ species reacts faster than $v$ species does, and $u$ species diffuses slower than $v$ species does; if $\tau=O(1/\epsilon)$, $u$ species and $v$ species react to the same degree, but $u$ species diffuses much slower than $v$ species does; if $\tau=O(\epsilon)$, $u$ species reacts faster than $v$ species does, $u$ species and $v$ species react to the same degree (see \cite{TMN} for more details).

When there is no couplings in system \eqref{C-LE-D} (say, $k_1=k_2=0$), we can obtain the following subsystem (indeed two subsystems, but they are exactly identical)
\begin{equation}\label{Decoupled}
	\begin{cases}
		\begin{split}
			\tau\dfrac{\partial u}{\partial t}&=\epsilon \Delta u+\dfrac{1}{\sigma\epsilon}\big(a-u-\dfrac{4uv}{1+u^2}\big),~&\text{in}\;\Omega_\infty,\\
			\dfrac{\partial v}{\partial t}&=d\Delta v+u-\dfrac{uv}{1+u^2},~&\text{in}\;\Omega_\infty,\\
			\dfrac{\partial u}{\partial n}&=\dfrac{\partial v}{\partial n}=0,~&\text{on}\;\Po_\infty.
		\end{split}
	\end{cases}
\end{equation}

Clearly, if $(U(x),V(x))$ is the steady state of subsystem \eqref{Decoupled}, then by $\int_{-\infty}^tF(t-s)ds=1$, $(U(x),V(x),U(x),V(x))$ is the steady state of system \eqref{C-LE-D}, which is called the symmetric steady state (or reactor-based \lq\lq spatially homogeneous\rq\rq steady state).

In the remaining part of the paper, we shall always concentrate on the dynamics of both the coupled system \eqref{C-LE-D} and its decoupled subsystem \eqref{Decoupled}. For the convenience of our discussions, throughout the paper, we always assume that $\Omega=(0,\ell)$, with $\ell>0$.

Firstly, for the two-component decoupled subsystem \eqref{Decoupled}, it is found that when $\ep$ is sufficiently small and $d$ is large enough, it can exhibit the non-constant positive steady state $(\widetilde{u}(x;\ep),\widetilde{v}(x;\epsilon))$ with large amplitudes (see Theorem \ref{th.1-0}). Moreover, it is proved that as the parameter $\tau$ varies, the stability of $(\widetilde{u}(x;\ep),\widetilde{v}(x;\epsilon))$ changes, leading to Hopf bifurcation. More precisely, there exists a $\tau_c^\ep\in(0,+\infty)$, such that $(\widetilde{u}(x;\ep),\widetilde{v}(x;\epsilon))$ is stable when $\tau\in(\tau_c^\ep, +\infty)$, unstable when $\tau\in(0,\tau_c^\ep)$, while Hopf bifurcation occurs at $\tau=\tau_c^\ep$ (see Theorem \ref{th.1-1}). 

Secondly, for the four-component coupled system \eqref{C-LE-D}, we are mainly interested in examining how $k_1$, $k_2$, and $\al$ can affect the stability of the symmetric steady state $(\widetilde{u}(x;\ep),\widetilde{v}(x;\epsilon), \widetilde{u}(x;\ep),\widetilde{v}(x;\epsilon))$. To that end, we first assume that $\tau\in(\tau_c^\ep, +\infty)$ so that $(\widetilde{u}(x;\ep),\widetilde{v}(x;\epsilon), \widetilde{u}(x;\ep),\widetilde{v}(x;\epsilon))$ is stable when there are no couplings between the two reactors. Need to note that $T=1/\alpha$ is the average delay. Thus, if $\al=+\infty$, then there is no time delay. If $0<\al<+\infty$, then there involves time delay. Moreover, the larger (resp., smaller) $\al$ is, the smaller (resp., larger) the average time delay $T$ is. We have the following results:
\begin{enumerate}
	\item The non-delayed coupling ($\al=+\infty$). It is proved that, roughly speaking, the first quadrant of the $(k_1,k_2)$ parameter space can be divided into two distinct regions: one is stable region $\Gamma_2^\ep\cup\Gamma_3^\ep$ and the other one is the \lq\lq Turing unstable\rq\rq region $\Gamma_1^\ep$, where $\Gamma_1^\ep$ and $\Gamma_2^\ep$ have the common boundary $\mathcal C^\ep$, which is indeed the primary steady state bifurcation curve (see Figure \ref{Fig.gamep}). More precisely, for sufficiently small $\ep>0$, there exists a constant $\rho_0^*>0$ such that if $k_1>\rho_0^*/2$, then $(\widetilde{u}(x;\ep),\widetilde{v}(x;\ep),\widetilde{u}(x;\ep),\widetilde{v}(x;\ep))$ is always stable regardless of $k_2>0$, and that if $k_1<\rho_0^*/2$, then there exists a $k_2^\ep>0$, such that $(\widetilde{u}(x;\ep),\widetilde{v}(x;\ep),\widetilde{u}(x;\ep),\widetilde{v}(x;\ep))$ is stable when $k_2<k_2^\ep$, while Turing unstable when $k_2>k_2^\ep$. At $k_2=k_2^\ep$, near the symmetric solution $(\widetilde{u}(x;\ep),\widetilde{v}(x;\ep),\widetilde{u}(x;\ep),\widetilde{v}(x;\ep))$, the Turing-type steady state bifurcation occurs (see Theorem \ref{th.2-1}). Meanwhile, Hopf bifurcations around $(\widetilde{u}(x;\ep),\widetilde{v}(x;\ep),\widetilde{u}(x;\ep),\widetilde{v}(x;\ep))$ can never be possible. Moreover, it seems that the smaller $\rho_0^*$ is, the larger the \lq\lq stable\rq\rq region is, where $\rho_0^*$ depends on the diffusion rate $d$ and the spatial interval $\Omega$ (see \eqref{App1} in Appendix). Need to note that the special case of $k_1=0$ and $k_2>0$ has been studied by Takaishi, Mimura and Nishiura \cite{TMN}, and our results indeed extend the results of \cite{TMN} to the general case of $k_1>0$ and $k_2>0$.
	\item The delayed coupling ($0<\al<+\infty$). We choose $\al$ to be the main bifurcation parameter, where $T=1/\alpha$ is the average delay. It is shown that the first quadrant of $(k_1,k_2)$ parameter space can be re-divided into four distinct regions (see Fig. \ref{Fig.s-dep}): $\Gamma_{1}$ is the unstable region (regardless of $\al$) and $\Gamma_{3-2}$ is the stable region (regardless of $\al$), indicating that in regions $\Gamma_{1}$ and $\Gamma_{3-2}$, the coupled system is robust to time delay (see Theorem \ref{cor.ue}); $\Gamma_{2}\cup\Gamma_{3-1}$ is the potential \lq\lq bifurcation\rq\rq region, where Hopf bifurcation may occur for some suitable $\alpha$. The regions $\Gamma_{1}$ and $\Gamma_{2}$ have the common boundary $\mathcal{C}$, where the bifurcation of higher co-dimension (with respect to $\alpha$) might occur. More precisely, in the potential \lq\lq bifurcation\rq\rq region $\Gamma_{2}\cup\Gamma_{3-1}$, under the conditions of Theorem \ref{th.3-1}, there exists a critical value $\alpha_H^\ep>0$ depending on $\ep$, such that the symmetric solution $(\widetilde{u}(x;\ep),\widetilde{v}(x;\epsilon), \widetilde{u}(x;\ep),\widetilde{v}(x;\epsilon))$ is stable when $\alpha\in(\alpha_H^\ep,\infty)$, while unstable when $\alpha$ is slightly less than $\alpha_H^\ep$; at $\alpha=\alpha_H^\ep$, system \eqref{C-LE-D} undergoes a Hopf bifurcation near the symmetric solution. 
\item Similarities and differences (between $\al=+\infty$ and $0<\al<+\infty$). As $\ep\to0^+$, $\Gamma_1^\ep\to\Gamma_1$, $\Gamma_2^\ep\to\Gamma_2$ and $\mathcal C^\ep\to\mathcal C$. For sufficiently small $\ep>0$, $\Gamma_3^\ep\equiv\Gamma_{3-1}\cup\Gamma_{3-2}$. When $\ep>0$ is sufficiently small, there exists at least a subregion $\Gamma_1^{\ep,u}$ of $\Gamma_1^\ep$ such that in this subregion $(\widetilde{u}(x;\ep),\widetilde{v}(x;\ep),\widetilde{u}(x;\ep),\widetilde{v}(x;\ep))$ is always unstable in both $\al=+\infty$ and $0<\al<+\infty$, showing that in such case time delay can never change the stability of the steady state; there also exists at least a subregion $\Gamma_2^{\ep,u}$ of $\Gamma_2^\ep\cup\Gamma_{3-2}$ such that in this subregion $(\widetilde{u}(x;\ep),\widetilde{v}(x;\ep),\widetilde{u}(x;\ep),\widetilde{v}(x;\ep))$ is stable when $\al=+\infty$, but unstable for some $\alpha\in(0,+\infty)$, implying that in such case, time delay does change the stability. In the subregion $\Gamma_{3-2}$ of $\Gamma_3^\ep$, $(\widetilde{u}(x;\ep),\widetilde{v}(x;\ep),\widetilde{u}(x;\ep),\widetilde{v}(x;\ep))$ is always stable in both $\al=+\infty$ and $0<\al<+\infty$, showing again that in such case time delay can never change the stability of the steady state.
\end{enumerate}

Finally, our main novelties are listed in the following way: for system \eqref{Decoupled}, it is firstly reported that it can exhibit Hopf bifurcation branching from non-constant steady state. For system \eqref{C-LE-D}, it is the first to report Turing instability of its non-constant steady state. A different scenery is that in our model there are two different kinds of diffusions: intra-reactor diffusions ($\ep$ and $d$) and inter-reactor diffusions ($k_1$ and $k_2$). Our Turing instability results are different from previous works (\cite{LE2, JSWY, Ma, Mo0, Mo, NT, De, RCF, SY0, SY1, SY2, SY3, WY}). Moreover, our analysis for system \eqref{C-LE-D} follows the lines of \cite{NM}, but we consider the problem in a quite different setting (4-component coupled system with distributed delays). We show that the basic framework of \cite{NM} works for the coupled system \eqref{C-LE-D}, but the generalization is non-trivial as analysis for the eigenvalue problems is difficult.

The present paper is organized as follows. In \S 2, for system \eqref{Decoupled}, we study the existence and stability of positive non-constant steady state $(\widetilde{u}(x;\ep),\widetilde{v}(x;\epsilon)$. In particular, we show the existence of Hopf bifurcation branching from the non-constant steady state solution. In \S 3 and \S 4, we study the stability and instability of $(\widetilde{u}(x;\ep),\widetilde{v}(x;\epsilon), \widetilde{u}(x;\ep),\widetilde{v}(x;\epsilon))$ with respect to system \eqref{C-LE-D} in the case of $\alpha=+\infty$ and $\alpha\in(0,+\infty)$, respectively. In \S 5, we end up our investigations by drawing some conclusions. In Appendix, we list some well-known results on the spectral properties of the elliptic operator. 

\section{Dynamics of the decoupled subsystem \eqref{Decoupled}: one reactor}
In this section, we consider the dynamics of subsystem \eqref{Decoupled}. For convenience, we copy \eqref{Decoupled} here:
\begin{equation}\label{2.1}
	\begin{cases}
		\begin{split}
			\tau\dfrac{\partial u}{\partial t}&=\epsilon u_{xx}+\dfrac{1}{\epsilon\sigma}\bigg(a-u-\dfrac{4uv}{1+u^2}\bigg),~&x\in(0,\ell),\;t>0,\\
			\dfrac{\partial v}{\partial t}&=dv_{xx}+u-\dfrac{uv}{1+u^2},~&x\in(0,\ell),\;t>0,\\
			u_x(0&,t)=u_x(\ell,t)=v_x(0,t)=v_x(\ell,t)=0,~&t>0.
		\end{split}
	\end{cases}
\end{equation}

Clearly, system \eqref{2.1} has a unique constant steady state solution $(u_*,v_*):=\big(\frac{a}{5}, 1+\frac{a^2}{25}\big)$.

The linearized operator of system \eqref{2.1} evaluated at $(u_*,v_*)$ is given by $L:=[L_{ij}]_{2\times2}$, where
\begin{equation*}
	L_{11}:=\dfrac{3a^2-125}{\sigma\epsilon\tau(a^2+25)},\;L_{12}:=-\dfrac{20a}{\sigma\epsilon\tau(a^2+25)},\;L_{21}:=\dfrac{2a^2}{a^2+25},\;L_{22}:=-\dfrac{5a}{a^2+25}.
\end{equation*}

Assume that $a>\dfrac{5}{3}\sqrt{15}$ so that system \eqref{2.1} is an activator-inhibitor system in the sense that $L_{11}>0$, $L_{12}<0$, $L_{21}>0$ and $L_{22}<0$.

The steady state problem of \eqref{2.1} is governed by the following elliptic equations
\begin{equation}\label{2.2}
	\begin{cases}
		\begin{split}
			&\epsilon^2u_{xx}+\dfrac{1}{\sigma}\big(a-u-\dfrac{4uv}{1+u^2}\big)=0,~&x\in(0,\ell),\\
			&dv_{xx}+u-\dfrac{uv}{1+u^2}=0,~&x\in(0,\ell),\\
			&u'(0)=u'(\ell)=v'(0)=v'(\ell)=0.
		\end{split}
	\end{cases}
\end{equation}

We denote $C_\epsilon ^2\left( [0,\ell] \right)$ by the subspace of $C^2\left( [0,\ell]\right)$ with norm
\begin{equation*}
	\Vert u  \Vert_{C_\epsilon ^2([0,\ell])}:= \sum\limits_{k = 0}^2 {\max\limits_{x\in[0,\ell]} \left| {{{\left( {\epsilon \frac{d}{{dx}}} \right)}^k}u\left( x \right)} \right|}.
\end{equation*}

Then, we have the following results on the existence of positive singularly perturbed steady state solutions of system \eqref{2.1} for small $\epsilon>0$.
\begin{theorem}\label{th.1-0}
	Suppose that $a>\dfrac{5}{3}\sqrt{15}$ holds. Then, the following conclusions hold true:
	\begin{enumerate}
		\item There exist positive constants $d_0$ and $\epsilon_0$, such that for any $0<\epsilon<\epsilon_0$ and $d>d_0$, system \eqref{2.1} has an $\epsilon$-family of positive non-constant steady state solutions $(\widetilde{u}(x;\ep),\widetilde{v}(x;\epsilon))\in C_{\ep}^{2}([0,\ell])\times C^{2}([0,\ell])$. 
		\item Moreover, for any $\kappa>0$ and $I_{\kappa}:=(x^*-\kappa,x^*+\kappa)$,
		\begin{equation}\label{tends}
			\begin{split}
				\lim\limits_{\ep\to 0^+} \widetilde{u}(x;\ep)=&U^*(x)\;\text{uniformly on $(0,\ell)\backslash I_{\kappa}$},\;\lim\limits_{\ep\to 0^+} \widetilde{v}(x;\ep)=V^*(x)\;\text{uniformly on $(0,\ell)$},
			\end{split}
		\end{equation}
		where $x^*\in(0,\ell)$ is the layer position, and $(U^*(x), V^*(x))$ is the discontinuous solution of the reduced system $($$\epsilon=0$$)$ of \eqref{2.1}, which is defined in \eqref{U*V*}.
	\end{enumerate}
\end{theorem}

\begin{proof} Our proof falls into the general framework of \cite{NM}. To prove the desired results, we need to verify (A0)-(A5) stated in \cite{NM} (pp. 484-485).
	
	Step 1. To show that there is an intersection of 
	\begin{equation*}
		f(u,v):=\dfrac{1}{\sigma}\big(a-u-\dfrac{4uv}{1+u^2}\big)=0,\;\text{and}\;g(u,v):=u-\dfrac{uv}{1+u^2}=0. 
	\end{equation*}
	This is obvious since the unique intersection point is $(u_*,v_*)$.
	
	Step 2. To check that: $f(u,v)=0$ is sigmoidal and it consists of three continuous curves $u=h_-(v)$, $u=h_0(v)$ and $u=h_+(v)$ which are defined in the intervals $I_-$, $I_0$, $I_+$, respectively, such that 
	$h_-(v)<h_0(v)<h_+(v)$ holds for $v\in (\underline{v},\overline{v})$ and $h_+(v)$ (resp., $h_-(v)$) coincides with $h_0(v)$ at only one point $v=\overline{v}$ (resp., $v=\underline{v}$), where 
	$\underline{v}:=\min I_-$, $\overline{v}:=\max I_+$. 
	
	Indeed, by Lemmas 4.1 and 4.2 of \cite{JNT}, if $a>\frac{5}{3}\sqrt{15}$, then $f(u,v)=0$ is sigmoidal and it consists of three continuous curves $u=h_-(v)$, $u=h_0(v)$, $u=h_+(v)$ (see Fig. \ref{Fig. f & g}). Moreover, by (4.9) and (4.10) of \cite{JNT}, we have
	\begin{equation*}
		\begin{split}
			\max\bigg\{1,&\frac{8}{a}\bigg\}<\underline{u}:=h_-(\underline v)<\frac{a}{5},\;\frac{2a}{5}<\overline{u}:=h_+(\overline v)<\frac{a}{2},\\
			\max&\left\{2,1+\frac{64}{a^2}\right\}<\underline{v}<1+\frac{a^2}{25}<\overline{v}<\frac{a^2}{12}-\frac{1}{4}.
		\end{split}
	\end{equation*}
	
	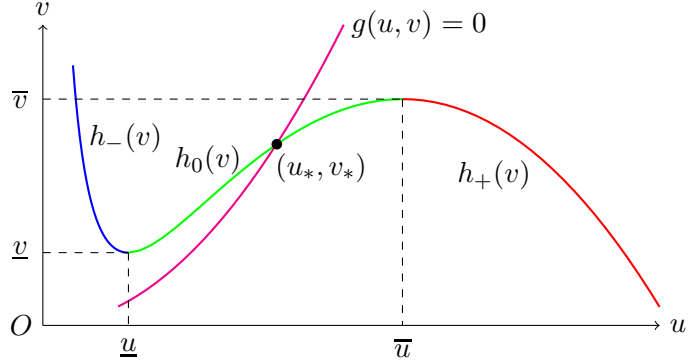
\begin{figure}[!t]
		\centering
		\begin{tikzpicture}
			\draw[->] (0,0) --(8.2,0) node[right] {$u$};
			\draw[->] (0,0) --(0,4) node[above] {$v$};
			\draw  (0,0) node[left] {$O$};
			\draw[blue,thick] plot[domain=0.4:1.1378,line width=12pt] (\x,{(\x^2 +1)*(10-\x)/(4*\x)-3.5});
			\draw[green,thick] plot[domain=1.1378:4.781] (\x,{(\x^2 +1)*(10-\x)/(4*\x)-3.5});
			\draw[red,thick] plot[domain=4.781:8.2,line width=12pt] (\x,{(\x^2 +1)*(10-\x)/(4*\x)-3.5});
			\draw[magenta,thick] plot[domain=1:4,line width=12pt] (\x,{\x^2/4});
			\draw node at (1.1,2.5){$h_-(v)$};
			\draw node at (2.2,2.2){$h_0(v)$};
			\draw node at (6,2){$h_+(v)$};
			\draw node at (5,4){$g(u,v)=0$};
			\fill (3.11,2.41) circle (2pt) node at (3.7,2.1) {$(u_*,v_*)$};
			\draw [dashed](1.1378,0.968076)--(1.1378,0) node at (1.1378,-0.3){$\underline{u}$};
			\draw [dashed](4.781,3.01091)--(4.781,0) node at (4.781,-0.3){$\overline{u}$};
			\draw [dashed](0,0.968076)--(1.1378,0.968076) node at (-0.3,1){$\underline{v}$};
			\draw [dashed](0,3.01091)--(4.781,3.01091) node at (-0.3,3.01091){$\overline{v}$};
		\end{tikzpicture}
		\caption{Nullclines of $u$ and $v$: $f(u,v)=0$ and $g(u,v)=0$.}
		\label{Fig. f & g}
	\end{figure}
	
	Step 3. To prove that $M(v)=0$ has a unique root $\widehat v$ in $(\underline v, \overline v)$ and that $M'(\widehat v)<0$, where 
	\begin{equation*}
		M(v):=\ds\int_{{h_ - }(v)}^{{h_ + }(v)}f(s,v)ds=\dfrac{1}{\sigma}\int_{{h_ - }(v)}^{{h_ + }(v)} \bigg(a-s-\dfrac{4sv}{1+s^2}\bigg)ds.
	\end{equation*}
	
	Indeed, for any $u,v>0$, $f_v(u,v)<0$. Thus,  in the first quadrant, $f(u,v)<0$ in the region above $f(u,v)=0$; while  $f(u,v)>0$ in the region below $f(u,v)=0$. Thereby, we have $M(\underline{v})>0$ and $M(\overline{v})<0$. Moreover, by $f(h_+(v),v)=f(h_-(v),v)=0$, we have
	\begin{equation*}
		M'(v)=-\dfrac{4}{\sigma}\int_{h_-(v)}^{h_+(v)}\dfrac{s}{1+s^2}ds<0,
	\end{equation*}
	implying that $M(v)=0$ has a unique root $\widehat v$ in $(\underline v, \overline v)$ and $M'(\widehat v)<0$.
	
	Step 4. To verify that for any $(u,v)\in \mathscr R_+\cup \mathscr R_-$, $f_u(u,v)<0$, where 
	\begin{equation*}
		\begin{split}
			\mathscr R_+:=&\big\{ (u,v)|u=h_+(v)\;\; \text{for} \;\;\widehat v\leq v<\overline{v} \big\},\; \mathscr R_-:=\big\{(u,v)|u=h_-(v)\;\; \text{for} \;\;\underline{v}\leq v<\widehat v\big\}.
		\end{split}
	\end{equation*}
	
	In fact, for any $u,v>0$, $f_v(u,v)<0$. Then, in the left (resp., right) side of $\mathscr R_\pm$, we have $f(u,v)>0$ (resp., $<0$). Thus, $f_u(u,v)\big|_{\mathscr R_{\pm}}<0$. 
	
	Step 5. To show that 
	\begin{equation}\label{det*}
		g(u,v)\big|_{(u,v)\in\mathscr R_-}<0<g(u,v)\big|_{(u,v)\in\mathscr R_+},\;\det\left[ \dfrac{\pa(f,g)}{\pa(u,v)}\right]\bigg|_{(u,v)\in\mathscr R_{\pm}} >0.
	\end{equation} 
	
	Indeed, for any $u,v>0$, $g_v(u,v)<0$. Thus, one can verify that $g(u,v)|_{\mathscr R_-}<0<g(u,v)|_{\mathscr R_+}$, since the intersection point $(u_*,v_*)$ lies on the middle branch of  $f(u,v)=0$. Moreover, 
	\begin{equation*}
		\dfrac{\pa(f,g)}{\pa(u,v)}\bigg|_{(u,v)\in\mathscr R_{\pm}}=\dfrac{5u}{\sigma(1+u^2)^3}\bigg|_{(u,v)\in\mathscr R_{\pm}}>0.
	\end{equation*}
	
	So far, we have verified $(A0)$-$(A5)$ in \cite{NM}. Then, by \cite{NM}, we can conclude the existence of the positive non-constant steady state solution $(\widetilde u(x;\epsilon), \widetilde v(x;\epsilon))$ for $0<\epsilon<\epsilon_0$ and $d>d_0$ with $\epsilon_0>0$ and $d_0>0$.
	
	In case of $\ep=0$, the elliptic system \eqref{2.2} is reduced to
	\begin{equation}\label{R2.2}
		\begin{cases}
			\begin{split}
				&a-u-\dfrac{4uv}{1+u^2}=0,&x\in(0,\ell),\\
				&dv_{xx}+u-\dfrac{uv}{1+u^2}=0,&x\in(0,\ell),\\
				&u'(0)=u'(\ell)=v'(0)=v'(\ell)=0.
			\end{split}
		\end{cases}
	\end{equation}
	
	Solving $u$ from the first equation of \eqref{R2.2}, we have $u=h_i(v)$ for $i=0,+,-$. Substituting $u=h_i(v)$ into the second equation of \eqref{R2.2}, we have (see \cite{MTH})
	\begin{equation*}\label{R2.3}
		\begin{cases}
			\begin{split}
				&dv_{xx}+G_i(v)=0,~&x\in(0,\ell),\\
				&v'(0)=v'(\ell)=0,
			\end{split}
		\end{cases}
	\end{equation*}
	where, for $i=0,+,-$,
	\begin{equation*}
		G_i(v):=h_i(v)-\dfrac{h_i(v)v}{1+h_i(v)^2}.
	\end{equation*}
	
	Choosing $\widehat{v}\in I_-\cup I_+$, given by Step 3, as the separating point, we can combine $G_\pm(v)$  by
	\begin{equation*}
		G(v;\widehat{v})=
		\begin{cases}
			G_-(v),~&v\in I_-\cup\{v<\widehat{v}\},\\
			G_+(v),~&v\in I_+\cup\{v>\widehat{v}\},
		\end{cases}
	\end{equation*}
	which is discontinuous at $\widehat{v}$ of the first kind.
	
	As in \cite{MTH,NF,NM}, there exists a $x^*\in I$, such that \eqref{R2.2} can be decomposed by the following way:
	\begin{equation}\label{R2.3-1}
		\begin{cases}
			\begin{split}
				&dV_{xx}+h_-(V)-\dfrac{h_-(V)V}{1+h_-(V)^2}=0,\;x\in(0,x^*),\\
				&V'(0)=0,\;\; V(x^*)=\widehat v,
			\end{split}
		\end{cases}
	\end{equation}
	and
	\begin{equation}\label{R2.3-2}
		\begin{cases}
			\begin{split}
				&dV_{xx}+h_+(V)-\dfrac{h_+(V)V}{1+h_+(V)^2}=0,\;x\in(x^*,\ell),\\
				&V(x^*)=\widehat v,\;\; V'(\ell)=0,
			\end{split}
		\end{cases}
	\end{equation}
	with $C^1$ matching condition: $\frac{dV^-}{dx}(x^*)=\frac{dV^+}{dx}(x^*)$, where $V^-$ and $V^+$ are the solutions of \eqref{R2.3-1} and \eqref{R2.3-2}, respectively.
	
	Then, $(U^*(x),V^*(x))$ is the weak solution of system \eqref{R2.2}, where
	\begin{equation}\label{U*V*}
		V^*(x)=
		\begin{cases}
			V^-,\;x\in(0,x^*),\\
			V^+,\;x\in(x^*,\ell).
		\end{cases},\;
		U^*(x)=
		\begin{cases}
			h_-(V^-),\;x\in(0,x^*),\\
			h_+(V^+),\;x\in(x^*,\ell).
		\end{cases}
	\end{equation}
	
Finally, the asymptotic behaviors in terms of \eqref{tends} can be obtained by Appendix 1 in \cite{NF} and Theorem 1.1 in \cite{NM}. For more details, we refer readers to \cite{MTH,NF,NM}.
\end{proof}

Next, we concentrate on the stability/instability of $(\widetilde{u}(x;\epsilon),\widetilde{v}(x;\epsilon))$ for $\epsilon\in(0,\epsilon_0)$. We have the following results.
\begin{theorem}\label{th.1-1}
	Suppose that $a>\dfrac{5}{3}\sqrt{15}$ holds so that $(\widetilde{u}(x;\epsilon),\widetilde{v}(x;\epsilon))$ exists for any $\epsilon\in(0,\epsilon_0)$. Then, there exists a $\tau_c^{\ep}>0$, such that $(\widetilde{u}(x;\epsilon),\widetilde{v}(x;\epsilon))$ is locally asymptotically stable with respect to system \eqref{2.1} when $\tau>\tau_c^{\ep}$, while unstable when $\tau<\tau_c^{\ep}$. In particular, at $\tau=\tau_c^{\ep}$, Hopf bifurcation occurs around $(\widetilde{u}(x;\epsilon),\widetilde{v}(x;\epsilon))$. 
\end{theorem}
\begin{proof}
	Our results are really the direct consequences of Theorem 3.1 in \cite{NM} since all the conditions $(A0)$-$(A5)$ in \cite{NM} are satisfied. For the convenience of our later discussions in next sections, we list the sketch of the proof below.
	The linearized system of \eqref{2.1} evaluated at $(\widetilde{u}(x;\epsilon),\widetilde{v}(x;\epsilon))$ is given by
	\begin{equation}\label{0L1}
		\begin{cases}
			\begin{split}
				\epsilon\tau\dfrac{\partial u}{\partial t}&=\epsilon^2u_{xx}+f_u^\ep u+f_v^\ep v,~&x\in(0,\ell),\;t>0,\\
				\dfrac{\partial v}{\partial t}&=dv_{xx}+g_u^\ep u+g_v^\ep v,~&x\in(0,\ell),\;t>0,\\
				u_x(x&,t)=v_x(x,t)=0,~&x=0,\ell,\;t>0,
			\end{split}
		\end{cases}
	\end{equation}
	where  
	\begin{equation}\label{fuvep}
		\begin{split}
			f_u^\ep:=&-\dfrac{1}{\sigma}\bigg(1+\dfrac{4(1-\widetilde u^2(x;\ep))\widetilde v(x;\ep)}{(1+\widetilde u^2(x;\ep))^2}\bigg),\\
			f_v^\ep:=&-\dfrac{4\widetilde u(x;\ep)}{\sigma(1+\widetilde u^2(x;\ep))},\;
			g_u^\ep:=1-\dfrac{(1-\widetilde u^2(x;\ep))\widetilde v(x;\ep)}{(1+\widetilde u^2(x;\ep))^2},\;g_v^\ep:=-\dfrac{\widetilde u(x;\ep)}{1+\widetilde u^2(x;\ep)}.
		\end{split}
	\end{equation}
	
	The eigenvalue problem associated with \eqref{0L1} is given by
	\begin{equation}\label{2.12}
		\begin{cases}
			\begin{split}
				\begin{pmatrix}
					\ep^2\frac{d^2}{dx^2}+f_u^\ep&f_v^\ep\\
					g_u^\ep&d\frac{d^2}{dx^2}+g_v^\ep\\
				\end{pmatrix}
				\begin{pmatrix}
					\phi\\
					\psi
				\end{pmatrix}&=\la
				\begin{pmatrix}
					\ep\tau&0\\
					0&1
				\end{pmatrix}
				\begin{pmatrix}
					\phi\\
					\psi
				\end{pmatrix},\;x\in(0,\ell),\\
				\phi'(0)=\phi'(\ell)=\psi'(0)=\psi'(\ell)&=0,
			\end{split}
		\end{cases}
	\end{equation}
	where $(\phi,\psi)^T\in\big(H^2(0,\ell)\cap H^1_N(0,\ell)\big)\times\big(H^2(0,\ell)\cap H^1_N(0,\ell)\big)$, and
	\begin{equation*}
		H^1_N(0,\ell)=\left\lbrace \phi\in H^1(0,\ell)| \phi'(0)=\phi'(\ell)=0\right\rbrace.
	\end{equation*}
	
	By Theorem 3.1 in \cite{NM}, it follows that there exists a unique $\tau_c^\epsilon>0$ such that for any $\tau\in(\tau_c^\epsilon, \infty)$, all the eigenvalues of the eigenvalue problem \eqref{2.12} have strictly negative real parts, while for any $\tau\in(0,\tau_c^\epsilon)$,  the eigenvalue problem \eqref{2.12} has at least one eigenvalue with positive real part. In particular, at $\tau=\tau_c^\epsilon$, the eigenvalue problem \eqref{2.12} has a pair of simple purely imaginary eigenvalues and the transversality condition holds. So far, based on these observations, we can complete the proof.
\end{proof}

\section{Dynamics of the coupled system \eqref{C-LE-D}: the case of $\al\rightarrow+\infty$}
In this section, we shall consider the following non-delayed system
\begin{equation}\label{P-1}
	\begin{cases}
		\begin{aligned}
			\tau\dfrac{\partial u_i}{\partial t}&=\epsilon\dfrac{\partial^2 u_i}{\partial x^2}+\dfrac{1}{\sigma\epsilon}\bigg(a-u_i-\dfrac{4u_iv_i}{1+u_i^2}\bigg)+k_1(u_j-u_i),~&x\in(0,\ell),\;t>0,\\
			\dfrac{\partial v_i}{\partial t}&=d\dfrac{\partial^2 v_i}{\partial x^2}+u_i-\dfrac{u_iv_i}{1+u_i^2}+k_2(v_j-v_i),~&x\in(0,\ell),\;t>0,\\
			\dfrac{\partial u_i}{\partial x}&=\dfrac{\partial v_i}{\partial x}=0,~&x=0,\ell,\;\;\;t>0,\\
		\end{aligned}
	\end{cases}
\end{equation}
where $i,j\in\{1,2\}$ and $i\neq j$; $\epsilon\in(0,\epsilon_0)$, in which $\epsilon_0>0$ is stated in Theorems \ref{th.1-0} and \ref{th.1-1}. Clearly, $(u_1,v_1,u_2,v_2)=(\widetilde{u}(x;\epsilon),\widetilde{v}(x;\epsilon),\widetilde{u}(x;\epsilon),\widetilde{v}(x;\epsilon))$ is the symmetric steady state solution of system \eqref{P-1}, where $\widetilde{u}(x;\epsilon)$ and $\widetilde{v}(x;\epsilon)$ are stated in Theorem \ref{th.1-0}.

We would like to mention that the non-delayed system \eqref{P-1} is indeed the limiting system of system \eqref{C-LE-D} when $\al\rightarrow+\infty$. In fact, we can extend the domain of the weak kernel $F(t)$ to $(-\infty, +\infty)$ by defining 
\begin{equation*}
	\widetilde{F}(t):=
	\begin{cases}
		\begin{split}
			&F(t)=\al e^{-\al t},~&t\geq0,\\
			&0,~&t<0.
		\end{split}
	\end{cases}
\end{equation*}

Then, we have 
\begin{equation*}
	\int_{-\infty}^{+\infty} \widetilde{F}(s)ds=\int^{+\infty}_{0} F(s)ds=1,\;\lim\limits_{\alpha\to+\infty} \widetilde F(t)=\delta(t)=
	\begin{cases}
		\begin{split}
			&+\infty,\;&\text{if $t=0$},\\
			&0,\;&\text{if $t\neq 0$}.
		\end{split}
	\end{cases}	
\end{equation*}

Therefore, as $\alpha\to+\infty$, for $i=1,2$, we have
\begin{equation*}
	\int_{-\infty}^{t} F(t-s)u_i(x,s)ds=\int_{-\infty}^{\infty} \widetilde F(t-s)u_i(x,s)ds\to\int_{-\infty}^{\infty} \delta(s-t)u_i(x,s)ds=u_i(x,t),
\end{equation*}
where we use the fact that $\delta(t-s)=\delta(s-t)$. Thus, as $\al\rightarrow+\infty$, the limiting system of system \eqref{C-LE-D} is indeed the non-delayed system \eqref{P-1}.

The linearized system of \eqref{P-1} evaluated at $(\widetilde{u}(x;\epsilon),\widetilde{v}(x;\epsilon),\widetilde{u}(x;\epsilon),\widetilde{v}(x;\epsilon))$ is given by
\begin{equation}\label{LP-pl-1}
	\begin{cases}
		\begin{aligned}
			\tau\dfrac{\partial \phi_i}{\partial t}&=\epsilon\dfrac{\partial^2\phi_i}{\partial x^2}+\dfrac{1}{\epsilon}\big(f_u^\ep\phi_i+f_v^\ep\psi_i\big)+k_1(\phi_j-\phi_i),~&x\in(0,\ell),\;t>0,\\
			\dfrac{\partial \psi_i}{\partial t}&=d\dfrac{\partial^2\psi_i}{\partial x^2}+g_u^\ep\phi_i+g_v^\ep\psi_i+k_2(\psi_j-\psi_i),~&x\in(0,\ell),\;t>0,\\
			\dfrac{\partial \phi_i}{\partial x}&=\dfrac{\partial \psi_i}{\partial x}=0,~&x=0,\ell,\;\;\;t>0,\\
		\end{aligned}
	\end{cases}
\end{equation}
where $i,j\in\{1,2\}$ and $i\neq j$, and $f_u^\ep, f_v^\ep, g_u^\ep, g_v^\ep$ are precisely defined in \eqref{fuvep}.

For $i=1,2$, rewrite $\phi_i(x,t)$ and $\psi_i(x,t)$ in the following form (separation of variables)
\begin{equation*}
	\begin{split}
		\phi_1(x,t):=&w_1(x)e^{\la t},\; \phi_2(x,t):=w_2(x)e^{\la t},\;
		\psi_1(x,t):=z_1(x)e^{\la t},\;\;\psi_2(x,t):=z_2(x)e^{\la t},
	\end{split}
\end{equation*}
where $\la$ is to be determined later.

Following \cite{LE2, TMN}, by introducing
\begin{equation}\label{nform}
	\begin{split}
		w_s(x):=&\frac{1}{2}(w_1(x)+w_2(x)),\; w_a(x):=\frac{1}{2}(w_1(x)-w_2(x)),\\
		z_s(x):=&\frac{1}{2}(z_1(x)+z_2(x)),\;\;\;\;z_a(x):=\frac{1}{2}(z_1(x)-z_2(x)),
	\end{split}
\end{equation}
we can reduce the linearized system \eqref{LP-pl-1} to the following two eigenvalue problems
\begin{equation}\label{P-L-2a}
	\begin{cases}
		\begin{split}
			&\ep^2w_s''(x)+f_u^\ep w_s(x)+f_v^\ep z_s(x)=\ep\tau\la w_s(x),~&x\in(0,\ell),\\
			&dz_s''(x)+g_u^\ep w_s(x)+g_v^\ep z_s(x)=\la z_s(x),~&x\in(0,\ell),\\
			&w_s'(0)=w_s'(\ell)=z_s'(0)=z_s'(\ell)=0,
		\end{split}
	\end{cases}
\end{equation}
and
\begin{equation}\label{P-L-2b}
	\begin{cases}
		\begin{split}
			&\ep^2w_a''(x)+(f_u^\ep-2\ep k_1) w_a(x)+f_v^\ep z_a(x)=\ep\tau\la w_a(x),~&x\in(0,\ell),\\
			&dz_a''(x)+g_u^\ep w_a(x)+(g_v^\ep-2k_2) z_a(x)=\la z_a(x),~&x\in(0,\ell),\\
			&w_a'(0)=w_a'(\ell)=z_a'(0)=z_a'(\ell)=0.
		\end{split}
	\end{cases}
\end{equation}

Clearly, the eigenvalue problem \eqref{P-L-2a} is exactly \eqref{2.12} in Section 2, and the distributions of its eigenvalues are clear (see the proof of Theorem \ref{th.1-1} in this paper, or Theorem 3.1 in \cite{NM}). It remains to consider the eigenvalue problem \eqref{P-L-2b}.

Following \cite{NF, NM}, by rewriting
\begin{equation*}
	\la=:\la_R+i\la_I,\;w_a=:w_R+iw_I,\;z_a=:z_R+iz_I,
\end{equation*}
we can reduce the eigenvalue problem \eqref{P-L-2b} to the following equations
\begin{equation}\label{P-L-3}
	\begin{aligned}
		\begin{cases}
			(\ep^2\dfrac{d^2}{dx^2}+f_u^\ep-2\ep k_1) w_R+f_v^{\ep}z_R=\epsilon\tau(\la_R w_R-\la_I w_I),\\
			(\ep^2\dfrac{d^2}{dx^2}+f_u^\ep-2\ep k_1)w_I+f_v^{\ep}z_I=\epsilon\tau(\la_R w_I+\la_I w_R),\\
			(d\dfrac{d^2}{dx^2}+g_v^\ep-2k_2)z_R+g_u^{\ep}w_R=\la_R z_R-\la_I z_I,\\
			(d\dfrac{d^2}{dx^2}+g_v^\ep-2k_2)z_I+g_u^{\ep}w_I=\la_I z_R+\la_R z_I,
		\end{cases}
	\end{aligned}
\end{equation}
subject to homogeneous Neumann boundary conditions.

Define
\begin{equation*}
	\mathcal{L}^{\ep}:=\ep^2\dfrac{d^2}{dx^2}+f_u^\ep-2\ep k_1.
\end{equation*}

Choose $\mu^*$ satisfying
\begin{equation}\label{mu^*}
	0<\mu^*<\inf\limits_{x\in(0,\ell)}\frac{f_u^*g_v^*-f_v^*g_u^*}{-f_u^*},
\end{equation} 
where 
\begin{equation*}
	\begin{split}
		f_u^*:=&f_u(U^*, V^*),~f_v^*:=f_v(U^*, V^*),\;
		g_u^*:=g_u(U^*, V^*),~g_v^*:=g_v(U^*, V^*),
	\end{split}
\end{equation*}
in which $(U^*,V^*):=(U^*(x), V^*(x))$ is stated in Theorem \ref{th.1-0}, and $(f_u^*g_v^*-f_v^*g_u^*)/(-f_u^*)>0$ (see the proof of Theorem \ref{th.1-0}). 

For sufficiently small $\ep>0$, by the similar argument of Lemma 2.1 in \cite{NF} and Remark 2.2 in \cite{NM},  $(\mathcal{L}^{\ep}-\ep\tau\la)^{-1}$ is well-defined for $\text{Re}\;\la>-\mu^*$. In the following, we always assume that $\text{Re}\;\la>-\mu^*$. In case of $\text{Re}\;\la\leq-\mu^*$, all the eigenvalues of \eqref{LP-pl-1} have negative real parts, which do not affect the stability of $(\widetilde{u}(x;\epsilon),\widetilde{v}(x;\epsilon),\widetilde{u}(x;\epsilon),\widetilde{v}(x;\epsilon))$.	

Then, by solving the first two equations of \eqref{P-L-3} with respect to $w_R$ and $w_I$, we have
\begin{equation*}
	\begin{split}
		w_R&=\big[I+(\ep\tau\la_I)^2(\mathcal{L}^{\ep}-\ep\tau\la_R)^{-2}\big]^{-1}\big[\ep\tau\la_I(\mathcal{L}^{\ep}-\ep\tau\la_R)^{-2}(f_v^{\ep}z_I)-(\mathcal{L}^{\ep}-\ep\tau\la_R)^{-1}(f_v^{\ep}z_R)\big],\\
		w_I&=-\big[I+(\ep\tau\la_I)^2(\mathcal{L}^{\ep}-\ep\tau\la_R)^{-2}\big]^{-1}\big[(\mathcal{L}^{\ep}-\ep\tau\la_R)^{-1}(f_v^{\ep}z_I)+\ep\tau\la_I(\mathcal{L}^{\ep}-\ep\tau\la_R)^{-2}(f_v^{\ep}z_R)\big],
	\end{split}
\end{equation*}
where $I$ is the identity operator.

By Lemma \ref{lemNF1} (Appendix), the eigenvalue problem
\begin{equation*}
	\begin{cases}
		\begin{split}
			\ep^2\dfrac{d^2\phi}{dx^2}+f_u^\ep\phi&=\mu\phi,\;x\in(0,\ell),\\
			\phi'(0)=\phi'(\ell)&=0,
		\end{split}
	\end{cases}
\end{equation*}
has complete orthonormal pairs $($in $L^2(0,\ell)$ sense$)$ of eigenvalues and eigenfunctions $\{(\mu_n^{\ep}, \phi_n^{\ep}(x))\}_{n=0}^{\infty}$. 

By eigenfunction expansions, we have
\begin{equation}\label{WRIex-1}
	w_R=\sum\limits_{n \ge 0}w_R^n\phi_n^{\ep}=w_R^0\phi_0^{\ep}+w_R^\dag,\;
	w_I=\sum\limits_{n \ge 0}w_I^n\phi_n^{\ep}=w_I^0\phi_0^{\ep}+w_I^\dag,
\end{equation}
where $w_R^\dag:=\sum\limits_{n \ge 1}w_R^n\phi_n^{\ep}$, $w_I^\dag:=\sum\limits_{n \ge 1}w_I^n\phi_n^{\ep}$, in which
\begin{equation*}
	\begin{split}
		w_R^n=\ds\frac{-\mu_n^{\ep}+2\ep k_1+\ep\tau\la_R}{(\mu_n^{\ep}-2\ep k_1-\ep\tau\la_R)^2+(\ep\tau\la_I)^2}\big\langle {f_v^\varepsilon {z_R},\phi _n^\varepsilon }\big\rangle+\ds\frac{\ep\tau\la_I}{(\mu_n^{\ep}-2\ep k_1-\ep\tau\la_R)^2+(\ep\tau\la_I)^2}\big\langle f_v^\varepsilon {z_I},\phi _n^\varepsilon\big\rangle,\\
		w_I^n=\ds\frac{-\mu_n^{\ep}+2\ep k_1+\ep\tau\la_R}{(\mu_n^{\ep}-2\ep k_1-\ep\tau\la_R)^2+(\ep\tau\la_I)^2}\big\langle {f_v^\varepsilon {z_I},\phi _n^\varepsilon }\big\rangle-\ds\frac{\ep\tau\la_I}{(\mu_n^{\ep}-2\ep k_1-\ep\tau\la_R)^2+(\ep\tau\la_I)^2}\big\langle f_v^\varepsilon {z_R},\phi _n^\varepsilon\big\rangle.
	\end{split}
\end{equation*}
Hereafter, $\left\langle\cdot,\cdot \right\rangle $ denotes the inner product in $L^2(0,\ell)$ sense.

By \eqref{WRIex-1}, we have
\begin{equation*}
	\begin{split}
		w_R^\dag&=\big[I+(\ep\tau\la_I)^2(\mathcal{L}_1^{\ep}-\ep\tau\la_R)^{-2}\big]^{\dag}\big[\ep\tau\la_I(\mathcal{L}_1^{\ep}-\ep\tau\la_R)^{2\dag}(f_v^{\ep}z_I)-(\mathcal{L}_1^{\ep}-\ep\tau\la_R)^{\dag}(f_v^{\ep}z_R)\big],\\
		w_I^\dag&=-\big[I+(\ep\tau\la_I)^2(\mathcal{L}_1^{\ep}-\ep\tau\la_R)^{-2}\big]^{\dag}\big[(\mathcal{L}_1^{\ep}-\ep\tau\la_R)^{\dag}(f_v^{\ep}z_I)+\ep\tau\la_I(\mathcal{L}_1^{\ep}-\ep\tau\la_R)^{2\dag}(f_v^{\ep}z_R)\big],
	\end{split}
\end{equation*}
where
\begin{equation*}
	(\mathcal{L}_1^{\ep}-\ep\tau\la_R)^{\dag}=\sum\limits_{n \ge 1}\dfrac{\left\langle \cdot,\phi_n^{\ep}\right\rangle }{\mu_n^{\ep}-2\ep k_1-\ep\tau\la_R}\phi_n^{\ep},\;\;(\mathcal{L}_1^{\ep}-\ep\tau\la_R)^{2\dag}=\sum\limits_{n \ge 1}\dfrac{\left\langle \cdot,\phi_n^{\ep}\right\rangle }{(\mu_n^{\ep}-2\ep k_1-\ep\tau\la_R) ^2}\phi_n^{\ep},
\end{equation*}
and
\begin{equation*}
	\big[I+(\ep\tau\la_I)^2(\mathcal{L}_1^{\ep}-\ep\tau\la_R)^{-2}\big]^{\dag}=\sum\limits_{n \ge 1}\dfrac{(\mu_n^{\ep}-2\ep k_1-\ep\tau\la_R)^2 }{(\mu_n^{\ep}-2\ep k_1-\ep\tau\la_R)^2+(\ep\tau\la_I)^2}\left\langle \cdot,\phi_n^{\ep}\right\rangle\phi_n^{\ep}.
\end{equation*}

Substituting \eqref{WRIex-1} into the last two equations of \eqref{P-L-3},  we have
\begin{equation}\label{mz-1}
		\begin{pmatrix}
			T_{\la_R}^\ep&-\lambda_I + \ep\tau\lambda _IS^\ep\\
			\lambda _I + \ep\tau\lambda _IS^\ep&T_{\la_R}^\ep
		\end{pmatrix}
		\begin{pmatrix}
			z_R\\
			z_I
		\end{pmatrix}=-\dfrac{g_u^\ep\phi _0^\ep}{\sqrt{\varepsilon}} {M^\varepsilon }\begin{pmatrix}
			\big\langle z_R,  \frac{f_v^\ep\phi _0^\varepsilon}{\sqrt{\varepsilon}} \big\rangle\\
			\big\langle z_I,  \frac{f_v^\ep\phi _0^\varepsilon}{\sqrt{\varepsilon}} \big\rangle\\
		\end{pmatrix},
\end{equation}
where
\begin{equation}\label{TSM}
	\begin{cases}
		\begin{aligned}
			T_{\la_R}^\ep:=&-d\frac{d^2}{dx^2}-g_v^{\ep}+2k_2+\la_R+g_u^{\ep}\big[I+(\ep\tau\la_I)^2(\mathcal{L}_1^{\ep}-\ep\tau\la_R)^{-2}\big]^{\dag}(\mathcal{L}_1^{\ep}-\ep\tau\la_R)^{\dag}f_v^{\ep},\\
			S^\ep:=&-g_u^{\ep}\big[I+(\ep\tau\la_I)^2(\mathcal{L}_1^{\ep}-\ep\tau\la_R)^{-2}\big]^{\dag}(\mathcal{L}_1^{\ep}-\ep\tau\la_R)^{2\dag}f_v^{\ep},\\
			M^\ep:=&\begin{pmatrix}
				\frac{\rho(\ep)-2 k_1-\tau\la_R}{(\rho(\ep)-2 k_1-\tau\la_R)^2+(\tau\la_I)^2}&-\frac{\tau\lambda _I}{(\rho(\ep)-2 k_1-\tau\la_R)^2+(\tau\la_I)^2}\\
				\frac{\tau\lambda _I}{(\rho(\ep)-2 k_1-\tau\la_R)^2+(\tau\la_I)^2}&\frac{\rho(\ep)-2 k_1-\tau\la_R}{(\rho(\ep)-2 k_1-\tau\la_R)^2+(\tau\la_I)^2}
			\end{pmatrix}.
		\end{aligned}
	\end{cases}
\end{equation}

By Lemma \ref{lemNF1} (Appendix), we have 
\begin{equation*}
	\begin{split}
		\lim\limits_{\ep\to 0^+}\ds\frac{f_v^{\ep}\phi_0^{\ep}}{\sqrt{\ep}}&=-c_1^*\delta(x - x^*),\;
		\lim\limits_{\ep\to 0^+}\ds\frac{g_u^{\ep}\phi_0^{\ep}}{\sqrt{\ep}}=c_2^*\delta(x - x^*),\;\text{in $H^{-1}$ sense,}
	\end{split}
\end{equation*}
where
\begin{equation}\label{c1c2}
		c_1^*:=-\kappa^*M'(\widehat{v})>0,\;c_2^*:=\kappa^*\left[h_+(\widehat{v})-h_-(\widehat{v})+\left(\frac{h_-(\widehat{v})}{1+(h_-(\widehat{v}))^2}-\frac{h_+(\widehat{v})}{1+(h_+(\widehat{v}))^2}\right)\widehat{v}\right]>0.
\end{equation}

Set $\lim\limits_{\ep\to 0^+}\la=:\la^*=:\la_R^*+i\la_I^*$ and $\lim\limits_{\ep\to 0^+}z=:z^*=:z_R^*+iz_I^*$. Then, as $\epsilon\rightarrow0^+$, \eqref{mz-1} can be reduced to 
\begin{equation}\label{mz^*-1}
	\begin{pmatrix}
		T_{\la_R^*}&-\lambda _I^*\\
		\lambda _I^*&T_{\la_R^*}
	\end{pmatrix}\begin{pmatrix}
		{{z_R^*}}\\
		{{z_I^*}}
	\end{pmatrix}=c_1^*c_2^*\delta(x-x^*)M^*
	\begin{pmatrix}
		\left\langle z_R^*, \delta(x - x^*)\right\rangle\\
		\left\langle z_I^*, \delta(x - x^*)\right\rangle
	\end{pmatrix},\;\text{in $H^{-1}$ sense,}
\end{equation}
where
\begin{equation*}
	\begin{cases}
		\begin{split}
			T_{\la_R^*}:=&-d\frac{d^2}{dx^2}-\frac{f_u^*g_v^*-f_v^*g_u^*}{f_u^*}+\la_R^*+2 k_2,\\
			M^*:=&\begin{pmatrix}
				\frac{\rho_0^{*}-2 k_1-\tau\la_R^*}{(\rho_0^{*}-2 k_1-\tau\la_R^*)^2+(\tau\la_I^*)^2}&-\frac{\tau\lambda _I^*}{(\rho_0^{*}-2 k_1-\tau\la_R^*)^2+(\tau\la_I^*)^2}\\
				\frac{\tau\lambda _I^*}{(\rho_0^{*}-2 k_1-\tau\la_R^*)^2+(\tau\la_I^*)^2}&\frac{\rho_0^{*}-2 k_1-\tau\la_R^*}{(\rho_0^{*}-2 k_1-\tau\la_R^*)^2+(\tau\la_I^*)^2}
			\end{pmatrix}.
		\end{split}
	\end{cases}
\end{equation*}

If for any given $h\in H^{-1}(I)$, $z$ is uniquely determined by
\begin{equation*}
	\left\langle T_{\la_R}^\ep z,\psi\right\rangle =\left\langle h,\psi\right\rangle \;\;\;\;( \forall \psi\in H_N^1(I)),
\end{equation*}
then we define the generalized inverse operator of $T_{\la_R}^{\ep}$ as mapping $(T_{\la_R}^{\ep})^{-1}: H^{-1}(I)\to H_N^1(I)$ such that $(T_{\la_R}^{\ep})^{-1}h=z$. The other generalized operators in this paper are defined similarly.

By \cite{NF,NM,TMN}, $T_{\la_R^*}$ has a generalized inverse operator $T_{\la_R^*}^{-1}$ when $\la_R^* + 2 k_2 > -\mu^*$. In the following, we always treat the case when $\la_R^* + 2 k_2 > -\mu^*$. In case of $\la_R^* + 2 k_2 \leq -\mu^*$, all the eigenvalues of \eqref{mz^*-1} have negative real parts, which do not affect the stability of the solution in the sense of $\ep\to0^+$.

Applying $T_{\la_R^*}^{-1}$ to system \eqref{mz^*-1} yields
\begin{equation*}
	\begin{pmatrix}
		I&-\lambda _I^*T_{\la_R^*}^{-1}\\
		\lambda _I^*T_{\la_R^*}^{-1}&I
	\end{pmatrix}
	\begin{pmatrix}
		z_R^*\\
		z_I^*
	\end{pmatrix}
	=c_1^*c_2^*T_{\la_R^*}^{-1}\delta(x-x^*)M^*
	\begin{pmatrix}
		{\left\langle {{z_R^*}, \delta(x - x^*)} \right\rangle }\\
		{\left\langle {{z_I^*}, \delta(x - x^*)} \right\rangle }
	\end{pmatrix},\;\text{in $H^{-1}$ sense},
\end{equation*}
or equivalently
\begin{equation}\label{mz^*-1-b}
	\begin{pmatrix}
		z_R^*\\
		z_I^*
	\end{pmatrix}
	= c_1^*c_2^*G
	\begin{pmatrix}
		I&  \lambda _I^*T_{\la_R^*}^{-1}\\
		-\lambda _I^*T_{\la_R^*}^{-1}&I
	\end{pmatrix}
	T_{\la_R^*}^{-1}\delta(x - x^*){M^* }
	\begin{pmatrix}
		{\left\langle {{z_R^*}, \delta(x - x^*)} \right\rangle }\\
		{\left\langle {{z_I^*}, \delta(x - x^*)} \right\rangle }
	\end{pmatrix},\;\text{in $H^{-1}$ sense},
\end{equation}
where $G:=\big[I+(\la_I^*)^2(T_{\la_R^*}^{-1})^2\big]^{-1}$.

Taking the inner product with $\delta(x - x^*)$ in \eqref{mz^*-1-b}, we have
\begin{equation}\label{mz^*-1-d}
	\mathcal N\begin{pmatrix}
		{\left\langle {{z_R^*}, \delta(x - x^*)} \right\rangle }\\
		{\left\langle {{z_I^*}, \delta(x - x^*)} \right\rangle }
	\end{pmatrix}=0,
\end{equation}
where
\begin{equation*}
	\mathcal N:=\left\langle c_1^*c_2^*G\begin{pmatrix}
		I&  \lambda _I^*T_{\la_R^*}^{-1}\\
		-\lambda _I^*T_{\la_R^*}^{-1}&I
	\end{pmatrix}T_{\la_R^*}^{-1} \delta(x - x^*) M^*,\;\delta(x - x^*) \right\rangle-I.
\end{equation*}

Let $\left\{(\gamma_n, \psi_n)\right\}_{n=0}^\infty$ be the complete orthonormal pairs of eigenvalues and eigenfunctions of the operator $-d\frac{d^2}{dx^2}-(f_u^*g_v^*-f_v^*g_u^*)/f_u^*$ subject to the homogeneous Neumann boundary condition in $L^2(0,\ell)$ sense. Moreover, $\gamma_n>0$ holds for any integer $n\ge0$. 

Then, by using lengthy but straightforward eigenfunction expansions, we can reduce $\mathcal N$ to
\begin{equation}\label{mat-N-P}
	\mathcal N=\begin{pmatrix}
		\frac{(\rho_0^{*}-2 k_1-\tau\la_R^*)X+\tau(\la_I^*)^2 Y}{(\rho_0^{*}-2 k_1-\tau\la_R^*)^2+(\tau\la_I^*)^2}-1  & -\lambda _I^*\frac{\tau X-(\rho_0^{*}-2 k_1-\tau\la_R^*)Y}{(\rho_0^{*}-2 k_1-\tau\la_R^*)^2+(\tau\la_I^*)^2}\\
		\lambda _I^*\frac{\tau X-(\rho_0^{*}-2 k_1-\tau\la_R^*)Y}{(\rho_0^{*}-2 k_1-\tau\la_R^*)^2+(\tau\la_I^*)^2} &  \frac{(\rho_0^{*}-2 k_1-\tau\la_R^*)X+\tau(\la_I^*)^2 Y}{(\rho_0^{*}-2 k_1-\tau\la_R^*)^2+(\tau\la_I^*)^2}-1
	\end{pmatrix},
\end{equation}
where
\begin{equation}\label{AB}
	\begin{split}
		X&=X(\la_R^*,(\la_I^*)^2, k_2):=\left\langle {c_1^*c_2^*G{(T_{\la_R^*})^{-1}}\delta(x - x^*),{\delta(x - x^*)}} \right\rangle=\sum\limits_{n \ge 0} \dfrac{(\gamma _n + \lambda _R^* + 2 k_2)c_1^*c_2^*\psi_n ^2(x^*)}{(\gamma _n + \lambda _R^* + 2 k_2)^2 + (\lambda _I^*)^2},\\
		Y&=Y(\la_R^*,(\la_I^*)^2, k_2):=\left\langle {c_1^*c_2^*G{(T_{\la_R^*})^{-2}}\delta(x - x^*),{\delta(x - x^*)}} \right\rangle=\sum\limits_{n \ge 0} \dfrac{c_1^*c_2^*\psi_n ^2(x^*)}{(\gamma _n + \lambda _R^* +2 k_2)^2 + (\lambda _I^*)^2}.
	\end{split}    
\end{equation}

To expect the existence of nonzero solution of \eqref{mz^*-1-d}, one requires $\det (\mathcal N)=0$, or equivalently
\begin{equation}\label{slep-p}
	\begin{cases}
		\begin{split}
			&\dfrac{\rho_0^{*}-2 k_1-\tau\la_R^*}{(\rho_0^{*}-2 k_1-\tau\la_R^*)^2+(\tau\la_I^*)^2}X + \dfrac{\tau {{(\lambda _I^*)}^2}}{(\rho_0^{*}-2 k_1-\tau\la_R^*)^2+(\tau\la_I^*)^2}Y = 1,\\
			&\dfrac{{\tau \lambda _I^*}}{(\rho_0^{*}-2 k_1-\tau\la_R^*)^2+(\tau\la_I^*)^2}X - \dfrac{\lambda _I^*(\rho_0^{*}-2 k_1-\tau\la_R^*)}{(\rho_0^{*}-2 k_1-\tau\la_R^*)^2+(\tau\la_I^*)^2}Y = 0,
		\end{split}
	\end{cases}
\end{equation}
which is called the real singular limit eigenvalue problem (SLEP) of \eqref{P-L-2b}. 

\begin{proposition}\label{prop.AB}
	Let $X(\la_R^*,(\la_I^*)^2, k_2)$ and $Y(\la_R^*,(\la_I^*)^2, k_2)$ be defined in \eqref{AB}. Then, 
	\begin{equation*}
		\begin{split}
			&X(0,0,0)=\sum\limits_{n \ge 0} \dfrac{c_1^*c_2^*}{\gamma _n }\psi_n ^2(x^*)>\rho_0^*,\;\;Y(0,0,0)=\sum\limits_{n \ge 0} \dfrac{c_1^*c_2^*}{\gamma _n^2 }\psi_n ^2(x^*)>0,\\
			&\lim\limits_{k_2\to+\infty}X(\la_R^*,(\la_I^*)^2, k_2)=\lim\limits_{k_2\to+\infty}Y(\la_R^*,(\la_I^*)^2, k_2)=0,\\
			&\lim\limits_{(\la_I^*)^2\to+\infty}X(\la_R^*,(\la_I^*)^2, k_2)=\lim\limits_{(\la_I^*)^2\to+\infty}Y(\la_R^*,(\la_I^*)^2, k_2)=0,\\
			&\ds\frac{\pa X}{\pa (\la_I^*)^2}<0,\;\;\ds\frac{\pa Y}{\pa (\la_I^*)^2}<0,\;\;\frac{1}{2}\ds\frac{\pa X}{\pa k_2}\bigg|_{\la_I^*=0}=\ds\frac{\pa X}{\pa \la_R^*}\bigg|_{\la_I^*=0}<0,\;\;\;\;\frac{1}{2}\ds\frac{\pa Y}{\pa k_2}=\ds\frac{\pa Y}{\pa \la_R^*}<0. 
		\end{split}
	\end{equation*}
	In particular, it holds true that
	\begin{equation*}
			\ds\frac{\pa X}{\pa \la_R^*}(0,0,k_2)=\frac{1}{2}\ds\frac{\pa X}{\pa k_2}(0,0,k_2)=-Y(0,0,k_2)<0.
	\end{equation*}
\end{proposition}
\begin{proof}
	Firstly, $X(0,0,0)>\rho_0^*$ is proved by \cite{NM}. Secondly, the rest of the results can be proved by using the eigenfunction expansions of \eqref{AB}, and we shall omit the proofs.
\end{proof}

Note that to expect zero eigenvalue of \eqref{slep-p}, we need to require that $2 k_1\neq\rho_0^*$. The following results are useful for our later stability analysis for $\epsilon\in(0,\epsilon_0)$.
\begin{proposition}\label{p*ei} Let $c_1^*,c_2^*, \rho_0^{*}$ be defined \eqref{c1c2} and \eqref{App1}, respectively, and let $\left\{(\gamma_n, \psi_n)\right\}_{n=0}^\infty$ be the complete orthonormal pairs of eigenvalues and eigenfunctions of the operator $-d\frac{d^2}{dx^2}-(f_u^*g_v^*-f_v^*g_u^*)/f_u^*$ subject to the homogeneous Neumann boundary condition in the $L^2(0,\ell)$ sense. Suppose that $\tau\in(\tau^*,+\infty)$ holds, where
	\begin{equation}\label{tau*}
		\tau^*:=\sum\limits_{n \ge 0} \dfrac{c_1^*c_2^*}{\gamma _n^2 }\psi_n ^2(x^*)=Y(0,0,0).
	\end{equation}
	Then, the following conclusions hold true:
	\begin{enumerate}
		\item For any $ k_1\in(0,\rho_0^*/2)$, there exists a unique curve $\Gamma=\left\lbrace(k_1,k_2) \big| k_2=\xi(k_1), 0< k_1<\rho_0^*/2 \right\rbrace$
		in the first quadrant of $(k_1, k_2)$ plane, such that for any $(k_1,k_2)\in\Gamma$, the limiting $($$\epsilon\rightarrow0^+$$)$ eigenvalue problem \eqref{mz^*-1} has zero as its eigenvalue, where $k_2=\xi(k_1)$ is a positive, continuous and increasing function of $ k_1$ satisfying 
		\begin{equation}\label{fff1}
			\lim\limits_{ k_1\rightarrow0^+}\xi( k_1)>0,\;\lim\limits_{ k_1\rightarrow (\rho_0^*/2)^-}\xi( k_1)=\infty.
		\end{equation}
		Moreover, for any fixed $ k_1\in(0,\rho_0^*/2)$, all the eigenvalues of the limiting $($$\epsilon\rightarrow0^+$$)$ eigenvalue problem \eqref{mz^*-1} have negative real parts for $ k_2\in(0,\xi( k_1))$, while \eqref{mz^*-1} has a positive eigenvalue for $ k_2\in(\xi( k_1), \infty)$.
		\item For any $ k_1\in(\rho_0^*/2,\infty)$, all the eigenvalues of the limiting $($$\epsilon\rightarrow0^+$$)$ eigenvalue problem \eqref{mz^*-1} have negative real parts regardless of $ k_2>0$.
	\end{enumerate}
\end{proposition}

\begin{proof} To begin with, we note that when $\la_I^*\neq0$, \eqref{slep-p} can be simplified into
	\begin{equation}\label{p-im}
		\tau\la_R^*=\rho_0^*-2 k_1-X(\la_R^*,(\la_I^*)^2, k_2),\;\;Y(\la_R^*,(\la_I^*)^2, k_2)=\tau.
	\end{equation}
	
	While if $\la_I^*=0$, then from the first equation of \eqref{slep-p}, we have
	\begin{equation}\label{p-re}
		\tau\la_R^*=\rho_0^*-2 k_1-X(\la_R^*,0, k_2).
	\end{equation}
	
	We firstly claim that the limiting eigenvalue problem \eqref{mz^*-1} has no complex eigenvalues $\la_R^*\pm i\la_I^*$ satisfying $\la_R^*\ge0$ and $\la_I^*\neq0$. Suppose not. Then, \eqref{p-im} holds. In particular, $Y(\la_R^*,(\la_I^*)^2, k_2)=\tau$. It follows from Proposition \ref{prop.AB} and \eqref{AB} that
	\begin{equation}\label{nexim}
		\tau>\tau^*=\sum\limits_{n \ge 0} \dfrac{c_1^*c_2^*}{\gamma _n^2 }\psi_n ^2(x^*)=Y(0,0, 0)>Y(\la_R^*,(\la_I^*)^2, k_2),
	\end{equation}
	which contradicts with $Y(\la_R^*,(\la_I^*)^2, k_2)=\tau$. Thus, we complete the proof of the claim.
	
	1. Suppose that the limiting eigenvalue problem \eqref{mz^*-1} has zero as its eigenvalue, say $\la_R^*=\la_I^*=0$. By $\la_R^*=\la_I^*=0$, \eqref{p-re} is equivalent to
	\begin{equation}\label{HK11}
		\mathcal H( k_1, k_2):=\rho_0^*-2 k_1-X(0,0, k_2)=0.
	\end{equation}
	
	For any fixed $ k_1\in(0,\rho_0^*/2)$, in view of Proposition  \ref{prop.AB}, we have
	\begin{equation}\label{Fpl-r}
		\lim\limits_{ k_2\to 0^+}\mathcal H( k_1, k_2)=\rho_0^*-2 k_1-X(0,0, 0)<0,\;\;\lim\limits_{ k_2\to +\infty}\mathcal H( k_1, k_2)=\rho_0^*-2 k_1>0.
	\end{equation}	
	
	Moreover, by Proposition \ref{prop.AB}, we have $\frac{\pa \mathcal H}{\pa  k_2}( k_1, k_2)=-\frac{\pa X}{\pa k_2}(0,0, k_2)>0$.
	Then, by \eqref{Fpl-r}, for any fixed $k_1\in(0,\rho_0^*/2)$, there exists a unique $ k_2=\widehat  k_2$ such that $\mathcal H( k_1, \widehat k_2)=0$. Clearly, $\widehat  k_2$ depends on $ k_1$ continuously. Define $\xi( k_1):=\widehat  k_2$. Then, we have
	\begin{equation}\label{H=0}
		\mathcal H( k_1, \xi( k_1))=0.
	\end{equation}
	
	Differentiating \eqref{H=0} with respect to $ k_1$, we have $\frac{\partial\mathcal H}{\partial  k_1}+\frac{\partial\mathcal H}{\partial  k_2}\frac{d\xi( k_1)}{d  k_1}=0$. Since $\frac{\pa \mathcal H}{\pa  k_1}=-2<0$, we have
	\begin{equation}\label{gamK1}
		\frac{d\xi( k_1)}{d  k_1}=-2/\frac{\pa X}{\pa  k_2}(0,0, k_2)>0.
	\end{equation}
	
	Then, $\xi( k_1)$ is an increasing function of $ k_1$. So far, we have proved that if $\la_R^*=\la_I^*=0$, then $(k_1, k_2)\in\Gamma=\left\lbrace( k_1, k_2) \big|  k_2=\xi( k_1), 0< k_1<\rho_0^*/2 \right\rbrace$. On the other hand, it is clear that for any $( k_1, k_2)\in\Gamma$, \eqref{mz^*-1} has zero as its eigenvalue. Moreover, we claim that \eqref{fff1} holds.
	Suppose otherwise that the first term of \eqref{fff1} is false. Then, we must have $\xi( k_1)\rightarrow0^+$ as $ k_1\rightarrow0^+$ since $\xi( k_1)>0$. Then, as $ k_1\rightarrow0^+$, $0=\mathcal H( k_1,\xi( k_1))\rightarrow\mathcal H(0,0)=\rho_0^*-X(0,0,0)<0$, which is a contradiction. Suppose otherwise that the second term of \eqref{fff1} is false. Then, as $ k_1\rightarrow(\rho_0^*/2)^-$, $0=\mathcal H( k_1,\xi( k_1))\rightarrow\mathcal H(\rho_0^*/2,\xi(\rho_0^*/2))=-X(0,0,\xi(\rho_0^*/2))<0$ since $\xi(\rho_0^*/2)<+\infty$ and $X(0,0,\xi(\rho_0^*/2))>0$. Again, it reaches a contradiction. 
	
	Next, we shall show that the transversality condition holds. To that end, we need to resort to the SLEP system of \eqref{P-L-2b} in its complex form, instead of the real SLEP equations \eqref{slep-p}. In the following, we always assume that $k_1\in(0,\rho_0^*/2)$.
	
	Recall that
	\begin{equation*}
		\lim\limits_{\ep\to 0^+}\la=:\la^*=:\la_R^*+i\la_I^*, \;\lim\limits_{\ep\to 0^+}z=:z^*=:z_R^*+iz_I^*.
	\end{equation*}
	
	Define
	\begin{equation*}
		T_{\la^*}:=-d\frac{d^2}{dx^2}-\frac{f_u^*g_v^*-f_v^*g_u^*}{f_u^*}+2 k_2+\la^*.
	\end{equation*}
	
	Then, by Remark 2.3 in \cite{NM}, for any $\la^*$ satisfying $\la_R^*+2 k_2>-\mu^*$, $T_{\la^*}$ has a generalized inverse operator $(T_{\la^*})^{-1}$.
	
	Multiplying the second equation of \eqref{mz^*-1} by $i$ and adding the first equation of \eqref{mz^*-1}, we have
	\begin{equation*}
		T_{\la^*}z^*=\frac{\big\langle z^*,\delta(x - x^*) \big\rangle }{\rho_0^*-2 k_1-\tau\la^*}c_1^*c_2^*\delta(x - x^*),
	\end{equation*}
	or equivalently
	\begin{equation}\label{Tla*}
		z^*=\frac{\big\langle z^*,\delta(x - x^*) \big\rangle }{\rho_0^*-2 k_1-\tau\la^*}(T_{\la^*})^{-1}c_1^*c_2^*\delta(x - x^*).
	\end{equation}
	
	Then, taking the inner product with $\delta(x - x^*)$ on the two sides of \eqref{Tla*}, we have
	\begin{equation}\label{CSLEP}
		\bigg(1-\dfrac{\big\langle (T_{\la}^*)^{-1}c_1^*c_2^*\delta(x - x^*),\delta(x - x^*) \big\rangle}{\rho_0^*-2 k_1-\tau\la^*}\bigg)\big\langle z^*,\delta(x - x^*) \big\rangle=0,
	\end{equation}
	where the eigenfunction expansion shows that
	\begin{equation*}
		\left\langle (T_{\la}^*)^{-1}c_1^*c_2^*\delta(x - x^*),\delta(x - x^*) \right\rangle=\sum_{n=0}^{\infty}\frac{c_1^*c_2^*}{\gamma_n+2 k_2+\la^*}\psi_n^2(x^*).	
	\end{equation*}
	
	Then, \eqref{CSLEP} is equivalent to 
	\begin{equation}\label{CSLEP-2}
		\dfrac{\mathcal{F}^* \left(\la^*, k_1, k_2\right)}{\rho_0^*-2 k_1-\tau\la^*}\cdot\big\langle z^*,\delta(x - x^*) \big\rangle=0,
	\end{equation}	 
	where
	\begin{equation}\label{comFP-3}
		\mathcal{F}^* \left(\la^*, k_1, k_2\right):= \rho_0^*-2 k_1-\tau\la^*-\sum_{n=0}^{\infty}\frac{c_1^*c_2^*}{\gamma_n+2 k_2+\la^*}\psi_n^2(x^*).
	\end{equation}
	
	Thus, to expect nonzero solutions of \eqref{CSLEP-2}, we need $\mathcal{F}^* \left(\la^*, k_1, k_2\right)=0$, which is called the complex SLEP system of \eqref{P-L-2b} and is equivalent to
	\begin{equation}\label{RIF}
		\text{Re}\;\mathcal{F}^*:= \rho_0^*-2 k_1-\tau\la_R^*-X(\la_R^*,(\la_I^*)^2, k_2)=0,\;\;\text{Im}\;\mathcal{F}^* :=\la_I^*(Y(\la_R^*,(\la_I^*)^2,k_2)-\tau)=0,
	\end{equation}
	where $\text{Re}$ and $\text{Im}$ stand for the real part and imaginary part, respectively.
	
	It follows from \eqref{comFP-3} and \eqref{HK11} that 
	\begin{equation}\label{GlK12}
		\begin{split}
			\mathcal F^*(0,  k_1,\xi( k_1))&=\rho_0^{*}-2 k_1-X(0, 0,\xi( k_1))=\mathcal H( k_1,\xi( k_1))=0,\\
			\frac{\pa \mathcal{F}^*}{\pa \la^*}\left(0, k_1,\xi( k_1)\right)&=-\tau+\sum_{n=0}^{\infty}\frac{c_1^*c_2^*}{(\gamma_n+2 k_2)^2}\psi_n^2(x^*)=-\tau+Y(0,0, k_2)<0.
		\end{split}
	\end{equation}
	
	Thus, by the implicit function theorem, for any $( k_1, k_2)\in\mathcal O$, the sufficiently small neighborhood of $\Gamma$, there exists a unique $\la^*=\la^*( k_1, k_2)$ such that 
	\begin{equation*}
		\mathcal F^*(\la^*( k_1, k_2),  k_1, k_2)=0,\;\la^*(k_1,\xi(k_1))=0.
	\end{equation*}
	
	We claim that $\la^*=\la^*( k_1, k_2)$ is real. Suppose not. Then, $\la_I^*\neq0$. Thus, from the second equation of \eqref{RIF}, we have $Y(\la_R^*,(\la_I^*)^2,k_2)-\tau=0$. However, this is impossible since in the small neighborhood of $\Gamma$, we have $Y(\la_R^*,(\la_I^*)^2,k_2)-\tau<0$ in that $Y(0,0, k_2)-\tau<0$.
	
	Then, we have $\la^*=\la_R^*$ and $\la_I^*=0$. Therefore, by \eqref{comFP-3} and $\la^*=\la_R^*$, we have
	\begin{equation}\label{Flak12}
		\mathcal{F}^*(\la^*( k_1, k_2),k_1, k_2)=\rho_0^*-2 k_1-\tau\la^*( k_1, k_2)-X(\la^*( k_1, k_2),0, k_2)=0.
	\end{equation}
	
	For any fixed $k_1\in(0,\rho_0^*/2)$, differentiating \eqref{Flak12} and setting $k_2=\xi(k_1)$, one obtains
	\begin{equation}\label{la*K2}
		\dfrac{\pa \la^*}{\pa k_2} (k_1,\xi(k_1))=-\frac{\frac{\pa X}{\pa k_2}(0,0, k_2)}{\tau+\frac{\pa X}{\pa \la_R^*}(0,0, k_2)}=-\frac{\frac{\pa X}{\pa  k_2}(0,0, k_2)}{\tau-Y(0,0, k_2)}>0.
	\end{equation}
	
	Thus, for any fixed $k_1\in(0,\rho_0^*/2)$, whenever $ k_2$ is sufficiently close to $\xi( k_1)$, $\la^*=\la^*( k_1, k_2)$ exists and is real. Clearly, $\la^*( k_1, k_2)$ persists for any larger $ k_2$ (say, $ k_2>\xi( k_1)$) and is always positive, since \eqref{mz^*-1} has no complex eigenvalues $\la^*=\la_R^*\pm i\la_I^*$ satisfying $\la_R^*\ge0$ and $\la_I^*\neq0$. On the other hand, $\la^*( k_1, k_2)$ persists for any smaller $ k_2$ (say, $ k_2<\xi( k_1)$) and it will always have a negative real part. Suppose not. Then, there exists a critical value in $(0,\xi( k_1))$, denoted by $ \widetilde{k}_2^*$, of $ k_2$ such that at $ k_2=\widetilde{k}_2^*$, the real part of $\la^*( k_1, k_2)$ is zero, however, this is impossible since for one hand, \eqref{mz^*-1} has no complex eigenvalues $\la^*=\la_R^*\pm i\la_I^*$ satisfying $\la_R^*\ge0$ and $\la_I^*\neq0$, thus eigenvalue with zero real part can never be purely imaginary number; on the other hand, this eigenvalue with zero real part can never be zero since only when $( k_1, k_2)\in\Gamma$, can the eigenvalue be zero.
	
	Thus, we can conclude that for any fixed $k_1\in(0,\rho_0^*/2)$, all the eigenvalues of \eqref{mz^*-1} have negative real parts for $ k_2\in(0,\xi( k_1))$, while \eqref{mz^*-1} has a positive eigenvalue for $ k_2\in(\xi( k_1), \infty)$. This proves 1.
	
	2. For any fixed $ k_1\in(\rho_0^*/2,\infty)$, regardless of $ k_2>0$, we have $\mathcal{H}( k_1, k_2)<0$. Thus, the eigenvalue problem \eqref{mz^*-1} has no zero eigenvalues. Note that $\Gamma$ is the boundary of the stable region and unstable region. Since $\{( k_1, k_2):  k_1\in(\rho_0^*/2,\infty),  k_2\in(0, +\infty)\}$ is in the right side of $\Gamma$, and the stable region is in the left side of $\Gamma$, it follows that all the eigenvalues of \eqref{mz^*-1} have negative real parts. This thus proves 2.
\end{proof}
		
	Finally, we have the following results describing the stability of $(\widetilde{u}(x;\ep),\widetilde{v}(x;\ep),\widetilde{u}(x;\ep),\widetilde{v}(x;\ep))$ with respect to \eqref{P-1} for small $\ep>0$.
		
				\begin{figure}[!t]
			\centering
			\begin{tikzpicture}
				\fill [fill=green!20]  (0,4) --  (0,1) parabola (2.4,4)--cycle;
				\fill [fill=blue!20]  (0,0) --  (0,1) parabola (2.4,4)--(2.5,4)--(2.5,0)--cycle;
				\fill [fill=magenta!20]  (2.5,0) rectangle (7.5,4);
				\draw[->] (0,0) --(7.55,0) node[right] {$k_1$};
				\draw[->] (0,0) --(0,4.05) node[above] {$k_2$};
				\draw  (0,0) node[left] {$O$};
				\draw (2.5,4)--(2.5,0) node [below]{$\rho_0^*/2$};
				\draw (0,1) parabola (2.4,4);
				\draw [densely dashed] (1.2,1.75)--(1.2,0) node [below] {$k_1$};
				\draw [densely dashed] (1.2,1.75)--(0,1.75) node [left] {$k_2^\ep$};
				\draw node at (1.8,1.2){$\Gamma_2^\ep$};
				\draw node at (1.1,2.9){$\Gamma_1^\ep$};
				\draw node at (4.8,2.2){$\Gamma_3^\ep$};
			\end{tikzpicture}
			\caption{Stability diagram of $(\widetilde{u}(x;\ep),\widetilde{v}(x;\ep),\widetilde{u}(x;\ep),\widetilde{v}(x;\ep))$ with respect to \eqref{P-1} at any fixed $\ep\in(0,\ep_0)$, where $\Gamma_{1}^\ep=(0,\rho_0^*/2)\times(k_2^\ep,\infty)$, $\Gamma_{2}^\ep=(0,\rho_0^*/2)\times(0,k_2^\ep)$, and $\Gamma_{3}^\ep=(\rho_0^*/2,\infty)\times(0,\infty)$. In $\Gamma_1^\ep$, $(\widetilde{u}(x;\ep),\widetilde{v}(x;\ep),\widetilde{u}(x;\ep),\widetilde{v}(x;\ep))$ is Turing unstable with respect to \eqref{P-1} driven by suitable $k_1$ and $k_2$. In $\Gamma_2^\ep\cup\Gamma_3^\ep$, $(\widetilde{u}(x;\ep),\widetilde{v}(x;\ep),\widetilde{u}(x;\ep),\widetilde{v}(x;\ep))$ is stable with respect to \eqref{P-1}.}
			\label{Fig.gamep}
		\end{figure}
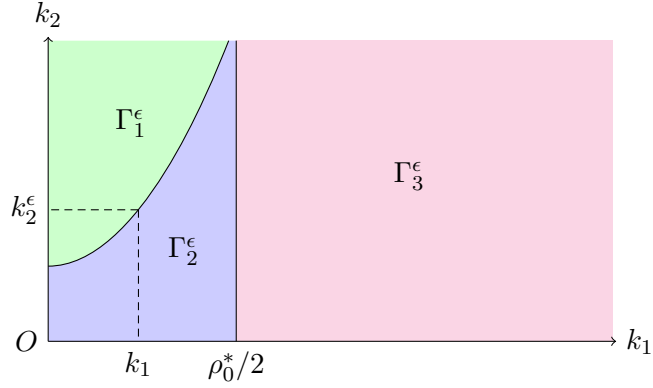
		
		\begin{theorem}\label{th.2-1}
			Suppose that $\tau\in(\tau^*,+\infty)$ holds. Let $\tau^*$ and $\rho_0^*$ be precisely defined in \eqref{nexim} and \eqref{App1}, respectively. Then, for any fixed $\ep\in(0,\ep_0)$, the following conclusions hold true:
			\begin{enumerate}
				\item If $k_1>\rho_0^*/2$, then the symmetric solution $(\widetilde{u}(x;\ep),\widetilde{v}(x;\ep),\widetilde{u}(x;\ep),\widetilde{v}(x;\ep))$ is stable with respect to system \eqref{P-1} for any $k_2>0$ $($see Fig. \ref{Fig.gamep}$)$.
				\item  If $k_1<\rho_0^*/2$, then there is a critical value $k_2^\ep>0$ such that the symmetric solution is stable with respect to system \eqref{P-1} for $k_2<k_2^\ep$, while unstable if $k_2>k_2^\ep$. 
				At $k_2=k_2^\ep$, the Turing-type steady state bifurcation occurs near the solution $(\widetilde{u}(x;\ep),\widetilde{v}(x;\ep),\widetilde{u}(x;\ep),\widetilde{v}(x;\ep))$  $($see Fig. \ref{Fig.gamep}$)$. In particular, $k_2^\ep$ is an increasing function of $k_1$, denoted by $k_2^\ep=\widehat \xi^\ep(k_1)$, satisfying $\lim\limits_{\ep\rightarrow0^+}\widehat \xi^\ep(k_1)=\xi( k_1)\in(0,\infty)$.
			\end{enumerate}		
		\end{theorem}
		\begin{proof}
			We shall derive the complex characteristic equation of \eqref{mz-1} firstly.
			
			Solving the first equation of \eqref{P-L-2b} after omitting subscription \lq\lq a\rq\rq yields
			\begin{equation}\label{zhy}
				w=-(\mathcal{L}_1^{\ep}-\ep\tau\la)^{-1}(f_v^\ep z)=-\frac{\left\langle f_v^{\ep}z,\phi_0^{\ep}\right\rangle }{\mu_0^{\ep}-2\ep k_1-\ep\tau\la}\phi_0^{\ep}-(\mathcal{L}_1^{\ep}-\ep\tau\la)^{\dag}(f_v^{\ep}z).
			\end{equation}
			
			We define
			\begin{equation*}
				T_\la^{\ep}:=-d\frac{d^2}{dx^2}-g_v^{\ep}+2k_2+g_u^{\ep}(\mathcal{L}_1^{\ep}-\ep\tau\la)^{\dag}f_v^{\ep}+\la.
			\end{equation*}
			
			Substituting \eqref{zhy} into the second equation of \eqref{P-L-2b}, we have
			\begin{equation}\label{Tlaepz}
				T_\la^{\ep}z=-\frac{\left\langle f_v^{\ep}z,\phi_0^{\ep}\right\rangle }{\mu_0^{\ep}-2\ep k_1-\ep\tau\la}g_u^{\ep}\phi_0^{\ep}=-\frac{\left\langle z,\frac{f_v^{\ep}\phi_0^{\ep}}{\sqrt{\ep}}\right\rangle }{\rho(\ep)-2 k_1-\tau\la}\frac{g_u^{\ep}\phi_0^{\ep}}{\sqrt{\ep}}.
			\end{equation}
			
			By using the methods in Lemma 4.1 of \cite{NM}, one can check that, for any $\la$ satisfying $\text{Re}\;\la+2k_2>-\mu^*$, $T_\la^{\ep}$ has an inverse operator $(T_\la^{\ep})^{-1}$.
			
			Applying $(T_\la^{\ep})^{-1}$ into \eqref{Tlaepz} and taking the inner product with $\frac{f_v^{\ep}\phi_0^{\ep}}{\sqrt{\ep}}$ leads to
			\begin{equation}\label{lcj}
				\left( 1+\frac{\left\langle (T_\la^{\ep})^{-1}\frac{g_u^{\ep}\phi_0^{\ep}}{\sqrt{\ep}},\frac{f_v^{\ep}\phi_0^{\ep}}{\sqrt{\ep}}\right\rangle}{\rho(\ep)-2 k_1-\tau\la}\right) \left\langle z,\frac{f_v^{\ep}\phi_0^{\ep}}{\sqrt{\ep}} \right\rangle =:\dfrac{\mathcal{F}\left(\la,k_1,k_2,\ep\right)}{\rho(\ep)-2 k_1-\tau\la}\left\langle z,\frac{f_v^{\ep}\phi_0^{\ep}}{\sqrt{\ep}} \right\rangle=0.
			\end{equation}
			
			To expect the nonzero $\left\langle z,\frac{f_v^{\ep}\phi_0^{\ep}}{\sqrt{\ep}} \right\rangle$ and zero eigenvalue in \eqref{lcj}, we need to require $k_1\neq \rho_0^*/2$ (for $\ep\in(0,\ep_0)$) and
			\begin{equation}\label{FPep}
				\mathcal{F}\left(\la,k_1,k_2,\ep\right):= \rho(\ep)-2 k_1-\tau\la+\left\langle (T_\la^{\ep})^{-1}\frac{g_u^{\ep}\phi_0^{\ep}}{\sqrt{\ep}},\frac{f_v^{\ep}\phi_0^{\ep}}{\sqrt{\ep}} \right\rangle=0.
			\end{equation}

			As $\ep\to0^+$, we have
			\begin{equation}\label{FtoF*}
				\mathcal{F} \big(\la,k_1,k_2,\ep\big)\to \mathcal F^*(\la,k_1, k_2),
			\end{equation}
			where $\mathcal F^*(\la,k_1, k_2)$ is given by \eqref{comFP-3}.
			
			1. Suppose that $k_1>\rho_0^*/2$, then by Proposition \ref{p*ei}, all the eigenvalues of the limiting eigenvalue problem \eqref{mz^*-1} have negative real parts. Thus, for sufficiently small $\epsilon$, say for $0<\ep<\ep_1$, all the eigenvalues of the eigenvalue problem \eqref{P-L-2b} have negative real parts owing to the continuity of eigenvalues with respect to $\ep$. Setting $\widehat\ep_0=\min\left\lbrace \ep_0,\ep_1\right\rbrace $, and dropping the hat of $\widehat\ep_0$. This proves 1 for any fixed $\ep\in(0,\ep_0)$.
			
			2. Suppose that $2 k_1<\rho_0^*$. By \eqref{GlK12} and \eqref{FtoF*}, we have
			\begin{equation*}
				\mathcal{F}(0,k_1,\xi(k_1),0)= \mathcal F^*(0,k_1,\xi(k_1))=0,\;\;\frac{\pa \mathcal{F}}{\pa \la}(0,k_1,\xi(k_1),0)=\frac{\pa \mathcal{F}^*}{\pa \la}(0,k_1,\xi(k_1))<0.
			\end{equation*}
			
			Then, by the implicit function theorem, for sufficiently small $\epsilon$, say $\ep\in(0,\ep_2)$ for some $\epsilon_2>0$, there exists a $\la=\la\left(k_1,k_2,\ep\right)$, such that 
			\begin{equation}\label{lakke}
				\mathcal{F} \left(\la\big(k_1,k_2,\ep\big),k_1,k_2,\ep\right)=0,
			\end{equation}
			and that as $\ep\to0^+$, \begin{equation}\label{latola*}
				\la\big(k_1,k_2,\ep\big)\to\la^*( k_1, k_2).
			\end{equation}
			
			In particular,
			\begin{equation}\label{0=ImF}
				0=\text{Im}\;\mathcal{F} \left(\la\big(k_1,k_2,\ep\big),k_1,k_2,\ep\right)=-\tau\text{Im}\;\la\left(k_1,k_2,\ep\right)-\text{Im}\left\langle (T_\la^{\ep})^{-1}\frac{g_u^{\ep}\phi_0^{\ep}}{\sqrt{\ep}},\frac{f_v^{\ep}\phi_0^{\ep}}{\sqrt{\ep}} \right\rangle.
			\end{equation}
			
			We claim that $\la\left(k_1,k_2,\ep\right)$ is real. Suppose not. Then, $\text{Im}\;\la\left(k_1,k_2,\ep\right)\neq0$. Then, by \eqref{0=ImF}, we have
			\begin{equation}\label{0=-t-}
				0=-\tau-\frac{\text{Im}\left\langle (T_\la^{\ep})^{-1}g_u^{\ep}\phi_0^\ep/\sqrt\ep,f_v^{\ep}\phi_0^{\ep}/\sqrt{\ep} \right\rangle}{\text{Im}\;\la\left(k_1,k_2,\ep\right)}.
			\end{equation}
			
			Then, as $\ep\to0^+$, by \eqref{0=-t-}, \eqref{FtoF*} and \eqref{RIF}, we have $0=-\tau+Y(\la^*,0, k_2)$. However, since $\tau>\tau^*$, we have $-\tau+Y(\la^*,0, k_2)<0$, which is a contradiction.
			
			From \eqref{la*K2} and \eqref{latola*}, we can obtain
			\begin{equation}\label{palak2}
				\la(k_1,\xi(k_1),0)=0,\;\;\frac{\pa\la(k_1,\xi(k_1),0)}{\pa k_2}=\frac{\pa\la^*(k_1,\xi(k_1))}{\pa k_2}>0.
			\end{equation}
			
			Then, by the implicit function theorem, we can obtain the existence of $k_2=\widehat\xi\left(k_1,\ep\right)$ for small $\ep$, satisfying
			\begin{equation}\label{laex}
				\la\left( k_1,\widehat\xi\left( k_1,\ep\right),\ep\right)=0,\;\;\widehat\xi\left( k_1,\ep\right)\to\xi( k_1)\;\;(\text{as}\;\;\ep\to0^+).
			\end{equation}
			
			Thus, $\frac{\pa \widehat\xi( k_1,\ep)}{\pa  k_1}>0$ for $\ep\in(0,\ep_3)$ due to $\frac{\pa \xi( k_1)}{\pa  k_1}>0$ (given by \eqref{gamK1}).
			
			It follows from \eqref{palak2} and \eqref{laex} that the real eigenvalue $\la\left(k_1,k_2,\ep\right)$ changes from negative to positive as $k_2$ increases around $\widehat\xi\left( k_1,\ep\right)$. Define 
			\begin{equation*}
				\widehat\xi^\ep\left( k_1\right):=\widehat\xi\left( k_1,\ep\right).
			\end{equation*}
			
			Setting $\widehat\ep_0=\min\left\lbrace \ep_0,\ep_2,\ep_3\right\rbrace $, and dropping the hat of $\widehat\ep_0$. Then for any fixed $\ep\in(0,\ep_0)$, we can conclude 
			\begin{equation*}
				\frac{\pa \widehat\xi^\ep(k_1)}{\pa k_1}=\frac{\pa \widehat\xi\left( k_1,\ep\right)}{\pa k_1}>0.
			\end{equation*}
			
			Thus, $\widehat\xi^\ep(k_1)$ is increasing in $k_1$ for any fixed $\ep\in(0,\ep_0)$, such that the symmetric solution $(\widetilde{u}(x;\ep),\widetilde{v}(x;\ep),\widetilde{u}(x;\ep),\widetilde{v}(x;\ep))$ is stable with respect to system \eqref{P-1} when $k_2<k_2^\ep:=\widehat\xi^\ep(k_1)$, while unstable when $k_2>k_2^\ep$.
			
			Finally, by applying the similar argument in \cite{NM,TMN}, we can derive the simplicity of the critical real eigenvalue $\la\left(k_1,k_2,\ep\right)$, which ends up the proof of 2 for any fixed $\ep\in(0,\ep_0)$.
		\end{proof}

		\section{Dynamics of the coupled system \eqref{C-LE-D}: case of $0<\al<\infty$}
		In this section, we shall consider system \eqref{C-LE-D} in case of $0<\al<\infty$. For $i=1,2$, we define
		\begin{equation*}
			U_i(x,t)=\int_{-\infty}^{t} \al e^{-\al(t-s)}u_i(x,s)ds.
		\end{equation*}
		
		Then, system \eqref{C-LE-D} can be reduced to 
		\begin{equation}\label{C-LE-D-E}
			\begin{cases}
				\begin{split}
					\tau\dfrac{\partial u_i}{\partial t}&=\epsilon\dfrac{\partial^2 u_i}{\partial x^2}+\dfrac{1}{\sigma\epsilon}\bigg(a-u_i-\dfrac{4u_iv_i}{1+u_i^2}\bigg)+k_1(U_j-u_i),~&x\in(0,\ell),\;t>0,\\
					\dfrac{\partial v_i}{\partial t}&=d\dfrac{\partial^2 v_i}{\partial x^2}+u_i-\dfrac{u_iv_i}{1+u_i^2}+k_2(v_j-v_i),~&x\in(0,\ell),\;t>0,\\
					\frac{\pa U_j}{\pa t}&=\alpha u_j-\alpha U_j,~&x\in(0,\ell),\;t>0,\\
					\dfrac{\partial u_i}{\partial x}&=\dfrac{\partial v_i}{\partial x}=0,~&x=0,\ell,\;\;\;t>0,
				\end{split}
			\end{cases}
		\end{equation}
		where $i,j\in\{1,2\}$ and $i\neq j$. Clearly, the fact that $(\widetilde{u}(x;\epsilon),\widetilde{v}(x;\epsilon),\widetilde{u}(x;\epsilon),\widetilde{v}(x;\epsilon))$ is the steady state solution of system \eqref{C-LE-D} is equivalent to the fact that 
		\begin{equation*}
			(u_1,v_1,u_2,v_2,U_1,U_2)=(\widetilde{u}(x;\epsilon),\widetilde{v}(x;\epsilon),\widetilde{u}(x;\epsilon),\widetilde{v}(x;\epsilon), \widetilde{u}(x;\epsilon),\widetilde{u}(x;\epsilon))
		\end{equation*}
		is the steady state solution of system \eqref{C-LE-D-E}.
		
		The eigenvalue problem associated with the linearized system of \eqref{C-LE-D-E} around $(\widetilde{u},\widetilde{v},\widetilde{u},\widetilde{v}, \widetilde{u},\widetilde{u})$ is determined by
		\begin{equation}\label{D-L-1}
			\begin{cases}
				\begin{split}
					&\ep^2 w_1''(x)+f_u^{\ep} w_1+f_v^{\ep} z_1+\ep k_1 (W_2-w_1)=\ep\tau\la w_1,\\
					&d z_1''(x)+g_u^{\ep} w_1+g_v^{\ep} z_1+k_2(z_2-z_1)=\la z_1,\\
					&\ep^2 w_2''(x)+f_u^{\ep} w_2+f_v^{\ep} z_2+\ep k_1 (W_1-w_2)=\ep\tau\la w_2,\\
					&d z_2''(x)+g_u^{\ep} w_2+g_v^{\ep} z_2+k_2(z_1-z_2)=\la z_2,\\
					&\alpha w_1-\alpha W_1=\la W_1, \alpha w_2-\alpha W_2=\la W_2,
				\end{split}
			\end{cases}
		\end{equation}
		subject to homogeneous Neumann boundary conditions for $w_1$, $z_1$, $w_2$ and $z_2$.

		By the last two equations of \eqref{D-L-1}, for $i=1,2$, we have
		\begin{equation}\label{W w Z z}
			W_i=\theta_{\alpha}w_i:=\dfrac{\alpha}{\alpha+\la}w_i.
		\end{equation}
		
		Substituting \eqref{W w Z z} into \eqref{D-L-1}, we have
		\begin{equation}\label{D-L-2-a}
			\begin{cases}
				\begin{split}
					&\ep^2 w_1''(x)+f_u^{\ep} w_1+f_v^{\ep} z_1+\ep k_1 (\theta_{\alpha}w_2-w_1)=\ep\tau\la w_1,\\
					&d z_1''(x)+g_u^{\ep} w_1+g_v^{\ep} z_1+k_2(z_2-z_1)=\la z_1,\\
					&\ep^2 w_2''(x)+f_u^{\ep} w_2+f_v^{\ep} z_2+\ep k_1 (\theta_{\alpha}w_1-w_2)=\ep\tau\la w_2,\\
					&d z_2''(x)+g_u^{\ep} w_2+g_v^{\ep} z_2+k_2(z_1-z_2)=\la z_2,
				\end{split}
			\end{cases}
		\end{equation}
		subject to homogeneous Neumann boundary conditions.
		
		By virtue of transformation \eqref{nform}, one can reduce \eqref{D-L-2-a} to the following two eigenvalue problems
		\begin{equation}\label{EP-1}
			\begin{cases}
				\begin{split}
					&\ep^2 w_s''(x)+(f_u^{\ep}+\ep k_1(\theta_{\alpha}-1)) w_s+f_v^{\ep} z_s=\ep\tau\la w_s,~&x\in(0,\ell),\\
					&d z_s''(x)+g_u^{\ep} w_s+g_v^{\ep} z_s=\la z_s,~&x\in(0,\ell),\\
					&w_s'(0)=w_s'(\ell)=z_s'(0)=z_s'(\ell)=0,
				\end{split}
			\end{cases}
		\end{equation}
		and
		\begin{equation}\label{EP-2}
			\begin{cases}
				\begin{split}
					&\ep^2 w_a''(x)+(f_u^{\ep}-\ep k_1(\theta_{\alpha}+1)) w_a+f_v^{\ep} z_a=\ep\tau\la w_a,~&x\in(0,\ell),\\
					&d z_a''(x)+g_u^{\ep} w_a+(g_v^{\ep}-2k_2) z_a=\la z_a,~&x\in(0,\ell),\\
					&w_a'(0)=w_a'(\ell)=z_a'(0)=z_a'(\ell)=0.
				\end{split}
			\end{cases}
		\end{equation}
		
		Hereafter, the subscriptions of $w_a$, $z_s$, $w_a$ and $w_a$ are ignored if there is no confusion.
		
		\subsection{Analysis of the eigenvalue problem \eqref{EP-1}}\label{S3.2.1}
		
		Rewriting $\theta_\al$ as $\theta_\al=:\theta_{\alpha,R}+i\theta_{\alpha,I}$ and $\la=:\la_R+i\la_I$, we have
		\begin{equation*}
			\theta_{\alpha,R}:=\dfrac{\alpha(\alpha+\la_R)}{(\alpha+\la_R)^2+\la_I^2}, \;\;\theta_{\alpha,I}:=-\dfrac{\alpha \la_I}{(\alpha+\la_R)^2+\la_I^2}.
		\end{equation*}
		
		Then, combining with $w=:w_R+iw_I,\;z=:z_R+iz_I$, by \eqref{EP-1}, we have
		\begin{equation}\label{re-im EP-1}
			\begin{cases}
				\begin{split}
					&(\ep^2\frac{d^2}{dx^2}+f_u^{\ep}+\ep k_1(\theta_{\alpha,R}-1)) w_R-\ep k_1 \theta_{\alpha,I} w_I+f_v^{\ep} z_R=\ep\tau(\la_R w_R-\la_I w_I),\\
					&(\ep^2\frac{d^2}{dx^2}+f_u^{\ep}+\ep k_1(\theta_{\alpha,R}-1)) w_I+\ep k_1 \theta_{\alpha,I} w_R+f_v^{\ep} z_I=\ep\tau(\la_R w_I+\la_I w_R),\\
					&d\frac{d^2}{dx^2} z_R+g_u^{\ep} w_R+g_v^{\ep} z_R=\la_R z_R-\la_I z_I,\\
					&d\frac{d^2}{dx^2} z_I+g_u^{\ep} w_I+g_v^{\ep} z_I=\la_R z_I+\la_I z_R.
				\end{split}
			\end{cases}
		\end{equation}
		
		Define
		\begin{equation*}
			\mathcal{L}_{\alpha}^{\ep}:=\ep^2\dfrac{d^2}{dx^2}+f_u^{\ep}+\ep k_1(\theta_{\alpha,R}-1)-\ep\tau\la_R.
		\end{equation*}
		
		Clearly, $\mathcal{L}_{\alpha}^{\ep}$ has complete and orthonormal pairs of eigenvalues and eigenfunctions
		\begin{equation*}
			\left\lbrace (\mu_n^{\ep}+\ep k_1(\theta_{\alpha,R}-1)-\ep\tau\la_R,\phi_n^{\ep})\right\rbrace _{n=0}^{\infty}
		\end{equation*}
		with respect to homogeneous Neumann boundary condition in $L^2(0,\ell)$ sense.
		
		Solving the first two equations of \eqref{re-im EP-1} with respect to $w_R$ and $w_I$, we have
		\begin{equation*}
			\begin{split}
				w_R&=\left[ I+(\ep\tau\la_I-\ep k_1\theta_{\alpha,I})^2(\mathcal{L}_{\alpha}^{\ep})^{-2}\right] ^{-1}\left[(\ep\tau\la_I-\ep k_1\theta_{\alpha,I})(\mathcal{L}_{\alpha}^{\ep})^{-2}(f_v^{\ep}z_I)-(\mathcal{L}_{\alpha}^{\ep})^{-1}(f_v^{\ep}z_R)\right],\\
				w_I&=-\left[ I+(\ep\tau\la_I-\ep k_1\theta_{\alpha,I})^2(\mathcal{L}_{\alpha}^{\ep})^{-2}\right] ^{-1}\left[(\mathcal{L}_{\alpha}^{\ep})^{-1}(f_v^{\ep}z_I)+(\ep\tau\la_I-\ep k_1\theta_{\alpha,I})(\mathcal{L}_{\alpha}^{\ep})^{-2}(f_v^{\ep}z_R)\right],
			\end{split}
		\end{equation*}
		where the invertibility of $\mathcal{L}_{\alpha}^{\ep}$ for small $\ep>0$ and $\text{Re}\la>-\mu^*$ can be shown similarly by Lemma 2.1 of \cite{NF} and Remark 2.2 of \cite{NM}.
		
		By using the eigenfunction expansions, we have
		\begin{equation}\label{WRIex-2}
			w_R=\sum\limits_{n \ge 0}w_R^n\phi_n^{\ep}=w_R^0\phi_0^{\ep}+w_R^\dag,\;\;
			w_I=\sum\limits_{n \ge 0}w_I^n\phi_n^{\ep}=w_I^0\phi_0^{\ep}+w_I^\dag,
		\end{equation}
		where $w_R^\dag=\sum\limits_{n \ge 1}w_R^n\phi_n^{\ep}$, $w_I^\dag=\sum\limits_{n \ge 1}w_I^n\phi_n^{\ep}$, and for $n\ge0$, it holds true that
		\begin{equation*}
			\begin{split}
				w_R^n&=\dfrac{(-\mu_n^{\ep}+\ep k_1(1-\theta_{\alpha,R})+\ep\tau\la_R)\left\langle {  f_v^{\ep} {z_R},\phi _n^\varepsilon } \right\rangle+(\ep\tau\la_I-\ep k_1\theta_{\alpha,I})\left\langle {  f_v^{\ep} {z_I},\phi _n^\varepsilon } \right\rangle }{(\mu_n^{\ep}+\ep k_1(\theta_{\alpha,R}-1)-\ep\tau\la_R)^2+(\ep\tau\la_I-\ep k_1\theta_{\alpha,I})^2},\\
				w_I^n&=\dfrac{(-\mu_n^{\ep}+\ep k_1(1-\theta_{\alpha,R})+\ep\tau\la_R)\left\langle {  f_v^{\ep} {z_I},\phi _n^\varepsilon } \right\rangle-(\ep\tau\la_I-\ep k_1\theta_{\alpha,I})\left\langle {  f_v^{\ep} {z_R},\phi _n^\varepsilon } \right\rangle }{(\mu_n^{\ep}+\ep k_1(\theta_{\alpha,R}-1)-\ep\tau\la_R)^2+(\ep\tau\la_I-\ep k_1\theta_{\alpha,I})^2}.
			\end{split}	
		\end{equation*}
		
		Similar to \eqref{WRIex-1}, we have
		\begin{equation*}
			\begin{split}
				\begin{split}
					w_R^{\dag}&=\left[ I+(\ep\tau\la_I-\ep k_1\theta_{\alpha,I})^2(\mathcal{L}_{\alpha}^{\ep})^{-2}\right] ^{\dag}\left[(\ep\tau\la_I-\ep k_1\theta_{\alpha,I})(\mathcal{L}_{\alpha}^{\ep})^{2\dag}(f_v^{\ep}z_I)-(\mathcal{L}_{\alpha}^{\ep})^{\dag}(f_v^{\ep}z_R)\right],\\
					w_I^{\dag}&=-\left[ I+(\ep\tau\la_I-\ep k_1\theta_{\alpha,I})^2(\mathcal{L}_{\alpha}^{\ep})^{-2}\right] ^{\dag}\left[(\mathcal{L}_{\alpha}^{\ep})^{\dag}(f_v^{\ep}z_I)+(\ep\tau\la_I-\ep k_1\theta_{\alpha,I})(\mathcal{L}_{\alpha}^{\ep})^{2\dag}(f_v^{\ep}z_R)\right],
				\end{split}
			\end{split}
		\end{equation*}
		where 
		\begin{equation*}
			(\mathcal{L}_{\alpha}^{\ep})^{\dag}=\sum\limits_{n \ge 1}\dfrac{\left\langle \cdot,\phi_n^{\ep}\right\rangle }{\mu_n^{\ep}-\ep\tau\la_R+\ep k_1(\theta_{\alpha,R}-1)}\phi_n^{\ep},\;\;(\mathcal{L}_{\alpha}^{\ep})^{2\dag}=\sum\limits_{n \ge 1}\dfrac{\left\langle \cdot,\phi_n^{\ep}\right\rangle }{(\mu_n^{\ep}-\ep\tau\la_R+\ep k_1(\theta_{\alpha,R}-1)) ^2}\phi_n^{\ep},
		\end{equation*}
		and
		\begin{equation*}
			\left[ I+(\ep\tau\la_I-\ep k_1\theta_{\alpha,I})^2(\mathcal{L}_{\alpha}^{\ep})^{-2}\right] ^{\dag}=\sum\limits_{n \ge 1}\frac{\left(\mu_n^{\ep}-\ep\tau\la_R+\ep k_1(\theta_{\alpha,R}-1)\right)^2 \left\langle \cdot,\phi_n^{\ep}\right\rangle }{\left(\mu_n^{\ep}-\ep\tau\la_R+\ep k_1(\theta_{\alpha,R}-1)\right)^2+\left(\ep\tau\la_I-\ep k_1\theta_{\alpha,I}\right)^2}\phi_n^{\ep}.
		\end{equation*}

		Substituting \eqref{WRIex-2} into the last two equations of \eqref{re-im EP-1} yields
		\begin{equation}\label{mz-2}
			\begin{pmatrix}
				\widehat{T}_{\la_R}^\varepsilon &- \lambda _I + \widehat{S}^{\ep}\\
				\lambda _I - \widehat{S}^{\ep}&\widehat{T}_{\la_R}^\varepsilon 
			\end{pmatrix}
			\begin{pmatrix}
				z_R\\
				z_I
			\end{pmatrix}
			= 
			-\frac{g_u^{\ep}\phi _0^\ep }{\sqrt \ep  }\widehat{M}^\varepsilon 
			\begin{pmatrix}
				\left\langle z_R,   \frac{f_v^{\ep}\phi _0^\ep}{\sqrt \ep} \right\rangle \\
				\left\langle z_I,   \frac{f_v^{\ep}\phi _0^\ep}{\sqrt \ep} \right\rangle
			\end{pmatrix},
		\end{equation}
		where
		\begin{equation*}
			\begin{cases}
				\begin{aligned}
					\widehat{T}_{\la_R}^{\ep}:=&-d\frac{d^2}{dx^2}-g_v^{\ep}+g_u^{\ep}\big\lbrace I+(\ep\tau\la_I-\ep k_1\theta_{\alpha,I})^2(\mathcal{L}_{\alpha}^{\ep})^{-2}\big\rbrace ^{\dag}(\mathcal{L}_{\alpha}^{\ep})^{\dag}f_v^{\ep} +\la_R,\\
					\widehat{S}^{\ep}:=&-(\ep\tau\la_I-\ep k_1\theta_{\alpha,I})g_u^{\ep}\big\lbrace I+(\ep\tau\la_I-\ep k_1\theta_{\alpha,I})^2(\mathcal{L}_{\alpha}^{\ep})^{-2}\big\rbrace ^{\dag}(\mathcal{L}_{\alpha}^{\ep})^{2\dag}f_v^{\ep},\\
					\widehat{M}^\ep:=&\begin{pmatrix}
						\frac{\rho(\ep)- k_1(1-\theta_{\alpha,R})-\tau\la_R}{\left(\rho(\ep)- k_1(1-\theta_{\alpha,R})-\tau\la_R\right)^2+\left(\tau\lambda_I- k_1\theta_{\alpha,I}\right)^2}
						&\frac{ k_1\theta_{\alpha,I}-\tau\lambda_I}{\left(\rho(\ep)-k_1(1-\theta_{\alpha,R})-\tau\la_R\right)^2+\left(\tau\lambda_I- k_1\theta_{\alpha,I}\right)^2}\\
						\frac{\tau\lambda_I- k_1\theta_{\alpha,I}}{\left(\rho(\ep)- k_1(1-\theta_{\alpha,R})-\tau\la_R\right)^2+\left(\tau\lambda_I- k_1\theta_{\alpha,I}\right)^2}
						&\frac{\rho(\ep)- k_1(1-\theta_{\alpha,R})-\tau\la_R}{\left(\rho(\ep)- k_1(1-\theta_{\alpha,R})-\tau\la_R\right)^2+\left(\tau\lambda_I- k_1\theta_{\alpha,I}\right)^2}
					\end{pmatrix}.
				\end{aligned}
			\end{cases}			
		\end{equation*}
		
		Analogous to \eqref{mz^*-1}, we can obtain the following singular limit system of \eqref{mz-2} 
		\begin{equation}\label{mz^*-2}
			\begin{pmatrix}
				\widehat{T}_{\la_R^*} & -\lambda _I^*\\
				\lambda _I^* &\widehat{T}_{\la_R^*}
			\end{pmatrix}
			\begin{pmatrix}
				z_R^*\\
				z_I^*
			\end{pmatrix}
			= c_1^*c_2^*\delta(x - x^*)\widehat{M}^*
			\begin{pmatrix}
				\left\langle z_R^*, \delta(x - x^*) \right\rangle \\
				\left\langle z_I^*, \delta(x - x^*) \right\rangle
			\end{pmatrix},\;\;\text{in $H^{-1}$ sense},
		\end{equation}
		where 
		\begin{equation*}
			\begin{split}
				\widehat{T}_{\la_R^*}:=&-d\frac{d^2}{dx^2}-\frac{f_u^*g_v^*-f_v^*g_u^*}{f_u^*}+\la_R^*,\\
				\widehat{M}^*:=&\begin{pmatrix}
					\frac{\rho_0^*-k_1(1-\theta_{\alpha,R}^*)-\tau\la_R^*}{\left(\rho_0^*-k_1(1-\theta_{\alpha,R}^*)-\tau\la_R^*\right)^2+\left(\tau\lambda_I^*-k_1\theta_{\alpha,I}^*\right)^2}
					&\frac{k_1\theta_{\alpha,I}^*-\tau\lambda_I^*}{\left(\rho_0^*-k_1(1-\theta_{\alpha,R}^*)-\tau\la_R^*\right)^2+\left(\tau\la_I^*-k_1\theta_{\alpha,I}^*\right)^2}\\
					\frac{\tau\lambda_I^*-k_1\theta_{\alpha,I}^*}{\left(\rho_0^*-k_1(1-\theta_{\alpha,R}^*)-\tau\la_R^*\right)^2+\left(\tau\la_I^*-k_1\theta_{\alpha,I}^*\right)^2}
					&\frac{\rho_0^*-k_1(1-\theta_{\alpha,R}^*)-\tau\la_R^*}{\left(\rho_0^*-k_1(1-\theta_{\alpha,R}^*)-\tau\la_R^*\right)^2+\left(\tau\la_I^*-k_1\theta_{\alpha,I}^*\right)^2}
				\end{pmatrix},
			\end{split}
		\end{equation*}
		in which $\rho_0^*$ is defined by \eqref{App1} in Appendix, and
		\begin{equation*}
			\theta_{\alpha,R}^*:=\frac{\alpha(\alpha+\la_R^*)}{(\alpha+\la_R^*)^2+(\la_I^*)^2}, \;\;\theta_{\alpha,I}^*:=-\frac{\alpha \la_I^*}{(\alpha+\la_R^*)^2+(\la_I^*)^2}.
		\end{equation*}
		
		By Lemma 2.4 of \cite{NM}, the generalized inverse operator $(\widehat{T}_{\la_R^*})^{-1}$ exists for $\la_R^*>-\mu^*$, where $\mu^*$ is given by \eqref{mu^*}.
		Similar to the derivation of \eqref{mz^*-1-b} in \S 2, we have
		\begin{equation}\label{mz^*-2-b}
			\begin{pmatrix}
				z_R^*\\
				z_I^*
			\end{pmatrix}
			= c_1^*c_2^*
			\widehat{G}
			\begin{pmatrix}
				I& \lambda _I^*(\widehat{T}_{\la_R^*})^{-1}\\
				-\lambda _I^*(\widehat{T}_{\la_R^*})^{-1}&I
			\end{pmatrix}
			(\widehat{T}_{\la_R^*})^{-1}\delta(x - x^*)
			\widehat{M}^* 
			\begin{pmatrix}
				\left\langle {{z_R^*}, \delta(x - x^*)} \right\rangle \\
				\left\langle {{z_I^*}, \delta(x - x^*)} \right\rangle
			\end{pmatrix},
		\end{equation}
		where $\widehat{G}:=\{I+(\lambda _I^*)^2(\widehat{T}_{\la_R^*})^{-2}\}^{-1}$. 
		
		Taking the inner product with $\delta(x - x^*)$ in \eqref{mz^*-2-b}, we have
		\begin{equation*}
			\mathcal N_\al
			\begin{pmatrix}
				\left\langle z_R^*, \delta(x - x^*) \right\rangle \\
				\left\langle z_I^*, \delta(x - x^*) \right\rangle
			\end{pmatrix}
			=0,
		\end{equation*}
		where
		\begin{equation}\label{mz^*-2-c}
			\mathcal N_\al:=\left\langle c_1^*c_2^*
			\widehat{G}
			\begin{pmatrix}
				I& \lambda _I^*(\widehat{T}_{\la_R^*})^{-1}\\
				-\lambda _I^*(\widehat{T}_{\la_R^*})^{-1}&I
			\end{pmatrix}
			(\widehat{T}_{\la_R^*})^{-1}\delta(x - x^*)
			\widehat{M}^*,\delta(x - x^*) \right\rangle  - I.
		\end{equation}
		
		Similarly, we can check that $\mathcal N_\al$ is equivalent to
		\begin{equation*}
			\begin{pmatrix}
				\frac{\left(\rho_0^*-k_1(1-\theta_{\alpha,R}^*)-\tau\la_R^*\right)\widehat{X}+\la_I^*\left(\tau\lambda_I^*-k_1\theta_{\alpha,I}^*\right)\widehat{Y}}{\left(\rho_0^*-k_1(1-\theta_{\alpha,R}^*)-\tau\la_R^*\right)^2+\left(\tau\lambda_I^*-k_1\theta_{\alpha,I}^*\right)^2}-1  & -\frac{(\tau\lambda_I^*-k_1\theta_{\alpha,I}^*)\widehat{X}-\la_I^*(\rho_0^*-k_1(1-\theta_{\alpha,R}^*)-\tau\la_R^*)\widehat{Y}}{\left(\rho_0^*-k_1(1-\theta_{\alpha,R}^*)-\tau\la_R^*\right)^2+\left(\tau\lambda_I^*-k_1\theta_{\alpha,I}^*\right)^2}\\
				\frac{(\tau\lambda_I^*-k_1\theta_{\alpha,I}^*)\widehat{X}-\la_I^*(\rho_0^*-k_1(1-\theta_{\alpha,R}^*)-\tau\la_R^*)\widehat{Y}}{\left(\rho_0^*-k_1(1-\theta_{\alpha,R}^*)-\tau\la_R^*\right)^2+\left(\tau\lambda_I^*-k_1\theta_{\alpha,I}^*\right)^2} &  \frac{\left(\rho_0^*-k_1(1-\theta_{\alpha,R}^*)-\tau\la_R^*\right)\widehat{X}+\la_I^*\left(\tau\lambda_I^*-k_1\theta_{\alpha,I}^*\right)\widehat{Y}}{\left(\rho_0^*-k_1(1-\theta_{\alpha,R}^*)-\tau\la_R^*\right)^2+\left(\tau\lambda_I^*-k_1\theta_{\alpha,I}^*\right)^2}-1
			\end{pmatrix},
		\end{equation*}
		where $\widehat{X}=\widehat{X}(\la_R^*,(\la_I^*)^2)$ and $\widehat{Y}=\widehat{Y}(\la_R^*,(\la_I^*)^2)$ are respectively given by
		\begin{equation}\label{AB1}
			\begin{split}
				\widehat{X}&:=\left\langle {c_1^*c_2^*\widehat{G}{(\widehat{T}_{\la_R^*})^{-1}}\delta(x - x^*),{\delta(x - x^*)}} \right\rangle=\sum\limits_{n \ge 0} \dfrac{(\gamma _n + \lambda _R^*)c_1^*c_2^*  }{(\gamma _n + \lambda _R^*)^2 + (\lambda _I^*)^2}\psi_n ^2(x^*),\\
				\widehat{Y}&:=\left\langle {c_1^*c_2^*\widehat{G}{(\widehat{T}_{\la_R^*})^{-2}}\delta(x - x^*),{\delta(x - x^*)}} \right\rangle=\sum\limits_{n \ge 0} \dfrac{c_1^*c_2^*  }{(\gamma _n + \lambda _R^*)^2 + (\lambda _I^*)^2}\psi_n ^2(x^*).
			\end{split}    
		\end{equation}
		
		To expect the existence of nonzero solutions of \eqref{mz^*-2-c}, one requires $\det (\mathcal N_\al)=0$, or equivalently
		\begin{equation}\label{SLEP-1}\tag{SLEP-1}
			\begin{cases}
				\begin{split}
					&\frac{\left(\rho_0^*-k_1(1-\theta_{\alpha,R}^*)-\tau\la_R^*\right)\widehat{X}+\la_I^*\left(\tau\lambda_I^*-k_1\theta_{\alpha,I}^*\right)\widehat{Y}}{\left(\rho_0^*-k_1(1-\theta_{\alpha,R}^*)-\tau\la_R^*\right)^2+\left(\tau\lambda_I^*-k_1\theta_{\alpha,I}^*\right)^2}=1,\\
					&\frac{(\tau\lambda_I^*-k_1\theta_{\alpha,I}^*)\widehat{X}-\la_I^*(\rho_0^*-k_1(1-\theta_{\alpha,R}^*)-\tau\la_R^*)\widehat{Y}}{\left(\rho_0^*-k_1(1-\theta_{\alpha,R}^*)-\tau\la_R^*\right)^2+\left(\tau\lambda_I^*-k_1\theta_{\alpha,I}^*\right)^2}= 0,			
				\end{split}
			\end{cases}
		\end{equation}
		which is called the real singular limit eigenvalue problem (SLEP) of \eqref{EP-1}. 
		
		Similar to Proposition \ref{prop.AB}, we have
		
		\begin{proposition}\label{propAB1}
			Let $\widehat{X}(\la_R^*,(\la_I^*)^2)$ and $\widehat{Y}(\la_R^*,(\la_I^*)^2)$ be defined in \eqref{AB1}. Then it holds true that 
			\begin{equation*}
				\begin{split}
					& \widehat{X}(0,0)=\sum\limits_{n \ge 0} \dfrac{c_1^*c_2^*}{\gamma _n }\psi_n ^2(x^*)>\rho_0^*,\;\;\widehat{Y}(0,0)=\sum\limits_{n \ge 0} \dfrac{c_1^*c_2^*}{\gamma _n^2 }\psi_n ^2(x^*)>0,\;\;\ds\frac{\pa \widehat{Y}}{\pa \la_R^*}<0,\\
					&\ds\frac{\pa \widehat{X}}{\pa (\la_I^*)^2}<0,\;\ds\frac{\pa \widehat{Y}}{\pa (\la_I^*)^2}<0,\;\lim\limits_{(\la_I^*)^2\to\infty}\widehat{X}(\la_R^*,(\la_I^*)^2)=\lim\limits_{(\la_I^*)^2\to\infty}\widehat Y(\la_R^*,(\la_I^*)^2)=0.\\
				\end{split}
			\end{equation*}
		\end{proposition}
		
		It turns out that the distributed-delay coupling cannot affect the stability of symmetric solution through \eqref{SLEP-1}.
		
		\begin{lemma}
			Suppose that $\tau\in(\tau^*,+\infty)$ holds. Then, for any $ k_1>0$, $k_2>0$ and $\alpha>0$, all the eigenvalues of \eqref{SLEP-1} lie on the left side of imaginary axis.
		\end{lemma}
		\begin{proof}
			It suffices to show that for any $ k_1>0$,  $k_2>0$ and $\alpha>0$, both zero and pure imaginary numbers are impossible to become eigenvalues of \eqref{SLEP-1}. 
			
			If $\la^*=0$ is a eigenvalue for some $(k_1, k_2,\alpha)$, then \eqref{SLEP-1} is reduced to 
			\begin{equation*}\label{SLEP-1: re}
				\widehat{X}(0,0)-\rho_0^*=0,
			\end{equation*}
			which is a contradiction with Proposition \ref{propAB1}.
			
			If $\la^*=i\la_I^*$ ($\neq0$) is a pure imaginary eigenvalue for some $(k_1, k_2,\alpha)$, then \eqref{SLEP-1} is simplified into
			\begin{equation*}
				\widehat{X}(0,(\la_I^*)^2)=\rho_0^*-\frac{ k_1(\la_I^*)^2}{\alpha^2+(\la_I^*)^2},\;\widehat{Y}(0,(\la_I^*)^2)=\tau+\frac{ k_1\alpha}{\alpha^2+(\la_I^*)^2},
			\end{equation*}
			where the second equation does not hold true owing to
			\begin{equation*}
				\tau+\dfrac{k_1\alpha}{\alpha^2+(\la_I^*)^2}>\tau>\tau^*=\widehat{Y}(0,0)>\widehat{Y}(0,(\la_I^*)^2),
			\end{equation*}
			which completes the proof.
		\end{proof}
		
		\subsection{Analysis of the eigenvalue problem \eqref{EP-2}}
		Repeating the similar procedure in \S \ref{S3.2.1}, we can derive the following real SLEP equations of the eigenvalue problem \eqref{EP-2}
		\begin{equation}\label{SLEP-2}\tag{SLEP-2}
			\begin{cases}
				\begin{split}
					&\frac{\left(\rho_0^*-k_1(1+\theta_{\alpha,R}^*)-\tau\la_R^*\right)X+\la_I^*\left(\tau\lambda_I^*+k_1\theta_{\alpha,I}^*\right)Y}{\left(\rho_0^*-k_1(1+\theta_{\alpha,R}^*)-\tau\la_R^*\right)^2+\left(\tau\lambda_I^*+k_1\theta_{\alpha,I}^*\right)^2}=1,\\
					&\frac{(\tau\lambda_I^*+k_1\theta_{\alpha,I}^*)X-\la_I^*(\rho_0^*-k_1(1+\theta_{\alpha,R}^*)-\tau\la_R^*)Y}{\left(\rho_0^*-k_1(1+\theta_{\alpha,R}^*)-\tau\la_R^*\right)^2+\left(\tau\lambda_I^*+k_1\theta_{\alpha,I}^*\right)^2}= 0,			
				\end{split}
			\end{cases}
		\end{equation}
		where $X=X(\la_R^*,(\la_I^*)^2, k_2)$ and $Y=Y(\la_R^*,(\la_I^*)^2,k_2)$ are given by \eqref{AB}.
		
		To show the stability of $(\widetilde{u}(x;\epsilon),\widetilde{v}(x;\epsilon),\widetilde{u}(x;\epsilon),\widetilde{v}(x;\epsilon))$ with respect to \eqref{EP-2}, we need to investigate the existence of zero or pure imaginary eigenvalues of \eqref{SLEP-2} with respect to $\alpha\in(0,+\infty)$ under some fixed $( k_1, k_2)\in \mathbb{R}_+^2\setminus \Gamma$, where $\Gamma$ is given by Proposition \ref{p*ei}.
		
		Suppose that $\la_R^*=\la_I^*= 0$. Then, we can reduce \eqref{SLEP-2} to
		\begin{equation}\label{SLEP-2: re}
			X(0,0, k_2)-\rho_0^*+2 k_1= c_1^*c_2^*\sum\limits_{n \ge 0} \dfrac{1}{\gamma _n+2 k_2 }\psi_n ^2(x^*)-\rho_0^*+2 k_1=0.
		\end{equation}
		
		Clearly, \eqref{SLEP-2: re} is independent of $\alpha$. In \S 3, we have shown that, if $\alpha=+\infty$, then for any $( k_1, k_2)\in \mathbb{R}_+^2\setminus \Gamma$, the eigenvalue problem has no zero eigenvalues. Thus, we can conclude that the results can be extended to the case of $\al\in(0,+\infty)$ and $( k_1, k_2)\in \mathbb{R}_+^2\setminus \Gamma$.
		
		Suppose that $\la_R^*=0$ and $\la_I^*\neq 0$. Then, \eqref{SLEP-2} can be reduced to
		\begin{subequations}\label{SLEP-2: im}
			\begin{numcases}{}
				X(0,(\la_I^*)^2, k_2)=\rho_0^*- k_1\bigg(1+\dfrac{\alpha^2}{\alpha^2+(\la_I^*)^2}\bigg)=:P((\la_I^*)^2),\label{SLEP-2: im-a}\\
				Y(0,(\la_I^*)^2,k_2)=\tau-\dfrac{k_1\alpha}{\alpha^2+(\la_I^*)^2}=:Q((\la_I^*)^2),\label{SLEP-2: im-b}	
			\end{numcases}
		\end{subequations}
		where (see \eqref{AB})
		\begin{equation*}
			X(0,(\la_I^*)^2, k_2)=\sum\limits_{n \ge 0} \frac{(\gamma _n +2 k_2) c_1^*c_2^*\psi_n ^2(x^*)}{(\gamma _n  +2 k_2)^2 + (\lambda _I^*)^2},\;
			Y(0,(\la_I^*)^2,k_2)=\sum\limits_{n \ge 0} \frac{c_1^*c_2^*\psi_n ^2(x^*)}{(\gamma _n  +2 k_2)^2 + (\lambda _I^*)^2}.			
		\end{equation*}
		
		We have the following results on the properties of $P((\la_I^*)^2)$ and $Q((\la_I^*)^2)$ (see Fig. \ref{Fig. ABPQ}). 
		
		\begin{figure}
			\centering  
			\subfloat{
				\begin{tikzpicture}
					\draw[->,thick] (0,0) --(4,0) node[right] {$(\la_I^*)^2$};
					\draw[->,thick] (0,0) --(0,4);
					\draw  (0,0) node[left] {$O$};
					\draw [thick] (4,0.1) .. controls (2,0.3) and (0.6,0.5).. (0,3.7);
					\draw [thick] (4,2.9) .. controls (2,2.8) and (0.7,2.6)..  (0,0.5);
					\draw [densely dashed](4,3)--(0,3) node [left] {$\rho_0^*-k_1$};
					\draw node at (3.5,0.7) {$X(0,(\la_I^*)^2, k_2)$};
					\draw node at (3,2.4) {$P((\la_I^*)^2)$};
					\draw node at (-0.7,0.6) {$\rho_0^*-2 k_1$};
				\end{tikzpicture}
			}\hspace{10pt}	
			\subfloat{
				\begin{tikzpicture}
					\draw[->,thick] (0,0) --(4,0) node[right] {$(\la_I^*)^2$};
					\draw[->,thick] (0,0) --(0,4);
					\draw  (0,0) node[left] {$O$};
					\draw [thick] (4,0.1) .. controls (2,0.3) and (0.6,0.5).. (0,3);
					\draw [thick] (4,3.5) .. controls (2,3.4) and (0.6,3.3)..  (0,1);
					\draw [densely dashed](4,3.6)--(0,3.6) node [left] {$\tau$};
					\draw node at (3,0.9) {$Y(0,(\la_I^*)^2, k_2)$};
					\draw node at (3,2.9) {$Q((\la_I^*)^2)$};
					\draw node at (-0.7,1) {$\tau-\frac{ k_1}{\alpha}$};
				\end{tikzpicture}
			}	
			\caption{Graph of $X(0,(\la_I^*)^2, k_2)$, $P((\la_I^*)^2)$, $Y(0,(\la_I^*)^2,k_2)$ and $Q((\la_I^*)^2)$.}
			\label{Fig. ABPQ}
		\end{figure}
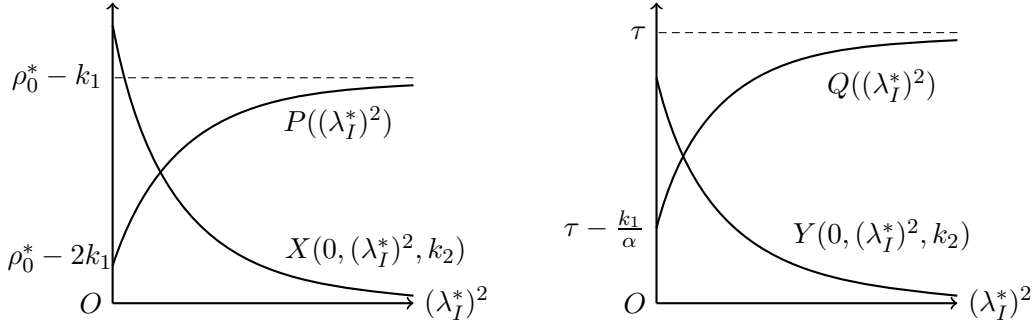
		
		\begin{proposition}\label{prop.PQ}
			Let $P((\la_I^*)^2)$ and $Q((\la_I^*)^2)$ be defined by \eqref{SLEP-2: im}. Then, the following conclusions hold true:
			\begin{equation*}
				P(0)=\rho_0^*-2 k_1,\;\lim\limits_{(\la_I^*)^2\rightarrow+\infty}P((\la_I^*)^2)=\rho_0^*- k_1,\;
				Q(0)=\tau-\frac{ k_1}{\alpha},\;\lim\limits_{(\la_I^*)^2\rightarrow+\infty}Q((\la_I^*)^2)=\tau,
			\end{equation*}
			and
			\begin{equation*}
				\frac{\pa P}{\pa (\la_I^*)^2}=\frac{ k_1\alpha^2}{(\alpha^2+(\la_I^*)^2)^2}>0,\;\;\frac{\pa Q}{\pa (\la_I^*)^2}=\frac{ k_1\alpha}{(\alpha^2+(\la_I^*)^2)^2}>0.
			\end{equation*}
		\end{proposition}
		\begin{proof}
			The proof is trivial and we shall omit it.
		\end{proof}
		
		Then, by Propositions \ref{prop.AB} and \ref{prop.PQ}, we have the following results (see Fig. \ref{Fig. ABPQ}).
		\begin{lemma}\label{le.nonex.} 
			Suppose that $\tau\in(\tau^*,+\infty)$ holds. Define
			\begin{equation}\label{G12312}
				\begin{split}
					&\Gamma_1:=\left\lbrace ( k_1, k_2)|k_1\in(0,\rho_0^*/2),X(0,0, k_2)<\rho_0^*-2 k_1 \right\rbrace,\\
					&\Gamma_2:=\left\lbrace ( k_1, k_2)|k_1\in(0,\rho_0^*/2),X(0,0, k_2)>\rho_0^*-2 k_1 \right\rbrace,\\
					&\Gamma_{3-1}:=(\rho_0^*/2,\rho_0^*)\times(0,\infty),\;\Gamma_{3-2}:=[\rho_0^*,\infty)\times(0,\infty).
				\end{split}
			\end{equation}
			If one of the following conditions holds: 
			\begin{enumerate}[leftmargin=4.5em]
				\item[$(1).$]\;$P(0)=\rho_0^*-2 k_1>X(0,0, k_2)$, and $\al\in(0,+\infty)$ $($see region $\Gamma_{1}$$)$;
				\item[$(2).$]\;$P(\infty)=\rho_0^*- k_1<0$, and $\al\in(0,+\infty)$ $($see region $\Gamma_{3-2}$$)$;
				\item[$(3).$]\;$Q(0)=\tau-\dfrac{ k_1}{\alpha}\geq Y(0,0, k_2)$, or equivalently $\al\in[\al_0,+\infty)$, where \begin{equation}\label{alp0}
					\alpha_0:=\frac{k_1}{\tau-Y(0,0, k_2)}>0,
				\end{equation}
			\end{enumerate}
			then the eigenvalue problem \eqref{SLEP-2} has no purely imaginary eigenvalues.
		\end{lemma}
		
		Then, by Lemma \ref{le.nonex.}, \eqref{G12312} and the continuity of eigenvalues with respect to $\ep$ up to $\ep\to0^+$, one can immediately obtain the following results.
		\begin{theorem}\label{cor.ue} 
			Suppose that $\tau\in(\tau^*,+\infty)$ holds. Then, there exists a $\ep_0>0$, such that for any fixed $\ep\in(0,\ep_0)$,  regardless of $\al\in(0,+\infty)$, $(\widetilde{u}(x;\epsilon),\widetilde{v}(x;\epsilon),\widetilde{u}(x;\epsilon),\widetilde{v}(x;\epsilon))$ is always unstable when $( k_1, k_2)\in \Gamma_1$, while always stable when $( k_1, k_2)\in \Gamma_{3-2}$.
		\end{theorem}
		
		In view of $(3)$ in Lemma \ref{le.nonex.} and Theorem \ref{cor.ue}, to expect zero eigenvalue or purely imaginary eigenvalues, it is sufficient to consider the case when $( k_1, k_2)\in \Gamma_2\cup\Gamma_{3-1}$ and $\alpha\in(0,\alpha_0)$. 
		
		It is clear that when $( k_1, k_2)\in \Gamma_2\cup\Gamma_{3-1}$,  \eqref{SLEP-2: im-a} always has a unique positive root, denoted by $\la_{I,1}^*(\al)$, which is defined in $(0,+\infty)$. On the other hand, when $( k_1, k_2)\in \Gamma_2\cup\Gamma_{3-1}$ and $\alpha\in(0,\alpha_0)$, \eqref{SLEP-2: im-b} has a unique root (non-zero), denoted by $\la_{I,2}^*(\al)$, which is defined in $(0,\al_0)$.

		We have the following lemma on the properties of $\la_{I,1}^*(\alpha)$ in $(0,+\infty)$ (see Fig. \ref{Fig. alpha_2}).
		\begin{lemma} \label{le.la1}
			Suppose that $\tau\in(\tau^*,+\infty)$, $( k_1, k_2)\in\Gamma_2\cup\Gamma_{3-1}$ and $\al\in(0,+\infty)$. Let $\la_{I,1}^*(\alpha)$ be the solution of \eqref{SLEP-2: im-a}. Then, $\la_{I,1}^*(\alpha)$ is increasing with respect to $\alpha\in(0,+\infty)$. Moreover, the following conclusions hold true:
			\begin{enumerate}
				\item Suppose that $k_1\ge 2\rho_0^*/3$ holds. Then, for any $\al\in(0,+\infty)$, $\la_{I,1}^*(\alpha)>\alpha$.
				\item Suppose that  $k_1\in(0,\rho_0^*/2)\cup(\rho_0^*/2,2\rho_0^*/3)$ holds. Then, there exists a $\widehat{k}_2^*>0$ such that if $ k_2\ge \widehat{k}_2^*$, then for any $\alpha\in(0,+\infty)$, $\la_{I,1}^*(\alpha)<\alpha$; while if $k_2<\widehat{k}_2^*$, then there is a unique $\alpha_1\in(0,+\infty)$, such that 
				\begin{equation}\label{la_1^*(alpha)}
					\la_{I,1}^*(\alpha)
					\begin{cases}
						\begin{split}
							>\alpha,~&\text{if}~\alpha\in(0,\alpha_1),\\
							=\alpha,~&\text{if}~\alpha=\alpha_1,\\
							<\alpha,~&\text{if}~\alpha\in(\al_1,+\infty).\\
						\end{split}
					\end{cases}
				\end{equation}
				In particular, $\widehat{k}_2^*<\xi( k_1)$, where $\xi(k_1)$ is given by \eqref{H=0}.	
			\end{enumerate}
		\end{lemma}
		
		\begin{proof}	
			For $( k_1, k_2)\in \Gamma_2\cup\Gamma_{3-1}$ and $\alpha\in(0,+\infty)$, we define
			\begin{equation}\label{conx-a}
				\widehat F_1(\eta,\alpha,k_1, k_2):=X(0,\eta\al^2,k_2)-\rho_0^*+ k_1\bigg(1+\dfrac{1}{1+\eta}\bigg),\;\eta_1^*(\al)=\bigg(\dfrac{\la_{I,1}^*(\al)}{\al}\bigg)^2,
			\end{equation}
			where $\eta>0$.
			
			Thus, by \eqref{SLEP-2: im-a}, we have $\widehat F_1(\eta_1^*(\alpha),\alpha,k_1, k_2)=0$. Differentiating it with respect to $\alpha$, we have 
			\begin{equation*}
				\frac{\pa \la_{I,1}^*(\alpha)}{\pa \alpha}=-\frac{ k_1\alpha\la_{I,1}^*(\alpha)}{(\alpha^*+(\la_{I,1}^*(\alpha))^2)\frac{\pa X(0,(\la_I^*)^2, k_2)}{\pa (\la_I^*)^2} - k_1\alpha}>0,
			\end{equation*} 
			where $\frac{\pa X(0,(\la_I^*)^2, k_2)}{\pa (\la_I^*)^2}<0$ is used. Then, $\la_{I,1}^*(\alpha)$ is increasing with respect to $\alpha\in(0,+\infty)$.

			By Proposition \ref{prop.AB}, we have
			\begin{equation*}
				\dfrac{\pa \widehat F_1}{\pa \eta}=\alpha^2\dfrac{\pa X}{\pa(\la_I^*)^2}(0,\eta\alpha^2,k_2)-\dfrac{ k_1}{(1+\eta)^2}<0,
			\end{equation*}
			which implies that
			\begin{equation}\label{eta1}
				\widehat F_1(\eta,\alpha,k_1, k_2)
				\begin{cases}
					\begin{split}
						>0,~&\mbox{if}~\eta<\eta_1^*(\alpha),\\
						=0,~&\mbox{if}~\eta=\eta_1^*(\alpha),\\
						<0,~&\mbox{if}~\eta>\eta_1^*(\alpha).\\
					\end{split}
				\end{cases}
			\end{equation}
			
			One can check that
			\begin{equation}\label{F_1}
				\widehat F_1(1,\alpha,k_1, k_2)=X(0,\alpha^2,k_2)+\dfrac{3 k_1}{2}-\rho_0^*,
			\end{equation}
			which is decreasing with respect to $\alpha\in(0,+\infty)$ due to Proposition \ref{prop.AB}.
			
			Suppose that $k_1\ge 2\rho_0^*/3$ holds. Since $X(0,\alpha^2,k_2)$ is positive, we have $\widehat F_1(1,\alpha,k_1, k_2)>0$. Then, by \eqref{eta1}, $1<\eta_1^*(\alpha)$ or equivalently $\alpha<\la_{I,1}^*(\alpha)$.
			
			Suppose that $0<k_1<2\rho_0^*/3$ and $ k_1\neq\rho_0^*/2$ hold. For one hand, we have
			\begin{equation*}
				\lim\limits_{\alpha\to+\infty}\widehat F_1(1,\alpha,k_1, k_2)=\dfrac{3 k_1}{2}-\rho_0^*<0.
			\end{equation*}
			
			On the other hand, we have
			\begin{equation}
				\lim\limits_{\alpha\to 0^+}\widehat F_1(1,\alpha,k_1, k_2)=X(0,0, k_2)+\dfrac{3 k_1}{2}-\rho_0^*\;
				\begin{cases}
					\begin{split}
						&>0,\;\;\mbox{if}~ k_2\in(0,\widehat{k}_2^*),\\
						&=0,\;\;\mbox{if}~ k_2=\widehat{k}_2^*,\\
						&<0,\;\;\mbox{if}~ k_2\in(\widehat{k}_2^*,\infty),
					\end{split}
				\end{cases}
			\end{equation}
			where $\widehat{k}_2^*=\widehat{k}_2^*( k_1)$ is the unique solution of
			\begin{equation}\label{2 k_2*}
				X(0,0, \widehat{k}_2^*)=\rho_0^*-\dfrac{3 k_1}{2}.
			\end{equation} 
			The existence and uniqueness of $\widehat{k}_2^*$ is really induced by the monotonicity of $X(0,0, k_2)$ on $ k_2$ (see Proposition \ref{prop.AB}) and the facts that
			\begin{equation*}
				X(0,0,0)>\rho_0^*>\rho_0^*-\dfrac{3 k_1}{2},\;\;X(0,0,\infty)=0<\rho_0^*-\dfrac{3 k_1}{2}.
			\end{equation*}
			Moreover, by comparing \eqref{H=0} with \eqref{2 k_2*}, we conclude that $\widehat{k}_2^*<\xi( k_1)$ for $k_1\in(0,\rho_0^*/2)$.
			
			In summary, we have
			
			(1). In case of $ k_2\in(0,\widehat{k}_2^*)$, we have $\lim\limits_{\alpha\to 0^+}\widehat F_1(1,\alpha,k_1, k_2)>0$, which indicates that there exists a unique solution $\alpha_1=\alpha_1( k_1, k_2)\in(0,\infty)$ such that
			\begin{equation}\label{alpha_1}
				\widehat F_1(1,\alpha,k_1, k_2)\begin{cases}
					\begin{split}
						&>0,\;\;\mbox{if}~\al\in(0,\al_1),\\
						&=0,\;\;\mbox{if}~\al=\al_1,\\
						&<0,\;\;\mbox{if}~\al\in(\al_1,+\infty),
					\end{split}
				\end{cases}
			\end{equation}
			since $\frac{\pa \widehat F_1(1,\alpha,k_1, k_2)}{\pa\alpha}<0$. Then, for any $\al\in(0,\al_1)$, by \eqref{alpha_1} and \eqref{eta1}, we have $1<\eta_1^*(\alpha)$ or equivalently $\alpha<\la_{I,1}^*(\alpha)$. For any $\al\in(\al_1,+\infty)$, we have $1>\eta_1^*(\alpha)$ or equivalently $\alpha>\la_{I,1}^*(\alpha)$. In particular, at $\al=\al_1$, we have $\alpha=\la_{I,1}^*(\alpha)$. 
			
			Moreover, one can check that $\alpha_1( k_1, k_2)$ is decreasing in $k_2$, and
			\begin{equation}\label{a1K20}
				\lim\limits_{ k_2\to (\widehat{k}_2^*)^-}\alpha_1( k_1, k_2)=0.
			\end{equation}
			
			(2). In case of $ k_2\in[\widehat{k}_2^*, \infty)$, we have $\lim\limits_{\alpha\to 0^+}\widehat F_1(1,\alpha,k_1, k_2)\leq0$. Thus, for any $\al\in(0,+\infty)$, $\widehat F_1(1,\alpha,k_1, k_2)<0$. Then, by \eqref{alpha_1} and \eqref{eta1}, we have $1>\eta_1^*(\alpha)$ or equivalently $\alpha>\la_{I,1}^*(\alpha)$. 
		\end{proof}
		
		\begin{figure}
			\centering  	
			\subfloat[]{
				\begin{tikzpicture}[font=\footnotesize]
					\draw[->] (0,0) --(3.2,0) node[right] {$\alpha$};
					\draw[->] (0,0) --(0,3);
					\draw  (0,0) node[left] {$O$};
					\draw[blue,thick] plot[domain=0:3](\x,{\x*(3-\x)/1.5});
					\draw [red,thick] (0,0.5)  parabola (1.6,2.6);
					\draw [densely dashed](0,0)--(3,3);
					\draw [densely dashed](3,0)--(3,3);
					\draw [densely dashed](1.5,1.5)--(1.5,0) node[below]{$\alpha_2$};
					\draw [densely dashed](1.5,1.5)--(0,1.5) node[left]{$\alpha_2$};
					\draw (3,0) node[below]{$\alpha_0$};
					\draw (0.8,2.4) node{$\la_{I,1}^*(\alpha)$};
					\draw (2.2,0.4) node{$\la_{I,2}^*(\alpha)$};
				\end{tikzpicture}
			}\hspace{10pt}	
			\subfloat[]{
				\begin{tikzpicture}[font=\footnotesize]
					\draw[->] (0,0) --(3.2,0) node[right] {$\alpha$};
					\draw[->] (0,0) --(0,3);
					\draw  (0,0) node[left] {$O$};
					\draw[blue,thick] plot[domain=0:3](\x,{\x*(3-\x)/1.5});
					\draw [red,thick] (3,1.6)  parabola (0,0.4);
					\draw [densely dashed](0,0)--(3,3);
					\draw [densely dashed](3,3)--(3,0)node [below]{$\alpha_0$};
					\draw [densely dashed](1.15,1.15)--(1.15,0);
					\draw [densely dashed](1.15,1.15)--(0,1.15) node[left]{$\alpha_1$};
					\draw [densely dashed](1.5,1.5)--(1.5,0) node[below]{$\alpha_2$};
					\draw [densely dashed](1.5,1.5)--(0,1.5) node[left]{$\alpha_2$};
					\draw (2.2,0.4) node{$\la_{I,2}^*(\alpha)$};
					\draw (3,1.6) node [right] {$\la_{I,1}^*(\alpha)$};
					\draw (1.05,0)node[below]{$\alpha_1$};
					\label{4(b)}
				\end{tikzpicture}
			}	\hspace{10pt}	
			\subfloat[]{
				\begin{tikzpicture}[font=\footnotesize]
					\draw[->] (0,0) --(3.2,0) node[right] {$\alpha$};
					\draw[->] (0,0) --(0,3);
					\draw  (0,0) node[left] {$O$};
					\draw[blue,thick] plot[domain=0:3](\x,{\x*(3-\x)/1.5});
					\draw [red,thick] (3,1.3)  parabola (0,0);
					\draw [densely dashed](0,0)--(3,3);
					\draw [densely dashed](3,3)--(3,0);
					\draw (3,0) node [below]{$\alpha_0$};
					\draw [densely dashed](1.5,1.5)--(1.5,0) node[below]{$\alpha_2$};
					\draw [densely dashed](1.5,1.5)--(0,1.5) node[left]{$\alpha_2$};
					\draw (3,0) node[below]{$\alpha_0$};
					\draw (2.2,0.4) node{$\la_{I,2}^*(\alpha)$};
					\draw (3,1.3) node [right] {$\la_{I,1}^*(\alpha)$};
					\label{4c}
				\end{tikzpicture}
			}			
			\caption{The rough figures of $\la_{I,1}^*(\alpha)$ (red) and $\la_{I,2}^*(\alpha)$ (blue). (a). $k_1\ge 2\rho_0^*/3$; (b). $k_1\in(0,\rho_0^*/2)\cup(\rho_0^*/2,2\rho_0^*/3)$ and $k_2<\widehat{k}_2^*$; (c). $k_1\in(0,\rho_0^*/2)$, $ k_2\in[\widehat{k}_2^*,\xi( k_1))$, or $ k_1\in(\rho_0^*/2,2\rho_0^*/3)$, $ k_2\in[ \widehat{k}_2^*,\infty)$.}
			\label{Fig. alpha_2}
		\end{figure}
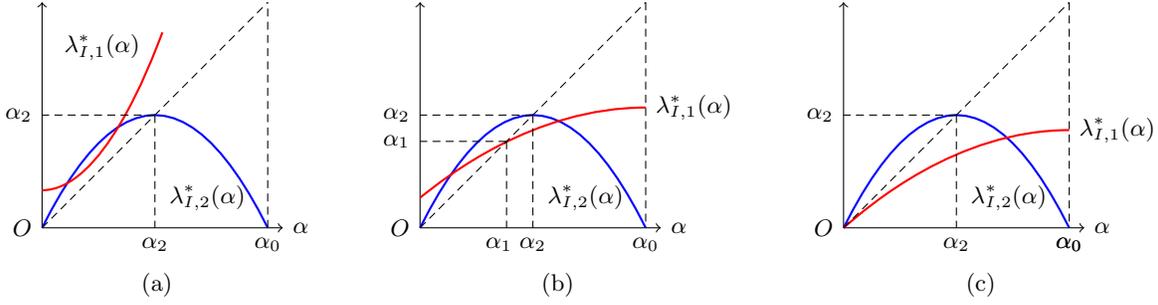
		
		We have the following lemma on the properties of $\la_{I,2}^*(\alpha)$ (see Fig. \ref{Fig. alpha_2}).
		\begin{lemma}\label{tcg-2} 
			Suppose that $\tau\in(\tau^*,+\infty)$, $( k_1, k_2)\in \Gamma_2\cup\Gamma_{3-1}$, and $\al\in(0,\al_0)$, where $\al_0$ is defined in \eqref{alp0}. Let $\la_{I,2}^*(\alpha)$ be the solution of \eqref{SLEP-2: im-b}. Then,
			\begin{equation}\label{la0a0}
				\lim\limits_{\alpha\to0^+}\la_{I,2}^*(\alpha)=\lim\limits_{\alpha\to\alpha_0^-}\la_{I,2}^*(\alpha)=0,
			\end{equation}
			and there exists a unique $\alpha_2\in(0,\alpha_0)$, such that $\la_{I,2}^*(\alpha)$ is increasing in $(0,\al_2)$, while decreasing in $(\al_2,\al_0)$, and at $\al=\al_2$, $\la_{I,2}^*(\alpha)$ attains its local maximal value. Moreover,
			\begin{equation}\label{henan-1}
				\la_{I,2}^*(\alpha)
				\begin{cases}
					\begin{split}
						>\alpha,~&\mbox{if}~\alpha\in(0,\alpha_2),\\
						=\alpha,~&\mbox{if}~\alpha=\alpha_2,\\
						<\alpha,~&\mbox{if}~\alpha\in(\alpha_2,\alpha_0).
					\end{split}
				\end{cases}
			\end{equation}
		\end{lemma}
		\begin{proof}
			1. By \eqref{alp0}, one can directly check that $\lim\limits_{\alpha\to\alpha_0^-}\la_{I,2}^*(\alpha)=0$. 
			
			We then show that $\la_{I,2}^*(\alpha)\to 0$ as $\al\to0^+$. Suppose not. Then, $\la_{I,2}^*(\alpha)\to c$ for some $c\in(0,\infty]$. 
			By \eqref{SLEP-2: im-b}, we have 
			\begin{equation}\label{SLEP-3: im-b}
				F_2((\la_{I,2}^*(\al))^2,\alpha,k_1, k_2):= Y(0,(\la_{I,2}^*(\al))^2,k_2)-\tau+\dfrac{k_1\alpha}{\alpha^2+(\la_{I,2}^*(\al))^2}=0.	
			\end{equation}
			
			Then, from \eqref{SLEP-3: im-b} and $\la_{I,2}^*(\alpha)\to c$, we have $Y(0,c,k_2)-\tau=0$. However, this is impossible since $\tau>\tau^*=Y(0,0,0)>Y(0,c,k_2)$.

			2. For $\al\in(0,\al_0)$, we define
			\begin{equation}\label{eta2(al)}
				\eta_2^*(\al)=\bigg(\dfrac{\la_{I,2}^*(\al)}{\al}\bigg)^2.
			\end{equation}
			
			For $( k_1, k_2)\in \Gamma_2\cup\Gamma_{3-1}$ and $\alpha\in(0,\alpha_0)$, we define
			\begin{equation}\label{sinx-a}
				\widehat F_2(\eta,\alpha,k_1, k_2):=Y(0,\eta\al^2,k_2)-\tau+\dfrac{k_1}{\alpha+\eta\al},
			\end{equation}
			where $\eta>0$.
			
			Then, by \eqref{SLEP-3: im-b}, we have
			\begin{equation}\label{sinx-b}
				\widehat F_2(\eta_2^*(\al),\alpha,k_1, k_2)=0.	
			\end{equation}
			
			By Proposition \ref{prop.AB}, we have
			\begin{equation}\label{wuyi-1}
				\dfrac{\pa \widehat F_2}{\pa \eta}=\alpha^2\dfrac{\pa Y}{\pa(\la_I^*)^2}(0,\eta\alpha^2,k_2)-\dfrac{ k_1}{(1+\eta)^2\alpha}<0.
			\end{equation}
			
			Then, by \eqref{sinx-b} and \eqref{wuyi-1}, we have
			\begin{equation}\label{the0}
				\widehat F_2(\eta,\alpha,k_1, k_2)
				\begin{cases}
					\begin{split}
						>0,~&\mbox{if}~\eta<\eta_{2}^*(\al),\\
						=0,~&\mbox{if}~\eta=\eta_{2}^*(\al),\\
						<0,~&\mbox{if}~\eta>\eta_{2}^*(\al).
					\end{split}
				\end{cases}
			\end{equation}
			
			On the other hand, we can check that
			\begin{equation}\label{F_2}
				\widehat F_2(1,\alpha,k_1, k_2)=Y(0,\alpha^2,k_2)+\dfrac{k_1}{2\alpha}-\tau.
			\end{equation}
			
			By Proposition \ref{prop.AB}, $\widehat F_2(1,\alpha,k_1, k_2)$ is decreasing in $\alpha$. Then, by
			\begin{equation*}
				\begin{split}
					\lim\limits_{\alpha\to0^+}\widehat F_2(1,\alpha,k_1, k_2)=&+\infty,\\
					\lim\limits_{\alpha\to\alpha_0^-}\widehat F_2(1,\alpha,k_1, k_2)=&Y(0,\alpha_0^2, k_2)-\frac{Y(0,0, k_2)}{2}-\tau<\frac{Y(0,0, k_2)-\tau}{2}<0,
				\end{split}
			\end{equation*}
			there exists a unique $\alpha_2=\alpha_2( k_1, k_2)\in(0,\alpha_0)$ satisfying
			\begin{equation}\label{wuyi-3}
				\widehat F_2(1,\alpha,k_1, k_2)
				\begin{cases}
					\begin{split}
						>0,~&\mbox{if}~\alpha<\alpha_2,\\
						=0,~&\mbox{if}~\alpha=\alpha_2,\\
						<0,~&\mbox{if}~\alpha>\alpha_2.
					\end{split}
				\end{cases}
			\end{equation}
			
			Therefore, combining \eqref{wuyi-3} with \eqref{the0} and \eqref{eta2(al)}, we can obtain \eqref{henan-1}.
			
			Moreover, by differentiating $\widehat F_2(1,\alpha_2,k_1, k_2)=0$ with respect to $ k_2$ (resp., $k_1$), we can conclude that $\alpha_2=\alpha_2( k_1, k_2)$ is decreasing (resp., increasing) in $ k_2$ (resp., $ k_1$). In particular, it follows from \eqref{F_2} that 
			\begin{equation}\label{alpha_2-r}
				\lim\limits_{k_2\to+\infty}\alpha_2( k_1, k_2)=\dfrac{k_1}{2\tau}.
			\end{equation}
			
			Finally, we show the monotone properties of $\la_{I,2}^*(\alpha)$ with respect to $\alpha$. Differentiating \eqref{SLEP-3: im-b} with respect to $\alpha$ yields
			\begin{equation*}
				\dfrac{d\la_{I,2}^*(\alpha)}{d\alpha}=\dfrac{ k_1(\alpha+\la_{I,2}^*(\alpha))(\alpha-\la_{I,2}^*(\alpha))}{2\la_{I,2}^*(\alpha)\left[\alpha^2+(\la_{I,2}^*(\alpha))^2\right]^2 \frac{\pa Y(0,(\la_I^*)^2,k_2)}{\pa (\la_I^*)^2}}
				\begin{cases}
					\begin{split}
						>0,~&\mbox{if}~\alpha<\alpha_2,\\
						=0,~&\mbox{if}~\alpha=\alpha_2,\\
						<0,~&\mbox{if}~\alpha>\alpha_2,\\
					\end{split}
				\end{cases}
			\end{equation*}
			where $\frac{\pa Y(0,(\la_I^*)^2,k_2)}{\pa (\la_I^*)^2}<0$ is used.
		\end{proof}

		We have the following two lemmas.
		
		\begin{lemma}\label{le.K2.sma}
			Suppose that $\tau\in(\tau^*,+\infty)$ and 
			\begin{equation}\label{H_1}
				k_1\in(0,\rho_0^*/2)\cup(\rho_0^*/2,2\rho_0^*/3),\;\;k_2\in(k_2^{*}-\kappa_1,\widehat{k}_2^*),
			\end{equation}
			where $\kappa_1>0$ is sufficiently small. Then, there exist $2m$ $($counting multiplicity, $m\geq1$ is an integer$)$ points $\al_{j}^*\in(0,\al_0)$, $1\leq j\leq 2m$, satisfying
			\begin{equation*}
				0<\al_1^*\leq\al_2^*\leq\cdots\leq\al_{2m-1}^*<\al_2<\al_{2m}^*<\al_0,
			\end{equation*}
			such that at $\al=\al_{j}^*$, $1\leq j\leq 2m$, $\la_{I,1}^*(\alpha)=\la_{I,2}^*(\alpha)$. In particular, $0<\la_{I,1}^*(\al_{2m}^*)=\la_{I,2}^*(\al_{2m}^*)<\al_{2m}^*$.
		\end{lemma}
		\begin{proof}			
		    Since $\alpha_2( k_1, k_2)$ is decreasing in $k_2$, by \eqref{alpha_2-r} and \eqref{a1K20}, we have
			\begin{equation*}
				\alpha_2(k_1, \widehat{k}_2^*)>\lim\limits_{ k_2\to+\infty}\alpha_2( k_1, k_2)=\frac{k_1}{2\tau}>0=\alpha_1(k_1, \widehat{k}_2^*),
			\end{equation*}
			which indicates that $\alpha_2=\alpha_2( k_1, k_2)>\alpha_1=\alpha_1( k_1, k_2)$ for $k_2\in(\widehat{k}_2^*-\kappa_1,\widehat{k}_2^*)$ with some sufficiently small $\kappa_1>0$.
			
			Under \eqref{H_1}, we can obtain the existence of $2m$ $($counting multiplicity, $m\geq1$ is an integer$)$ points $\al_{j}^*\in(0,\al_0)$, $1\leq j\leq 2m$, such that at $\al=\al_{j}^*$, $1\leq j\leq 2m$, $\la_{I,1}^*(\alpha)=\la_{I,2}^*(\alpha)$. One can see Fig. \ref{4(b)}. Moreover, $\la_{I,1}^*(\alpha)-\la_{I,2}^*(\alpha)$ is increasing in $\alpha\in(\alpha_2,\alpha_0)$ and satisfies
			\begin{equation*}
				\la_{I,1}^*(\alpha_2)-\la_{I,2}^*(\alpha_2)<0,\;\;	\la_{I,1}^*(\alpha_0)-\la_{I,2}^*(\alpha_0)>0.
			\end{equation*}
			
			Then, there exists a unique $\al$ such that $\la_{I,1}^*(\alpha)-\la_{I,2}^*(\alpha)=0$ in $(\alpha_2,\alpha_0)$. This proves that $\al_{2m}^*\in(\alpha_2,\alpha_0)$ but $\al_{2m-1}^*\notin(\alpha_2,\alpha_0)$.
			In particular, $\la_{I,1}^*(\al_{2m}^*)<\al_{2m}^*$.
		\end{proof}
		
		\begin{lemma}\label{le.ex.K2.lar}
			Suppose that $\tau\in(\tau^*,+\infty)$ and that one of the following conditions holds true:
			\begin{align}
				&k_1\in\left(0,\rho_0^*/2\right),\; k_2\in[\widehat{k}_2^*,\xi( k_1)),\label{3.88a}\\
				&k_1\in\left(\rho_0^*/2,2\rho_0^*/3\right),\;  k_2\in[\widehat{k}_2^*,\infty),\label{3.88b}\\
				&k_1\in\left[2\rho_0^*/3,\rho_0^*\right),\;  k_2\;\text{is sufficiently large}.\label{3.88c}
			\end{align}
			Then, there exists a unique $\alpha^*\in(0,\al_0)$, such that at $\al=\al^*$,  the eigenvalue problem \eqref{EP-2} has a pair of purely imaginary eigenvalues $\pm i \la^*$. Moreover, 
			\begin{enumerate}
				\item if either \eqref{3.88a} or \eqref{3.88b} holds, then $\al_2<\la^*<\alpha^*<\al_0$. 
				\item if either \eqref{3.88b} or \eqref{3.88c} holds, then in the limit of $ k_2\rightarrow+\infty$, we have
				\begin{equation}\label{tends2}
					\alpha^*=\frac{\rho_0^*- k_1}{\tau},\;\;\la^*=\sqrt{\frac{(2 k_1-\rho_0^*)( \rho_0^*- k_1)}{ \tau^{2}}}. 
				\end{equation}
			\end{enumerate}
		\end{lemma}
		
		\begin{proof}
			Suppose that \eqref{3.88a} holds. Then, by Lemma \ref{tcg-2}, for any $\al\in(0, \al_2)$, we have $\la_{I,2}^*(\alpha)>\al$. By Lemma \ref{le.la1}, for any $\al\in(0, \al_2)$, we have $\la_{I,1}^*(\alpha)<\al$. Thus, in $(0,\al_2)$, $\la_{I,1}^*(\alpha)<\al<\la_{I,2}^*(\alpha)$. By Lemma \ref{le.la1}, $\la_{I,1}^*(\alpha)$ is increasing in $(\al_2,\al_0)$ and $\lim\limits_{\al\rightarrow\al_0^-}\la_{I,1}^*(\alpha)>0$. By Lemma \ref{tcg-2}, $\la_{I,2}^*(\alpha)$ is decreasing in $(\al_2,\al_0)$ and $\lim\limits_{\al\rightarrow\al_0^-}\la_{I,2}^*(\alpha)=0$. Moreover, $\lim\limits_{\al\rightarrow\al_2^+}\la_{I,1}^*(\alpha)<\lim\limits_{\al\rightarrow\al_2^+}\la_{I,2}^*(\alpha)$. Then, in $(\al_2,\al_0)$, there exists a unique $\al^*$, such that at $\al=\al^*$, we have $\la_{I,1}^*(\al^*)=\la_{I,2}^*(\al^*)$. 
			
			Similarly, we can show the existence and uniqueness of $\al^*$ in the case of \eqref{3.88b}. And it is clear to show that under either \eqref{3.88a} or \eqref{3.88b}, it holds true that $\al_2<\la_{I,1}^*(\al^*)=\la_{I,2}^*(\al^*)<\alpha^*<\al_0$. 
			
			Suppose that either \eqref{3.88b} or \eqref{3.88c} holds.  Then, 
			\begin{equation*}
				\lim\limits_{ k_2\to\infty}\la_{I,1}^*(\alpha)=\alpha\sqrt{\dfrac{2 k_1-\rho_0^*}{\rho_0^*-k_1}},\;\lim\limits_{ k_2\to\infty}\la_{I,2}^*(\alpha)=\sqrt{\dfrac{k_1\alpha}{\tau}-\alpha^2}.
			\end{equation*}
			
			Solving $\lim\limits_{ k_2\to\infty}\la_{I,1}^*(\alpha)=\lim\limits_{ k_2\to\infty}\la_{I,2}^*(\alpha)$, we can obtain \eqref{tends2}. This completes the proof.
		\end{proof}
		
		Define
		\begin{equation}\label{lzxx-1}
			\alpha_H^*:=
			\begin{cases}
				\al_{2m}^*,\;&\text{in case of Lemma \ref{le.K2.sma}},\\
				\al^*,\;&\text{in case of Lemma \ref{le.ex.K2.lar}},
			\end{cases}\;\la_{I,H}^*=\la_{I,1}^*(\alpha_H^*)=\la_{I,2}^*(\alpha_H^*)>0.
		\end{equation}
		
		We have the following results.
		\begin{lemma}\label{comei}
			Let $\al_H^*$ and $\la_H^*$ be defined in \eqref{lzxx-1}. Then, there exists a sufficiently small $\kappa_2>0$, such that for any $\al\in(\al_H^*-\kappa_2, \al_H^*+\kappa_2)$, the eigenvalue problem \eqref{EP-2} has the simple complex conjugate eigenvalues $\la^*(\alpha):=\la_R^*(\alpha)\pm i\la_I^*(\alpha)$ such that $\la_R^*(\alpha_H^*)=0$, $\la_I^*(\alpha_H^*)=\la_{I,H}^*$. Moreover, if in addition $d>d^*$ for some $d^*\geq d_0$, where $d_0$ is defined in Theorem \ref{th.1-0}, then $\frac{d}{d\al}\la_R^*(\alpha_H^*)<0$.
		\end{lemma}
		
		\begin{proof}
			Similar to \eqref{comFP-3}, we can obtain the following complex SLEP system of \eqref{EP-2}
			\begin{equation}\label{SLEP-4}
				\mathcal{G}^*(\la^*,\alpha):=\rho_0^*-\tau\la^*- k_1\left(\frac{\alpha}{\alpha+\la^*}+1\right)-\sum_{n=0}^{\infty}\dfrac{c_1^*c_2^*}{\gamma_n+\la^*+2 k_2}\psi_n^2(x^*)=0,
			\end{equation}
			where $\la^*=\la_R^*+i\la_I^*$.
			
			We only consider the existence of the desired eigenvalue $\la_R^*(\alpha)+ i\la_I^*(\alpha)$, while the existence of the eigenvalue $\la_R^*(\alpha)- i\la_I^*(\alpha)$ can be proved similarly.
			
			Differentiating \eqref{SLEP-4}  with respect to $\la^*$ at $(i\la_{I,H}^*,\alpha_H^*)$ yields
			\begin{equation}\label{F*/la*}
				\dfrac{\pa \mathcal{G}^*}{\pa \la^*}( i\la_{I,H}^*,\alpha_H^*)=-\tau+\dfrac{k_1\alpha_H^*}{(\alpha_H^*+i\la_{I,H}^*)^2}+\sum_{n=0}^{\infty}\dfrac{c_1^*c_2^*}{(\gamma_n+i\la_{I,H}^*+2 k_2)^2}\psi_n^2(x^*).
			\end{equation}
			
			One can directly check that 
			\begin{equation}\label{IM}
				\text{Im}\bigg\{\dfrac{\pa \mathcal{G}^*}{\pa \la^*}(i\la_{I,H}^*,\alpha_H^*)\bigg\}=-\dfrac{2 k_1(\alpha_H^*)^2\la_{I,H}^*}{\big((\alpha_H^*)^2+(\la_{I,H}^*)^2\big)^2}-\sum_{n=0}^{\infty}\dfrac{2c_1^*c_2^*\la_{I,H}^*(\gamma_n+2 k_2)}{\big((\gamma_n+2 k_2)^2+(\la_{I,H}^*)^2\big)^2}\psi_n^2(x^*)< 0.
			\end{equation}
			
			Then, $\frac{\pa \mathcal{G}^*}{\pa \la^*}( i\la_{I,H}^*,\alpha_H^*)\neq0$. Therefore, by the implicit function theorem, there exists a sufficiently small $\kappa_2>0$, such that for any $\al\in(\al_H^*-\kappa_2, \al_H^*+\kappa_2)$, the eigenvalue problem \eqref{EP-2} has the complex eigenvalue $\la^*(\alpha)=\la_R^*(\alpha)+ i\la_I^*(\alpha)$ such that $\la_R^*(\alpha_H^*)=0$, $\la_I^*(\alpha_H^*)=\la_{I,H}^*$.
			
			Next, we show that the transversality condition holds. That is, $\dfrac{d}{d\al}\la_R^*(\alpha_H^*)<0$.
			
			Substituting $\la^*(\alpha)$ into \eqref{SLEP-4}, differentiating \eqref{SLEP-4} with respect to $\alpha$, and setting $(\la^*(\al),\al)=(i\la_H^*,\alpha_H^*)$, we can obtain
			\begin{equation}\label{F^* pa alp}
				\frac{\pa \mathcal{G}^*}{\pa \la^*}(i\la_H^*,\alpha_H^*)\dfrac{d \la^*(\alpha_H^*)}{d \alpha}+\frac{\pa \mathcal{G}^*}{\pa \alpha}(i\la_H^*,\alpha_H^*)=0,
			\end{equation}
			where $\frac{\pa \mathcal{G}^*}{\pa \la^*}(i\la_H^*,\alpha_H^*)$ is given by \eqref{F*/la*}, and $\frac{\pa \mathcal{G}^*}{\pa \alpha}(i\la_H^*,\alpha_H^*)=-\frac{ k_1 i \la_H^*}{(i \la_H^*+\alpha_H^*)^2}$.
			
			Separating the real and imaginary parts of \eqref{F^* pa alp}, we have
			\begin{equation}\label{dla/dalp}
				\frac{d}{d\alpha}\la_R^*(\alpha_H^*)=-\frac{\text{Re}\big\{\frac{\pa \mathcal{G}^*}{\pa \la^*}\big\}\text{Re}\big\{\frac{\pa \mathcal{G}^*}{\pa \al}\big\}
					+\text{Im}\big\{\frac{\pa \mathcal{G}^*}{\pa \la^*}\big\}\text{Im}\big\{\frac{\pa \mathcal{G}^*}{\pa \al}\big\}}
				{\big( \text{Re}\big\{\frac{\pa \mathcal{G}^*}{\pa \la^*}\big\}\big) ^2+\big( \text{Im}\big\{\frac{\pa \mathcal{G}^*}{\pa \la^*}\big\}\big)^2},
			\end{equation}
			where $\frac{\pa \mathcal{G}^*}{\pa \la^*}$ and $\frac{\pa \mathcal{G}^*}{\pa \al}$ are evaluated at $(i\la_{I,H}^*,\alpha_H^*)$. In particular, $\text{Im}\big\{\frac{\pa \mathcal{G}^*}{\pa \la^*}\big\}$ is given by \eqref{IM}, and
			\begin{equation*}
				\begin{split}
					\text{Re}\bigg\{\frac{\pa \mathcal{G}^*}{\pa \la^*}\bigg\}=&-\tau+\frac{\alpha_H^*   k_1 \big((\alpha_H^*)^2-(\la_{I,H}^*)^2\big)}{\big((\alpha_H^*)^2+(\la_{I,H}^*)^2\big)^2}+c_1^*c_2^*\sum_{n=0}^{\infty}\frac{(\gamma_n+2 k_2)^2-(\la_{I,H}^*)^2}{\big((\gamma_n+2 k_2)^2+(\la_{I,H}^*)^2\big)^2}\psi_n^2(x^*),\\
					\text{Re}\bigg\{\frac{\pa \mathcal{G}^*}{\pa \al}\bigg\}=&\frac{-2 \alpha_H^*  k_1 (\la_{I,H}^*)^2 }{\big((\alpha_H^*)^2+(\la_{I,H}^*)^2\big)^2},\;
					\text{Im}\bigg\{\frac{\pa \mathcal{G}^*}{\pa \al}\bigg\}=\frac{\la_{I,H}^*  k_1 \big((\la_{I,H}^*)^2-(\alpha_H^*)^2 \big)}{\big((\alpha_H^*)^2+(\la_{I,H}^*)^2\big)^2}.
				\end{split}
			\end{equation*}
			
			One can check that 
			\begin{equation}\label{dla/dalp-2}
				\frac{d}{d\alpha}\la_R^*(\alpha_H^*)=-\frac{I_1+I_2}
				{\big( \text{Re}\big\{\frac{\pa \mathcal{G}^*}{\pa \la^*}(  i\la_{I,H}^*,\alpha_H^*)\big\}\big) ^2+\big( \text{Im}\big\{\frac{\pa \mathcal{G}^*}{\pa \la^*}(  i\la_{I,H}^*,\alpha_H^*)\big\}\big)^2},
			\end{equation}
			where
			\begin{equation}\label{I_1I_2}
				\begin{split}
					I_1&:=\bigg( \tau-c_1^*c_2^*\sum_{n=0}^{\infty}\frac{(\gamma_n+2 k_2)^2-(\la_{I,H}^*)^2}{((\gamma_n+2 k_2)^2+(\la_{I,H}^*)^2)^2}\psi_n^2(x^*)\bigg) \dfrac{ 2\alpha_H^*  k_1 (\la_{I,H}^*)^2 }{((\alpha_H^*)^2+(\la_{I,H}^*)^2)^2},\\
					I_2&:=\dfrac{2 k_1(\la_{I,H}^*)^2 ((\alpha_H^*)^2-(\la_{I,H}^*)^2)}{((\alpha_H^*)^2+(\la_{I,H}^*)^2)^2}\sum_{n=0}^{\infty}\dfrac{c_1^*c_2^*(\gamma_n+2 k_2)}{((\gamma_n+2 k_2)^2+(\la_{I,H}^*)^2)^2}\psi_n^2(x^*).
				\end{split}	
			\end{equation}
			
			Since $\la_{I,H}^*<\alpha_H^*$, we have $I_2>0$. It remains to show that $I_1>0$.
			
			Case 1: For any integer $n\ge 0$, if $ k_2=+\infty$, then $I_1$ is reduced to
			\begin{equation}\label{I_3}
				I_1=\dfrac{2\tau\alpha_H^*  k_1 (\la_{I,H}^*)^2 }{((\alpha_H^*)^2+(\la_{I,H}^*)^2)^2}>0.
			\end{equation}
			
			Case 2: For any integer $n\ge 0$ and $ k_2\in(0,+\infty)$.
			
			It is clear from \eqref{alp0}, Lemmas \ref{le.K2.sma} and \ref{le.ex.K2.lar} that
			\begin{equation}\label{K2laM}
				\la_{I,H}^*<\al_{H}^*<\alpha_0=\al_0( k_2):=\dfrac{ k_1}{\tau-Y(0,0, k_2)}.
			\end{equation}
			
			Define 
			\begin{equation*}
				h( k_2):=\dfrac{\al_0( k_2)}{\sqrt{3}}-2 k_2.
			\end{equation*}
			
			Then, one can check that $h'( k_2)<0$ since $\al_0'( k_2)<0$. Note that 
			\begin{equation*}
				\lim\limits_{ k_2\rightarrow0^+}h( k_2)=\dfrac{ k_1}{\sqrt{3}(\tau-Y(0,0,0))}=\dfrac{ k_1}{\sqrt{3}(\tau-\tau^*)}>0.
			\end{equation*}
			
			Recall that $\gamma_0>0$ is the principle eigenvalue of the operator $-d\frac{d^2}{dx^2}-(f_u^*g_v^*-f_v^*g_u^*)/f_u^*$ subject to the homogeneous Neumann boundary condition in the $L^2(0,\ell)$ sense. In particular, $\gamma_0$ is increasing in $d$.  On the other hand, one can check that $\tau^*$ is decreasing in $d$ (see \eqref{tau*}), and so is $\lim\limits_{ k_2\rightarrow0^+}h( k_2)$. Then, there exists a $d_1>0$, such that when $d>d_1$, $\gamma_0>\lim\limits_{ k_2\rightarrow0^+}h( k_2)$. 
			
			Choose $d^*=\max\{d_0,d_1\}$. Then, for any $d>d^*$, we have $\gamma_0>h( k_2)$. Thus, for any $n\geq0$, by $\al_0=\al_0( k_2)>\la_{I,H}^*$, we have
			\begin{equation*}
				\gamma_n\geq\gamma_0>h( k_2)=\dfrac{\al_0( k_2)}{\sqrt{3}}-2 k_2>\dfrac{\la_{I,H}^*}{\sqrt{3}}-2 k_2,
			\end{equation*}
			which implies that $\la_{I,H}^*<\sqrt{3}(\gamma_n+2 k_2)$ or equivalently $(\la_{I,H}^*)^2\in(0,3(\gamma_n+2 k_2)^2)$.
			
			One can check that, if $(\la_I^*)^2\in(0,3(\gamma_n+2 k_2)^2)$, then
			\begin{equation*}
				\dfrac{d}{d(\la_I^*)^2}\bigg\{\frac{(\gamma_n+2 k_2)^2-(\la_I^*)^2}{\left((\gamma_n+2 k_2)^2+(\la_I^*)^2\right)^2}\bigg\}=\frac{\left((\la_I^*)^2-3(\gamma_n+2 k_2)^2\right)\left((\la_I^*)^2+(\gamma_n+2 k_2)^2\right)}{\left((\gamma_n+2 k_2)^2+(\la_I^*)^2\right)^4}<0,
			\end{equation*}
			which yields
			\begin{equation*}
				\frac{(\gamma_n+2 k_2)^2-(\la_I^*)^2}{\left((\gamma_n+2 k_2)^2+(\la_I^*)^2\right)^2}\le \frac{(\gamma_n+2 k_2)^2-(\la_I^*)^2}{\left((\gamma_n+2 k_2)^2+(\la_I^*)^2\right)^2}\bigg|_{(\la_I^*)^2=0}=\frac{1}{(\gamma_n+2 k_2)^2}.
			\end{equation*}
			
			Then, by $\la_{I,H}^*<\sqrt{3}(\gamma_n+2 k_2)$, we can obtain
			\begin{equation*}
				c_1^*c_2^*\sum_{n=0}^{\infty}\frac{(\gamma_n+2 k_2)^2-(\la_{I,H}^*)^2}{\left((\gamma_n+2 k_2)^2+(\la_{I,H}^*)^2\right)^2}\psi_n^2(x^*)\le c_1^*c_2^*\sum_{n=0}^{\infty}\frac{1}{(\gamma_n+2 k_2)^2}\psi_n^2(x^*)=Y(0,0, k_2),
			\end{equation*}
			which leads to
			\begin{equation*}
				\tau-c_1^*c_2^*\sum_{n=0}^{\infty}\frac{(\gamma_n+2 k_2)^2-(\la_{I,H}^*)^2}{\left((\gamma_n+2 k_2)^2+(\la_{I,H}^*)^2\right)^2}\psi_n^2(x^*)\ge\tau-Y(0,0, k_2)>\tau-\tau^*> 0.
			\end{equation*}
			
			Then, $I_1>0$. 
			
			So far, we have proved that under either $ k_2=+\infty$ or $ k_2\in(0,+\infty)$, $I_1>0$ so that $\frac{d}{d\alpha}\la_R^*(\alpha_H^*)<0$. This completes the proof.
		\end{proof}

		As the consequence of Lemma \ref{comei}, we have the following results (in the sense of $\ep\to0^+$).
		\begin{lemma}
			Suppose that all the conditions of Lemma \ref{comei} hold. Then, the following conditions hold true:
			\begin{enumerate}
				\item When  $\alpha>\alpha_H^*$, all the eigenvalues of the limiting eigenvalue problem \eqref{SLEP-2} have negative real parts. 
				\item At $\alpha=\alpha_H^*$, the limiting eigenvalue problem \eqref{SLEP-2} has a pair of simple purely imaginary eigenvalues $\pm i\la_H^*$, with all the remaining eigenvalues having strictly negative real parts.
				\item When $\alpha\in(\alpha_H^*-\kappa_3,\alpha_H^*)$ for some sufficiently small $\kappa_3>0$, the eigenvalue problem \eqref{SLEP-2} has one eigenvalue with strictly positive real part. In particular, if \eqref{3.88b} holds, then for any $\al\in(0,\al_H^*)$, the limiting eigenvalue problem \eqref{SLEP-2} has one eigenvalue with strictly positive real part.
			\end{enumerate}
		\end{lemma}
		
		Finally, we shall state the main results on Hopf bifurcation of \eqref{C-LE-D} around the symmetric solution $(\widetilde{u}(x;\ep),\widetilde{v}(x;\ep),\widetilde{u}(x;\ep),\widetilde{v}(x;\ep))$ for any fixed $\ep\in(0,\ep_0)$, see Fig. \ref{Fig.s-dep}.
		
		\begin{figure}[!t]
			\centering
			\begin{tikzpicture}
				\fill [fill=green!20]  (0,4) --  (0,1) parabola (2.4,4)--cycle;
				\fill [fill=blue!20]  (0,0) --  (0,1) parabola (2.4,4)--(2.5,4)--(2.5,0)--cycle;
				\fill [fill=red!20]  (2.5,0) rectangle (5,4);
				\fill [fill=magenta!20]  (5,0) rectangle (7.5,4);
				
				\fill [pattern=dots]  (0,1) parabola (2.4,4)--(3.2,4) parabola bend (0,1) (0,1);
				
				\draw[->] (0,0) --(7.55,0) node[right] {$k_1$};
				\draw[->] (0,0) --(0,4.05) node[above] {$k_2$};
				\draw (2.5,4)--(2.5,0) node [below] {$\frac{\rho_0^*}{2}$};
				\draw (5,4)--(5,0) node [below] {$\rho_0^*$};
				\draw (0,1) parabola (2.4,4);
				\draw[densely dashed] (0,1) parabola (3.2,4);
				\draw[densely dashed] (3.33,4) -- (3.33,0) node[below]{$\frac{2\rho_0^*}{3}$};
				\draw [densely dashed](2,3.07)--(2,0) node [below]{$k_1$};
				\draw [densely dashed](2,2.17)--(0,2.17) node [left] {$\widehat{k}_2^*$};
				\draw [densely dashed](2,3.07)--(0,3.07) node [left] {$\xi(k_1)$};
				\draw  (0,0) node[left] {$O$};
				\draw node at (1.4,0.8){$\Gamma_2$};
				\draw node at (1,2.6){$\Gamma_1$};
				\draw node at (4.2,2){$\Gamma_{3-1}$};
				\draw node at (6.25,2){$\Gamma_{3-2}$};			
			\end{tikzpicture}
			\caption{Stability diagram of the symmetric solution $(\widetilde{u}(x;\ep),\widetilde{v}(x;\ep),\widetilde{u}(x;\ep),\widetilde{v}(x;\ep))$ with respect to \eqref{C-LE-D} at any fixed $\ep\in(0,\ep_0)$, where $\Gamma_{1}$, $\Gamma_{2}$, $\Gamma_{3-1}$ and $\Gamma_{3-2}$ are explicitly given by \eqref{G12312}. The symmetric solution is always unstable in $\Gamma_{1}$, while stable in $\Gamma_{3-2}$ regardless of $\alpha>0$. In the dotted subregion of $\Gamma_{2}\cup\Gamma_{3-1}$, i.e.  $(k_1,k_2)\in((0,\rho_0^*/2)\times(\widehat{k}_2^*,\xi(k_1)))\cup((\rho_0^*/2,2\rho_0^*/3)\times(\widehat{k}_2^*,\infty))$, system \eqref{C-LE-D} undergoes Hopf bifurcation near the symmetric solution at $\alpha=\alpha_H^\ep$, which can also be triggered if $k_2$ is slightly less than $\widehat{k}_2^*$; or $k_2$ is sufficiently large with $k_1\in [2\rho_0^*/3,\rho_0^*)$.}
			\label{Fig.s-dep}
		\end{figure}
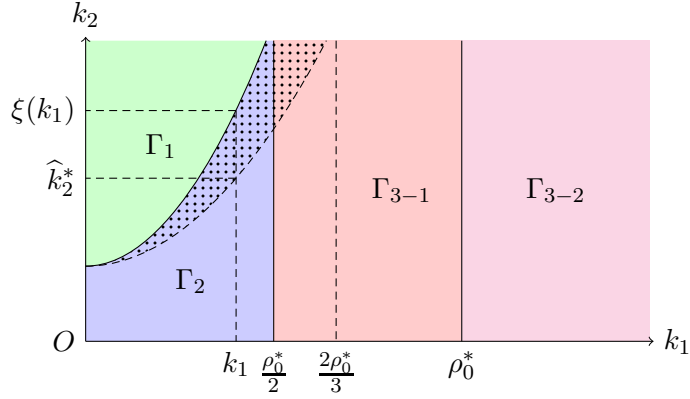
		
		\begin{theorem}\label{th.3-1} 
			Suppose that $\tau\in(\tau^*,+\infty)$ and $d>d^*$, where $d^*$ is defined in Lemma \ref{comei}. Let $(k_1,k_2)$ be fixed such that one of \eqref{H_1}-\eqref{3.88c} is satisfied. Then  for any fixed $\ep\in(0,\ep_0)$, there exists a critical value $\alpha_H^\ep>0$, such that the solution $(\widetilde{u}(x;\ep),\widetilde{v}(x;\ep),\widetilde{u}(x;\ep),\widetilde{v}(x;\ep))$ is stable with respect to system \eqref{C-LE-D} when $\alpha>\alpha_H^{\ep}$, while unstable when $\al$ is slightly less than $\alpha_H^{\ep}$. In particular, system \eqref{C-LE-D} undergoes a Hopf bifurcation at $\alpha=\alpha_H^{\ep}$ for $\ep\in(0,\ep_0)$. That is to say, a family of periodic solutions of \eqref{C-LE-D} occurs around $(\widetilde{u}(x;\ep),\widetilde{v}(x;\ep),\widetilde{u}(x;\ep),\widetilde{v}(x;\ep))$.
		\end{theorem}
		
		\begin{proof}
			Similar to \eqref{FPep}, for \eqref{EP-2}, we can derive the following complex characteristic equation
			\begin{equation}\label{malG}
				\mathcal{G}\left(\la,\alpha,\ep\right):= \rho(\ep)-k_1\left(1+\frac{\alpha}{\la+\alpha}\right)-\tau\la-\left\langle (T_{\la,\alpha}^{\ep})^{-1}\frac{g_u^{\ep}\phi_0^{\ep}}{\sqrt{\ep}},\frac{-f_v^{\ep}\phi_0^{\ep}}{\sqrt{\ep}} \right\rangle=0,
			\end{equation}
			where $(T_{\la,\alpha}^{\ep})^{-1}$, the inverse operator of 
			\begin{equation*}
				T_{\la,\alpha}^{\ep}:=-d\frac{d^2}{dx^2}-g_v^{\ep}+2k_2+g_u^{\ep}\left(\ep^2\dfrac{d^2}{dx^2}+f_u^{\ep}-\ep k_1\left(1+\frac{\alpha}{\la+\alpha}\right)-\ep\tau\la\right)^{\dag}f_v^{\ep}+\la,
			\end{equation*}
			exists for $\text{Re}\la+2k_2>-\mu^*$ by Lemma 4.1 of \cite{NM}, where $\mu^*$ could be chosen smaller if necessary.
			
			One can observe that $\mathcal{G} (\la,\alpha,\ep)\to \mathcal{G}^* (\la,\alpha)$ as $\ep\to0^+$, where $\mathcal{G}^* (\la,\alpha)$ is given by \eqref{SLEP-4}, then it follows that
			\begin{equation*}
				\begin{split}
					\mathcal{G} (\pm i \la_{I,H}^*,\alpha_H^*,0)=\mathcal{G}^*(\pm i \la_{I,H}^*,\alpha_H^*)=0,\;\frac{\pa \mathcal{G}}{\pa \la}(\pm i \la_{I,H}^*,\alpha_H^*,0)=\frac{\pa \mathcal{G}^*}{\pa \la}(\pm i \la_{I,H}^*,\alpha_H^*)\neq 0,
				\end{split}		
			\end{equation*}
			where Lemma \ref{comei} is used. 
			
			By the implicit function theorem, for $(\alpha,\ep)$ in a small neighborhood $\mathcal{O}$ of $(\alpha_H^*,0)$, there exists a unique $\la^\ep(\alpha):=\la_R(\alpha,\ep)+i\la_I(\alpha,\ep)$, such that $\mathcal{G}(\la^\ep(\alpha),\alpha,\ep)=0$ and in particular,
			\begin{equation*}
				\lim\limits_{\ep\to0^+}\la^\ep(\alpha)=\la^*(\alpha),\; \lim\limits_{\ep\to0^+}\la^\ep(\alpha_H^*)=\la^*(\alpha_H^*)=\pm i\la_{I,H}^*.
			\end{equation*}
			
			Since
			\begin{equation*}
				\la_R(\alpha_H^*,0)=\la_R^*(\alpha_H^*)=0,\;\la_I(\alpha_H^*,0)=\la_{I,H}^*\neq0,\; \frac{d\la_R}{d\alpha}(\alpha_H^*,0)=\frac{d\la_R^*}{d\alpha}(\alpha_H^*)< 0,
			\end{equation*}
			by the implicit function theorem, for small $\ep>0$, there exists a unique continuous differentiable function $\al_H^\ep=\al_H(\ep)$ such that 
			\begin{equation*}
				\la_R(\alpha_H^\ep,\ep)=0,\;\la_{I,H}^\ep:=\la_I(\alpha_H^\ep,\ep)\neq0,\; \frac{d\la_R}{d\alpha}(\alpha_H^\ep,\ep)< 0.
			\end{equation*}
			
			Thus, at $\al=\al_H^\ep$, the eigenvalue problem \eqref{malG} has a pair of simple purely imaginary eigenvalues and the transversality condition holds. Hence, by the Hopf bifurcation theorem \cite{CR,HKW}, the system undergoes Hopf bifurcation at $\al=\al_H^\ep$ for small $\ep>0$.
			
			Therefore, when $\alpha>\alpha_H^\ep$, all the eigenvalues of \eqref{D-L-2-a} have negative real parts. On the contrary, if $\alpha\in(\alpha_H^\ep-\kappa_4,\alpha_H^\ep)$ for sufficiently small $\kappa_4>0$, the eigenvalue problem \eqref{D-L-2-a} has a pair of complex eigenvalues with positive real part.
			
			Let $(0,\widehat{\ep}_0)$ be the new definition region of $\ep$ and drop the hat of $\widehat{\ep}_0$. Then the simplicity of the pair of critical complex eigenvalues of \eqref{EP-2} for $\ep\in(0,\ep_0)$ can be obtained by the similar argument in \cite{NM}, which completes the proof.
		\end{proof}
		
		\section{Concluding remarks}
		
		In this paper, we mainly consider the dynamics of a kind of two-layer coupled Lengyel-Epstein model characterizing the spatiotemporal pattern formations in the CIMA chemical reactions. We make the following concluding remarks:
		\begin{enumerate}
			\item In the classical activator-inhibitor system (decoupled), if the diffusion rate $\epsilon$ of the activator $u$ is small but the diffusion rate $d$ of the inhibitor $v$ is large, then we can expect the occurrence of Turing patterns (see for instance \cite{JNT, JSWY}). In this paper, for the two-layered coupled system, we can obtain the same phenomena, that is, if the inter-reactor diffusion rate $k_1$ of the activator $u$ is not that large but the inter-reactor diffusion rate $k_2$ of the inhibitor $v$ is large enough, then we can also find Turing patterns (see Theorem \ref{th.2-1}). 
			\item A slightly large $k_1$ tends to offset the impacts of $k_2$ and $\alpha$ in the sense that once $k_1$ is larger than $\rho_0^*$ (see Theorem \ref{cor.ue}), the symmetric steady state solution is always stable regardless of $k_2$ and $\alpha$. 
			\item It is well-known that delay can induce the instability of the steady state solutions and trigger Hopf bifurcations. However, in the existing literatures, most of the works focus on Hopf bifurcations branching from the constant steady state solutions. It is rare to consider the Hopf bifurcation from non-constant steady state solutions as delay varies. For our first attempt, we study in this paper how the distributed-delayed coupling can affect the stability of the non-constant steady state solution and bring out the existence of Hopf bifurcation. Another interesting question is what happens if we replace distributed-delayed coupling with the discrete-delayed coupling. This is still an open question, which will be our next attempt.
			\item In the present paper, we mainly concentrate on the stability/Turing instability of the symmetric steady state solution $(\widetilde{u}(x;\ep),\widetilde{v}(x;\epsilon), \widetilde{u}(x;\ep),\widetilde{v}(x;\epsilon))$ affected by $k_1,k_2$ and $\alpha$. However, we are more than willing to know how $k_1$ and $k_2$ affect the stability of the symmetric periodic solution $(\widetilde{u}_p(t,x;\ep),\widetilde{v}_p(t, x;\epsilon), \widetilde{u}_p(t,x;\ep),\widetilde{v}_p(t, x;\epsilon))$. We would like to mention that in \cite{Yi}, the Turing instability of the spatially homogeneous periodic solutions is considered. We conjecture that the perturbation technique used in \cite{Yi} seems to work well for our model when studying the Turing instability of the symmetric periodic solution.
			\item In system with only one kind of diffusion, it is meaningless to define Turing instability of the spatially non-homogeneous solutions (steady state solutions and periodic solutions). However, when system has more than one kind of diffusion, Turing instability of the above solutions can be well-defined.
		\end{enumerate}
		
		\appendix	
		\section{Appendix: The spectral properties of $\ep^2\dfrac{d^2}{dx^2}+f_u^\ep$ for small $\epsilon$}
		\renewcommand{\thelemma}{A\arabic{lemma}}		
		The following results are due to Nishiura and Fujii \cite{NF} (see also Lemmas 2.2 and 2.3 of \cite{NM}):
		\begin{lemma}[Spectral properties (\cite{NF, NM})]\label{lemNF1} 
			Assume that $0<\ep<\ep_0$, so that $(\widetilde{u}(x;\epsilon),\widetilde{v}(x;\epsilon))$ is the singularly perturbed steady state solution of system \eqref{2.1}. Then, the eigenvalue problem
			\begin{equation*}
				\begin{cases}
					\begin{split}
						\ep^2\dfrac{d^2\phi}{dx^2}+f_u^\ep\phi&=\mu\phi,\;x\in(0,\ell),\\
						\phi'(0)=\phi'(\ell)&=0,
					\end{split}
				\end{cases}
			\end{equation*}
			has complete and orthonormal pairs of eigenvalues and eigenfunctions $\{(\mu_n^{\ep}, \phi_n^{\ep}(x))\}_{n=0}^{\infty}$ in $L^2(0,\ell)$ sense. Moreover,
			\begin{enumerate}
				\item There exists a $\mu_*<0$ independent of $\ep$, such that $\mu_0^{\ep}>0>\mu_*>\mu_1^{\ep}>\cdots$.
				In particular, $\mu_0^{\ep}=\rho(\ep)\ep$, where $\rho(\ep)$ is a positive continuous function up to $\ep\to 0^+$ and
				\begin{equation}\label{App1}		
					\rho_0^{*}=\lim_{\ep\to 0^+}\rho(\ep)=\dfrac{(\kappa^*)^2}{d}M'(\widehat{v} )\int_0^{x^*}g(U^*,V^*)dx,
				\end{equation}
				where $x^*\in (0,\ell)$ is the layer position of the reduced solution $(U^*(x), V^*(x))$, $\widehat{v}$ is the zero of $M(\widehat{v})$, and $\kappa^*$ is a positive constant.
				\item Let $F(u, v)$ be a smooth function of $u$ and $v$. Then, as $\ep\to 0^+$, 
				\begin{equation*}
					(\ep^2\dfrac{d^2}{dx^2}+f_u^\ep- \ep\tau\la)^{\dag}(F(\widetilde{u}(x; \ep), \widetilde{v}(x;\ep))h) \to \ds\frac{F(U^*(x), V^*(x))h}{f_u(U^*(x), V^*(x))} \;\text{in strongly $L^2$-sense}
				\end{equation*}
				for any function $h\in L^2(0,\ell) \cap L^{\infty}(0,\ell)$, $\tau\in \mathbb{R}$, and $\la\in\mathbb{C}$, where for any $\widehat h\in L^2(0,\ell) \cap L^{\infty}(0,\ell)$,
				\begin{equation*}
					(\ep^2\dfrac{d^2}{dx^2}+f_u^\ep- \ep\tau\la)^{\dag}\widehat h:=\sum_{n\geq1}\dfrac{\langle \widehat h, \phi_n^\ep\rangle}{\mu_0^\ep-\ep\tau\la}\phi_n^\ep.
				\end{equation*}
				\item Let $\delta(x-x^*)$ be the Dirac's delta function at $x = x^*$ and $\kappa^*$ be stated in \eqref{App1}. Then, in the $H^{-1}$ sense, it holds true that
				\begin{equation*}
					\begin{split}
						\lim\limits_{\ep\to 0^+}\ds\frac{f_v^{\ep}\phi_0^{\ep}}{\sqrt{\ep}}&=\kappa^*M'(\widehat{v})\delta(x - x^*),\\
						\lim\limits_{\ep\to 0^+}\ds\frac{g_u^{\ep}\phi_0^{\ep}}{\sqrt{\ep}}&=\kappa^*\bigg[h_+(\widehat{v})-h_-(\widehat{v})+\bigg(\dfrac{h_-(\widehat{v})}{1+(h_-(\widehat{v}))^2}-\dfrac{h_+(\widehat{v})}{1+(h_+(\widehat{v}))^2}\bigg)\widehat{v}\bigg]\delta(x - x^*),
					\end{split}
				\end{equation*}
				where $M'(\widehat{v})<0$, $h_+(\widehat{v})-h_-(\widehat{v})+\big(\frac{h_-(\widehat{v})}{1+(h_-(\widehat{v}))^2}-\frac{h_+(\widehat{v})}{1+(h_+(\widehat{v}))^2}\big)\widehat{v}>0$.
			\end{enumerate}
		\end{lemma}

\begin{center}
	\vskip1cm
	
	{\small
		\bibliographystyle{plain}
		}
\end{center}

\end{document}